\documentclass[1 [leqno,11pt]{amsart}
\usepackage{amssymb, amsmath,latexsym,amsfonts,amsbsy,bbm, amsthm,mathtools,graphicx,CJKutf8,CJKnumb,CJKulem,color,xcolor,mathrsfs}

\usepackage{float}
\usepackage{hyperref}
\usepackage{braket}
\numberwithin{equation}{section}

\allowdisplaybreaks

\let\al=\alpha

\let\la=\lambda

\let\f=\frac

\let\La=\Lambda

\let\pa=\partial

\def\cK{{\mathcal K}_k}

\def\cQ{{\mathcal Q}}

\def\cT{{\mathcal T}}

\def\R{\mathbb R}
\def\Z{\mathbb Z}
\def\T{\mathbb T}

\def\no{\noindent}

\def\dive{\mathop{\rm div}\nolimits}
\def\curl{\mathop{\rm curl}\nolimits}

\def\eqdef{\buildrel\hbox{\footnotesize def}\over =}

\def\oo{\infty}
\def\Ltr{{L^2_r}}
\def\Lor{{L^\oo_r}}
\def\Lpr{{L^p_r}}
\def\Llr{{L^1_r}}
\def\Htr{{H^2_r}}
\def\Hlro{{H^1_{r,0}}}
\def\LtT{{L^2_T}}
\def\LoT{{L^\oo_T}}

\def\LpT{{L^p_T}}

\def\LtO{{L^2(\Omega)}}
\def\LpO{{L^p(\Omega)}}

\def\hkz{{H^1_k}}
\def\hkf{{H^{-1}_k}}
\def\tdelta{{\tilde{\delta}}}
\def\wki{w_{k,I}}
\def\wkh{w_{k,H}}
\def\pki{\varphi_{k,I}}
\def\pkh{\varphi_{k,H}}

\def\hc{\hat{C}}
\def\lap{\Delta}

\newcommand{\beq}{\begin{equation}}
\newcommand{\eeq}{\end{equation}}
\newcommand{\ben}{\begin{eqnarray}}
\newcommand{\een}{\end{eqnarray}}
\newcommand{\beno}{\begin{eqnarray*}}
\newcommand{\eeno}{\end{eqnarray*}}
\newcommand{\andf}{\quad\hbox{and}\quad}
\def\eqdefa{\buildrel\hbox{\footnotesize def}\over =}
\newcommand{\with}{\quad\hbox{with}\quad}

\newtheorem{theorem}{Theorem}[section]

\newtheorem{lemma}[theorem]{Lemma}
\newtheorem{proposition}[theorem]{Proposition}

\begin{document}

\title[Nonlinear  stability of 2D Taylor-Couette flow]
{Nonlinear asymptotic stability of 2D Taylor-Couette flow in the exterior disk}

\author[T. Li]{Te Li}%
\address[T. Li]
 {State Key Laboratory of Mathematical Sciences, Academy of Mathematics $\&$ Systems Science, The Chinese Academy of
	Sciences, Beijing 100190, CHINA.}
\email{teli@amss.ac.cn}

\author[P. Zhang]{Ping Zhang}%
\address[P. Zhang]
 {State Key Laboratory of Mathematical Sciences, Academy of Mathematics $\&$ Systems Science, The Chinese Academy of
	Sciences, Beijing 100190, China, and School of Mathematical Sciences, University of Chinese Academy of Sciences, Beijing 100049, CHINA. }
\email{zp@amss.ac.cn}

\author[Y. Zhang]{Yibin Zhang}%
\address[Y. Zhang]
 {Academy of Mathematics $\&$ Systems Science, The Chinese Academy of
Sciences, Beijing 100190, CHINA.  } \email{zhangyibin22@mails.ucas.ac.cn}

\date{\today}

\begin{abstract}
In this paper, we consider the asymptotic stability of the 2D Taylor-Couette flow in the exterior disk, with a small kinematic viscosity  \( \nu \ll 1 \) and a large rotation coefficient \( |B| \). Due to the degeneracy of the Taylor-Couette flow at infinity, we cannot expect the solution to decay exponentially in a space-time decoupled manner. As stated in previous work \cite{LZZ-25}, even space-time coupled exponential decay can not be expected, and at most, we can obtain space-time coupled polynomial decay. To handle the space-time coupled decay multiplier, the previous time-independent resolvent estimate methods no longer work. Therefore, this paper introduces time-dependent resolvent estimates to deal with the space-time coupled decay multiplier \( \Lambda_k \). We remark that the choice of \( \Lambda_k \) is not unique, here we just provide one way to construct it. Finally, as an application, we derive a transition threshold bound of $\f12$, which is the same as that for the Taylor-Couette flow in the bounded region.

\end{abstract}
\maketitle

\tableofcontents

\section{Introduction}
We consider 2D incompressible Navier-Stokes equations in the exterior disk $\Omega\eqdefa \bigl\{x\in \R^2: \ |x|\geq 1 \bigr\}$:
\begin{align}
\label{full nonlinear equation}\left\{
\begin{aligned}
&\partial_tV-\nu\Delta V+V\cdot \nabla V+\nabla P=0,\quad (t,x)\in \R^+\times\Omega,\\
&\text{div } V=0,\\
&V(0,x)=V_0(x),
\end{aligned}
\right.
\end{align}
where $V\in\mathbb{R}^2$ denotes the fluid velocity and $P$ the scalar pressure function, which guarantees
 the divergence free condition of the velocity field, $\nu$ designates the kinematic viscosity of the fluid.

In this paper, we study the stability of a two-dimensional rotating flow in the exterior region $\Omega$, which is a steady solution to \eqref{full nonlinear equation}. The corresponding velocity, vorticity, and scalar pressure are given as follows:
\begin{subequations} \label{S1eqq}
\begin{gather}
\label{2D TC}
V^\star(r,\theta)=\left(
  \begin{array}{cc}
   -\sin\theta \\
   \cos\theta
    \end{array}
   \right)\Big(Ar+\f{B}{r}\Big){\eqdef\left(
  \begin{array}{cc}
   v^{\star,1} \\
   v^{\star,2}
    \end{array}
   \right)}, \quad W^\star(r) {\eqdefa} \pa_1 v^{\star,2}-\pa_2 v^{\star,1} = 2A,   \\
  P^\star(r,\theta)=P^\star(r) \quad \text{with}\quad  \partial_rP^\star(r)=\f{1}{r}\Big(Ar+\f{B}{r}\Big)^2,
\end{gather}
\end{subequations}
where $A,B$ are two constants and we assume that $\nu \ll |B|$ throughout this paper. The steady flow of the form \eqref{2D TC} is well-known as Taylor-Couette flow.

As a classic example of rotating flow, Taylor-Couette (TC) flow has been extensively studied in both the physical and mathematical contexts.  In physical experiments, TC flow always refers to the fluid flow between two concentric cylinders, where the inner cylinder rotates while the outer cylinder can be either stationary or rotating. This flow configuration was first studied by G. I. Taylor in the 1920s \cite{Tay}. Despite its simple structure, TC flow exhibits complex and rich stability properties, and has become a classic problem in fluid dynamics due to its rich flow patterns, as referenced in   such review books as \cite{AH, CI, MW-06, S-98, GLS-16}. These corresponding flow patterns are controlled by the Reynolds number and the radius ratio of the cylinders. For example, the TC apparatus is also an excellent prototype for studying the transport properties of most astrophysical or geophysical rotating shear flows \cite{DDDLRZ}. Depending on the rotational speed of each cylinder, different flow states can be achieved, corresponding to increasing or decreasing angular velocity and angular momentum. Likewise, regarding the issue of asymptotic stability at high Reynolds numbers \cite{BST}, TC flow, as a typical rotating shear flow, also exhibits rich decay properties.

This paper primarily investigates the fully nonlinear long-time asymptotic stability behavior near the TC flow in the exterior disk. Earlier, the well-posedness and stability of the two-dimensional stationary Navier-Stokes equations in exterior regions \cite{BM,GHM} were widely studied and discussed. The following results are related to the fully nonlinear stability of the time-dependent Navier-Stokes equations near the 2D TC flow.

\begin{itemize}
    \item Under the condition $|B|\ll1$ and with the domain being the exterior disk, \cite{M1,M2} used perturbation theory to study the $L^p-L^q$ estimate for the semigroup generated by the linearized operator, and further addressed the global well-posedness and nonlinear asymptotic stability for small initial data in the $L^2$ space.

    \item In \cite{AHL-2}, nonlinear enhanced dissipation and transition thresholds near the TC flow were obtained in the bounded annular region \((r, \theta) \in [1, R] \times \mathbb{T}\). Since the domain is compact, the shear effect of the TC flow is non-degenerate, i.e., \((A + \frac{B}{r^2})'\) is non-vanishing.
\end{itemize}

In the case when $\nu$ is sufficiently small and  $|B|$ is large enough, that is, when the Reynolds number is high and the rotational effects are strong, the fully nonlinear long-time asymptotic stability problem of the 2D TC flow in the exterior disk is both interesting and challenging.

\subsection{Derivation of the perturbation equation}
Let $(V^\star, P^\star)$ be determined by \eqref{S1eqq}. We denote \( U {\eqdefa} V - V^\star \), \( p {\eqdefa} P - P^\star \),  and consider the following perturbation equations of
the system \eqref{full nonlinear equation} near the 2D TC flow:
\begin{align}
    \label{velocity pertubation of the Taylor-Couette flow}
\left\{
\begin{aligned}
&\partial_t U-\nu\lap U + V^\star \cdot
\nabla U+ U\cdot \nabla V^\star + U\cdot \nabla U + \nabla p=0,  \\
&\dive U=0,\quad (t,x)\in \R^+\times \Omega,\\
& U(t=0)=U(0).
\end{aligned}
\right.
\end{align}
We shall present a local well-posedness result concerning the system \eqref{velocity pertubation of the Taylor-Couette flow}
in Appendix \ref{Appc}.

Since this paper  aims at investigating the impact of degeneracy, which is caused by the external region,
 on the nonlinear stability of TC flow, we impose the impermeable and Navier-slip boundary conditions
 for the velocity field:
\begin{equation}\label{impermeable & navier slip boundary condition}
    U\cdot \mathbf{n}|_{\pa\Omega}=0, \quad \bigl(\pa_1 U^2-\pa_2 U^1\bigr)|_{\pa\Omega}=0.
\end{equation}
Let \( W \eqdefa \curl U {\eqdefa} \partial_1 U^2 - \partial_2 U^1 \) be the vorticity of \( U \), and we rewrite (\ref{velocity pertubation of the Taylor-Couette flow}-\ref{impermeable & navier slip boundary condition}) in the equivalent  vorticity-stream  formulation:
\begin{align}\label{vorticity pertubation of the Taylor-Couette flow}
\left\{
\begin{aligned}
&\partial_t W-\nu\lap W + V^\star \cdot
\nabla W+  U\cdot \nabla W =0, \qquad (t,x)\in \R^+\times \Omega,\\
&\lap \varPhi= W, \quad  U =\nabla^\perp \varPhi{\eqdefa}\begin{pmatrix}
    -\pa_2\varPhi, \pa_1\varPhi
\end{pmatrix}^T, \\
&  W|_{\pa\Omega} = \varPhi|_{\pa \Omega}=0, \quad W(t=0)=W(0).
\end{aligned}
\right.
\end{align}

Let $\mathbf{e_r}= (\cos\theta, \sin\theta)$ and $\mathbf{e_\theta}= (-\sin\theta,\cos\theta).$
In polar coordinates, \eqref{vorticity pertubation of the Taylor-Couette flow} is equivalent to
\begin{align}\label{1.8}
\left\{
\begin{aligned}
&\partial_t W-\nu\Big(\partial_r^2+\f{1}{r}\partial_r+\f{1}{r^2}\partial_{\theta}^2\Big)W+\Big(A+\f{B}{r^2}\Big)\partial_{\theta}W+\f{1}{r}(\partial_r\varPhi\partial_{\theta}W-\partial_{\theta}\varPhi\partial_rW)=0,\\
&\Big(\partial_r^2+\f{1}{r}\partial_r+\f{1}{r^2}\partial_{\theta}^2\Big)\varPhi=W, \quad
(t,r,\theta)\in\R^+\times[1,\oo)  \times\mathbb{T},\\
&  U=U^r \mathbf{e_r} + U^\theta \mathbf{e_\theta} \with U^r = -\frac{\pa_\theta\varPhi}{r} \andf U^{\theta} = \pa_r \varPhi, \\
& W(t=0)=W(0).
\end{aligned}
\right.
\end{align}
By expanding the functions $(W,\varPhi,U^r,U^\theta)$ in Fourier series, we write
\begin{equation}\label{Fourier decomposition}
    \begin{split}
&W(t,r,\theta)=\sum\limits_{k\in\mathbb{Z}}e^{ik\theta}\widehat{W}_k(t,r), \quad \varPhi(t,r,\theta)=\sum\limits_{k\in\mathbb{Z}}e^{ik\theta}\widehat{\varPhi}_k(t,r), \quad  \\
&U^r(t,r,\theta)=\sum\limits_{k\in\mathbb{Z}}e^{ik\theta}\widehat{U}^r_k(t,r), \quad U^\theta(t,r,\theta)=\sum\limits_{k\in\mathbb{Z}}e^{ik\theta}\widehat{U}^\theta_k(t,r).
    \end{split}
\end{equation}
Then by substituting \eqref{Fourier decomposition} into \eqref{1.8} and comparing the Fourier coefficients of the resulting equations, we obtain
\begin{align}\label{pertubation of NS vor-fourier}
\left\{
\begin{aligned}
&\partial_t\widehat{W}_k-\nu\Big(\partial_{r}^2+\f{1}{r}\partial_{r}-\f{k^2}{r^2}\Big)\widehat{W}_k+\Big(A+\f{B}{r^2}\Big)ik\widehat{W}_k\\
&\qquad +\f{i}{r}\sum_{\ell \in\mathbb{Z}}\big ((k-\ell)\partial_r\widehat{\varPhi}_\ell\widehat{W}_{k-\ell}-\ell \widehat{\varPhi}_{\ell}\partial_{r}\widehat{W}_{k-\ell}\big) =0, \quad (t,r,\theta)\in\R^+\times[1,\oo) \times\mathbb{T}\\
&\Big(\partial_r^2+\f{1}{r}\partial_r-\f{k^2}{r^2}\Big)\widehat{\varPhi}_k=\widehat{W}_k, \quad \widehat{U}^r_k = - \frac{ik}{r}\widehat{\varPhi}_k, \quad \widehat{U}^\theta_k = \pa_r\widehat{\varPhi}_k,
\\
&
\widehat{W}_k(t=0)=\widehat{W}_k(0).
\end{aligned}
\right.
\end{align}

To get rid of the terms involving \(\frac{1}{r}\partial_r\) and \(A\) in \eqref{pertubation of NS vor-fourier}, we  introduce
\begin{align}\label{definition of w_k}
    & w_k{\eqdefa}r^{\f12}e^{ikAt}\widehat{W}_k, \quad \varphi_k{\eqdefa}r^{\f12} e^{ikAt}\widehat{\varPhi}_k, \quad   u^r_k{\eqdefa}r^{\f12}e^{ikAt}\widehat{U}^r_k, \quad u^\theta_k{\eqdefa}r^{\f12}e^{ikAt}\widehat{U}^\theta_k.
\end{align}
It is easy to observe that
    \begin{align*}
&\|W\|_\LtO^2 \approx \sum_{k\in \Z} \|\widehat{W}_k\|_{L^2((1,\oo);rdr)}^2 =  \sum_{k\in \Z} \|w_k\|_{L^2((1,\oo);dr)}^2,\\
&\|\pa_rW\|_\LtO^2 \approx \sum_{k\in \Z} \|\pa_r \widehat{W}_k\|_{L^2((1,\oo);rdr)}^2 =  \sum_{k\in \Z} \|r^{\f12}\pa_r(r^{-\f12}w_k)\|_{L^2((1,\oo);dr)}^2,
    \end{align*}
which will be frequently used throughout this paper. Moreover, we deduce from \eqref{pertubation of NS vor-fourier} that
\begin{subequations}\label{1.10}
    \begin{gather}
\label{nonlinear equation of w_0}
\partial_t w_0+\mathcal{T}_0w_0=r^{-\f12}\sum_{\ell \in\mathbb{Z}\backslash\{0\}}i\ell\partial_r\big( r^{-1}\varphi_\ell w_{-\ell}\big),   \quad (t,r)\in \R^+ \times [1,+\oo), \\
\label{nonlinear equation of w_k}
\partial_t w_k + \mathcal{T}_k w_k
=\sum_{\ell \in\mathbb{Z}\backslash\{0\}}i\bigl[r^{-\f12}\ell\partial_r\big( r^{-1}\varphi_\ell w_{k-\ell}\big) - kr^{-1}\partial_r(r^{-\f12}\varphi_\ell)w_{k-\ell}\bigr]-ikr^{-\f32}u^\theta_0w_k, \\
\label{equation of varphi_k}
\Big(\partial_r^2-\f{k^2-\f14}{r^2}\Big)\varphi_k=w_k,\\
\label{expression of u}
u^r_k = -i\frac{k}{r}\varphi_k, \quad u^\theta_k = r^{\f12}\pa_r(r^{-\f12} \varphi_k)=\pa_r \varphi_k -\frac{1}{2r} \varphi_k,\\
\label{boundary conditions}
w_k|_{r=1}=\varphi_k|_{r=1}=0, \qquad \textrm{for any}\ k\in \mathbb{Z},
\end{gather}
\end{subequations}
where $\mathcal{T}_k$ is defined as
\begin{align}
\label{def of Tk}& \mathcal{T}_k {\eqdefa} -\nu\Big(\pa_r^2 - \frac{k^2-\f14}{r^2}\Big) +i\frac{kB}{r^2}, \quad D(\mathcal{T}_k){\eqdefa}H^2\bigl((1,\oo); dr\bigr)\cap H^1_0\bigl((1,\oo); dr\bigr).
\end{align}
So far, we have transformed the original problem into exploring the dynamic and long-time behavior of the system \eqref{1.10}.

\subsection{Key ideas}
It's well known that, for transport-diffusion equations, the enhanced dissipation scale is closely related to the degree of degeneracy of the velocity field, see for instance \cite{ABN22,BZ17}. However, as pointed out in \cite{LZZ-25}, due to the infinite degeneracy of TC flow in
the exterior disk (i.e., $\pa_r^n(r^{-2})$ vanishes at infinity for any $n\geq 0$), the previous results are no longer applicable to $e^{-\mathcal{T}_kt}$, which motivates us to investigate the mapping properties of $e^{-\mathcal{T}_kt
}$ in some space-time weighted spaces in \cite{LZZ-25}. In particular, we observe  from Theorem 1.2 of \cite{LZZ-25} that due to the degeneracy at infinity, the non-homogeneous equation \eqref{nonlinear equation of w_k} cannot exhibit exponential decay. Specifically,
\begin{lemma}[Theorem 1.2, \cite{LZZ-25}] \label{main them-2 of LZZ25}
For any positive continuous functions \( a_1(r) \), \( a_2(r) \) and strictly decreasing function \( \phi(r) \), there exists a sequence of functions,
 \(\{(w_n, f_n)\}\), that solves
\begin{align*}
  \partial_t w_n +\mathcal{T}_kw_n=f_n, \quad w_n|_{t=0}=0, \quad  w_n|_{r=1,\infty}=0,
\end{align*}
and satisfies for any $n \in \mathbb{N}$
\begin{align*}
\|e^{t \phi(r)} a_1(r)  w_n\|_{L^2(\R^+; \Ltr)} + \|e^{t\phi(r)} a_2(r) f_n\|_{L^2(\R^+; \Ltr)} < +\oo.
\end{align*}
However, there holds
\begin{equation*}
    \limsup_{n \rightarrow +\oo} \frac{\|e^{t\phi(r)} a_1(r)  w_n\|_{L^2(\R^+; \Ltr)}}{ \|e^{t\phi(r)} a_2(r) f_n\|_{L^2(\R^+; \Ltr)}} = +\oo.
\end{equation*}
\end{lemma}
Lemma \ref{main them-2 of LZZ25} indicates that even if employing space-time coupling to mitigate the adverse effects of degeneracy at infinity in the Taylor-Couette flow, one still cannot obtain exponential nonlinear enhanced dissipation decay as in the Couette flow. Consequently, for the long-time behavior near the Taylor-Couette flow in the exterior region, even with space-time coupling, at most polynomial nonlinear decay can be expected.

In fact, Theorem 1.1 of \cite{LZZ-25} establishes that the solution to the linear homogeneous equation exhibits polynomial space-time coupled decay.
Precisely:

\begin{lemma}[Theorem 1.1, \cite{LZZ-25}] \label{main them-1 LZZ25}
   Let $k\in \mathbb{Z}\backslash\{0\}$,  $q\in \mathbb{N}$, $\frac{\nu}{|kB|}\ll (1+q)^{-3}$ and let $w$ be determined by
 \begin{align*}
  \partial_t w +\mathcal{T}_kw=0, \quad w|_{t=0}=w(0), \quad  w|_{r=1,\infty}=0.
\end{align*}
    There exists constant $C=C(q)>0$ independent of $\nu,k,B$ and $w_k(0)$, so that
\begin{equation*}
   \bigl\|(1+\kappa_k t/r^2)^{q} w\bigr\|_{L^{\infty}(\R^+; \Ltr)} + \kappa_k^{1/2} \bigl\| (1+\kappa_k t/r^2)^{q} r^{-1} w\bigr\|_{L^2(\R^+; \Ltr)} \leq C \|w(0)\|_{\Ltr}.
\end{equation*}
\end{lemma}

The reasons for exploring the stability of the TC flow in the exterior disk list as follows:

\begin{itemize}
    \item[(1)] Firstly, due to the degeneracy of the Taylor-Couette flow at infinity, we cannot expect the solution to decay exponentially in a space-time decoupled manner. One may check \cite{CKR-2,CDE20} for details.

    \item[(2)]Lemma \ref{main them-2 of LZZ25} tells us that even the coupled space-time exponential decay is not valid.

   \item[(3)] Lemma \ref{main them-2 of LZZ25} and Lemma \ref{main them-1 LZZ25} tell us that for the fully nonlinear equation \eqref{nonlinear equation of w_k}, one can at most expect polynomial space-time coupling decay.
\end{itemize}

At this point, a new issue arises: to close the {\it a priori} estimates for the nonlinear system \eqref{1.8}, we need space-time coupling multipliers. The previous resolvent estimate methods \cite{W21,CLWZ-2D-C,DL22,AHL-2} are all aimed at space-time decoupled cases (essentially performing a Fourier transform in the time direction). Therefore, here we introduce time-dependent resolvent estimate in Section \ref{Resolvent estimates with time weight}.

In the case of polar coordinates for the exterior disk, there is a critical issue: the weight of the stream function with respect to \( r \) differs from \( w \) by \( 1/r^2 \). A simple basic fact is that the derivatives of \( 1/r^2 \) are non-vanishing in \([1,R]\) yet they
 may  converge to zero at infinity. Therefore, for technical convenience, we shall choose \( \Lambda_k \) as the long-time space-time decay weight. The basic properties of \( \Lambda_k \) will be provided with a detailed explanation after we present the main theorem.

\subsection{Main theorems}
The following result is the main result of this paper, which concerns global well-posedness and the asymptotic behavior of the solution to the
system  \eqref{1.10}.

\begin{theorem}\label{thm1}
 {\sl  Let $0< \varepsilon < 2$ and  $\hc$ be a large constant, we denote $\Lambda_k{\eqdefa}\log\big(\hat{C}r^2 +  \nu^\f13|kB|^\f23 t\big) $ to be the enhanced dissipation weight. Then there exist   constants  $C_\varepsilon, C'_\varepsilon, \hc_\varepsilon$ depending only on $\varepsilon$, such that if $ C_\varepsilon \geq \hc_\varepsilon$ and
\begin{equation}\label{S1eq1}
\begin{split}
   \mathcal{M}(0)&\eqdefa  \|rw_0(0)\|_\Ltr + \|u^\theta_0(0)\|_\Ltr +\sum_{k\in\Z\backslash \{0\}}  M_k(0)   \leq  C_\varepsilon |\nu B|^\f12, \with\\
    M_k(0)&{\eqdefa} \|r^{\varepsilon+2} w_k(0) \La_k(0,\cdot)\|_\Ltr + \|r^{\varepsilon+3} \pa_r w_k(0) \La_k(0,\cdot)\|_\Ltr, \quad \text{for}\ k\neq 0,
   \end{split}
\end{equation}
 the system \eqref{1.10} has a unique global solution $(w_0,w_k,\varphi_k,u_k^r,u_k^\theta)$ which  satisfies
\begin{equation}
\label{S1eq2}
\begin{split}
    \mathcal{E}(T)\eqdefa & E_0(T) +100
 \sum_{k\in \Z\backslash\{0\}}  E_k(T) \leq C'_\varepsilon   \mathcal{M}(0), \quad \text{for any} \ T >0\with\\
 E_0(T){\eqdefa}& \|rw_0\|_{\LoT(\Ltr)} + \| u^\theta_0\|_{\LoT(\Ltr)} + \nu^\f12 \|w_0\|_{\LtT (\Ltr)},\\
    E_k(T){\eqdefa}& \|r^{\varepsilon+1}w_k \La_k\|_{\LoT(\Ltr)}+  \mu_k^\frac{1}{2} \|r^{\varepsilon}w_k \La_k\|_{\LtT (\Ltr)}\\
    &+ |kB|^{\f12} |k|^\f12 \Bigl(\|r^{\varepsilon-1}\pa_r\varphi_k\La_k\|_{\LtT (\Ltr)} + |k|\|r^{\varepsilon-2}\varphi_k\La_k\|_{\LtT (\Ltr)}\Bigr), \quad \text{for }\ k\neq 0.
  \end{split}
\end{equation}}
\end{theorem}

It is easy to observe from  Lemma \ref{Appendix A1-1} and \eqref{expression of u} that
\begin{equation}\label{ineq 8.2}
    |kB|^\f12 |k| \Big\|r^{\varepsilon-\f32}\varphi_k\La_k\Big\|_{\LtT (\Lor)} \lesssim E_k(T).
\end{equation}

As a corollary of Theorem \ref{thm1}, we obtain the pointwise logarithmic-type nonlinear enhanced dissipation estimate.
\begin{theorem}\label{cor 1}
    Under the assumptions of Theorem \ref{thm1}, we have
    \begin{equation}\label{ineq 1.19}
        \|r^{1+\varepsilon} w_k(t)\|_\Ltr \lesssim_\varepsilon \frac{1}{\log\big(\hc+\nu^\f13 |kB|^\f23 t\big)} \mathcal{M}(0), \quad \text{for} \  k\neq 0, t>0,
    \end{equation}
    and
    \begin{equation}\label{ineq 1.20}
        \lim_{t\rightarrow +\oo } t^\f12\|w_0(t)\|_\Ltr =0.
    \end{equation}
\end{theorem}

We have the following remarks in order concerning Theorem  \ref{thm1}:


\no(a) \textbf{The Choice of \( \La_k \)}:

In fact, throughout this paper, we only require the following properties for \( \La_k \):
\begin{itemize}
    \item[A.] \( \mu_k^{-1} r^2 \left|\frac{\pa_t \La_k}{\La_k}\right| + r^2 \left|\frac{\pa_r^2 \La_k}{\La_k}\right| + r \left|\frac{\pa_r \La_k}{\La_k}\right| \ll 1,\) where $\mu_k{\eqdefa}\max\bigl(\nu k^2, \nu^\f13|kB|^\f23\bigr);$
    \item[B.] \( \La_k \leq \La_\ell \La_{k-\ell} \) for \( k, \ell \in \mathbb{Z} \);
    \item[C.]\( \frac{\La_k(t,r)}{\La_k(t,s)} \lesssim \max \left( \left( \frac{r}{s} \right)^{\varepsilon_1}, \left( \frac{s}{r} \right)^{\varepsilon_2} \right) \) for some \( \varepsilon_1, \varepsilon_2 \ll 1 \).
\end{itemize}
Rather than using the explicit expression for \( \La_k \),
\begin{itemize}
    \item[(1)] Property A will be used to absorb terms arising from commutators \( [\pa_t, \La_k] \) and \( [\pa_r^2, \La_k] \).
    \item[(2)]Property B is designed to deal with the frequency interactions in the nonlinear terms.
    \item[(3)]Property C will be used to recover the elliptic estimate (see Lemma \ref{basic properties of ck} and the estimation of terms involving \( \widetilde{K}_k \)).
\end{itemize}

\begin{itemize}
    \item Thus, if one prefers polynomially enhanced dissipation for \( w \), we can
\begin{align*}
  \text{replace
  }  \log\left( \hc r^2 + \nu^{\frac{1}{3}} |kB|^{\frac{2}{3}} t \right)  \text{ by
  }   \left( e + \frac{\nu^{\frac{1}{3}} |kB|^{\frac{2}{3}} t}{r^2} \right)^\delta ,
\end{align*}
 where \( \delta \ll_\varepsilon 1 \). However, for simplicity, this paper considers only the logarithmic-type enhanced dissipation.
  \item
Furthermore, we can replace \( \nu^{\frac{1}{3}} |kB|^{\frac{2}{3}} \) by \( \mu_k \) to provide better dissipation for high frequencies (\( |k| \geq \nu^{-\frac{1}{2}} |B|^{\frac{1}{2}} \)), and the proof is the same. We retain \( \nu^{\frac{1}{3}} |kB|^{\frac{2}{3}} \) in \( \La_k \) to emphasize the dissipation scale \( \mathcal{O}(\nu^{-\frac{1}{3}} |B|^{-\frac{2}{3}}) \).
\end{itemize}

\no (b) \textbf{The Criticality of Theorem \ref{thm1}}:
\begin{itemize}
    \item[(1)] If we impose additional spatial weights \( r^\beta \) (\( \beta > 0 \)) on \( w_0 \), beyond those in \( E_0(T) \), similar results as Theorem \ref{thm1} hold only for \( t \lesssim \nu^{-1} \). In this case, we cannot expect global stability results for general data, since Proposition \ref{prop lower bound of w_0} indicates that \( e^{-\mathcal{T}_0 t} \) is linearly unstable in these weighted spaces. On the other hand, if we remove some spatial weights \( r^\beta \) from \( E_0(T) \), we cannot close the energy estimate for \( E_k(T) \). Therefore, up to this point, \( E_0(T) \) is critical.
    \item[(2)]In the linear theory of TC flow in the exterior domain (see \cite{LZZ-25}), the linear evolution \( e^{-\mathcal{T}_k t} \) can exhibit polynomial dissipation \( (1 + \kappa_k t)^{-\delta} \) when \( \frac{\nu}{|kB|} \ll (1 + \delta)^{-3} \), at the cost of losing equivalent spatial weights. For nonlinear evolution, we can only deal with \( \delta \ll_\varepsilon 1 \), because the elliptic estimate
\[
\|r^{-2} \La_k \varphi_k\|_\Ltr \leq C \|\La_k w_k\|_\Ltr \quad \text{for some $C$ depending only on $\delta$},
\]
fails for $\La_k=\Big( e + \frac{\nu^{\frac{1}{3}} |kB|^{\frac{2}{3}} t}{r^2}\Big)^\delta$ and large \( |\delta| \). Therefore, up to this point, the enhanced dissipation weight \( \left( e + \frac{\nu^{\frac{1}{3}} |kB|^{\frac{2}{3}} t}{r^2} \right)^\delta \) with \( \delta \ll_\varepsilon 1 \) may be optimal for the nonlinear problem.
\end{itemize}

\no (c) \textbf{Transition threshold:}
The mathematic version of Transition threshold
was formulated by Bedrossian-Gremain-Masmoudi \cite{BGM-bams} as follows:
\begin{itemize}
	\item[]  Given a norm $\|\cdot\|_{X}$, find a $\beta=\beta(X)$ so that
	\begin{align*}
	\left\|  u_{\mathrm{in}}  \right\|_{X}& \leqslant \; \nu^{\beta } \,\, \Rightarrow  \text{ stability}, \nonumber \\
	\left\|  u_{\mathrm{in}}  \right\|_{X}&  \gg  \;\nu^{\beta } \,\, \Rightarrow  \text{ instability},
	\end{align*}
\end{itemize}
where the exponent $\beta=\beta (X)>0$  is called  the transition threshold. In the \( \La_k \)-integrable space in this paper, Theorem \ref{thm1} shows that we can obtain transition threshold estimates similar to those in bounded domain, as can be seen in \cite{AHL-2}, despite the infinite degeneracy  of $r^{-2}$ in exterior disk.

\no (d) \textbf{Biot-Savart Law:}
The corresponding integral kernels are different for $k\neq 0$ and $k=0$, specifically as follows:
\begin{itemize}
    \item For the case of \( k \neq 0 \) in \eqref{equation of varphi_k}, we consider \( u^\theta_k, u^r_k, rw_k \in L^2\big((1, \infty), dr\big) \). In this case, \(\varphi_k\) can be explicitly expressed as
\begin{equation}\label{expression of phi_k}
\begin{split}
   \varphi_k(r)= -\cK[w_k](r)\eqdefa&- \int_1^\oo K_k(r,s)w_k(s)\,ds \with \\
    K_k(r,s){\eqdefa}&\frac{1}{2|k|} (rs)^\frac{1}{2}  \Big( \min\bigl(\frac{r}{s}, \frac{s}{r}\bigr)^{|k|} - \f{1}{r^{|k|} s^{|k|}} \
    \Big).
   \end{split}
\end{equation}

Here we list some basic properties of the integral kernel function \( K_k(r,s) \):
\begin{subequations}
    \begin{gather}
\label{1.17}
  \pa_r K_k(r,s) = \frac{1}{2|k|} \Bigl( \bigl(|k|+\f12\bigr) \bigl(\frac{r}{s}\bigr)^{|k|-\f12} \mathbbm{1}_{r\leq s} + \bigl(\f12-|k|\bigr)\bigl(\frac{s}{r}\bigr)^{|k|+\f12} \mathbbm{1}_{s\leq r}  \\
 \nonumber \qquad \qquad \qquad \qquad\quad+ \bigl(|k|-\f12\bigr)\f{1}{r^{|k|+\f12} s^{|k|-\f12}}   \Bigr),\\
 \label{1.17b}K_k(1,s)=K_k(r,1)=\pa_rK_k(r,1)=0, \\
 \label{1.17c}\pa_r \pa_s K_k(r,s)= \delta(r-s) + \widetilde{K}_k(r,s) \quad  \text{in the sense of distribution},\\
\label{estimate of cK}
  |k|^{-1} rs |\widetilde{K}_k(r,s)| + s |\pa_s K_k(r,s)| +r |\pa_r K_k(r,s)| +  |k||K_k(r,s)| \lesssim  (rs)^{\f12}\min(\frac{r}{s},\frac{s}{r})^{|k|},
\end{gather}
\end{subequations}
where
\begin{align*}
    \widetilde{K}_k(r,s) = -\frac{k^2-\f14}{2|k|} \Big( \frac{r^{|k|-\f12}}{s^{|k|+\f12}} \mathbbm{1}_{r\leq s} + \frac{s^{|k|-\f12}}{r^{|k|+\f12}} \mathbbm{1}_{s\leq r}\Big) - \frac{(|k|-\f12)^2}{2|k|} \f{1}{r^{|k|+\f12} s^{|k|+\f12}}.
\end{align*}

 \item For the case of $k=0$, it's easy to observe that $u^r_0=0$ and it follows from $W=r^{-1}\pa_r\big(r U^\theta\big) - r^{-1}\pa_\theta U^r$
 that
\begin{equation}\label{w_0, u_0}
    w_0 = r^{-\f12}\pa_r\big(r^\f12 u^\theta_0\big)=\pa_r u^\theta_0 +\frac{1}{2r} u^\theta_0.
\end{equation}
It is natural to assume that \( u^\theta_0, rw_0 \in L^2\big((1,\infty),dr\big) \). Then, according to Lemma \ref{Appendix A1-1}, we obtain
\[
\lim\limits_{r\rightarrow\infty} |r^{\frac{1}{2}} u^\theta_0|=0.
\]
Thus, in view of \eqref{w_0, u_0}, we have
\begin{equation}\label{w_0, u_0,a}
    u^\theta_0(r)=-r^{-\frac{1}{2}}\int_r^\infty s^{\frac{1}{2}} w_0(s)ds.
\end{equation}
Due to the logarithmic singularity of \(\varphi_0\) at \( r=\infty \), we directly use its explicit expression for the velocity field \( u^\theta_0 \) instead of recovering it through the stream function.
\end{itemize}

Finally, we sketch the organization of the paper:

\begin{itemize}
    \item In Section \ref{Fundamental energy estimates}, we present some fundamental energy estimates;
    \item In Sections \ref{Resolvent estimates without enhanced dissipation weight} and \ref{Resolvent estimates with time weight}, we discuss the resolvent estimates with or without the enhanced dissipation weight \( \La_k \);
    \item In Section \ref{Space-time estimates of the vorticity in nonzero modes}, we use the resolvent estimates to derive the estimates of \( E_k(T) \) for \( k \neq 0 \);
    \item   In Section \ref{Sect7}, we  provide the proofs of Theorem \ref{thm1} and Theorem \ref{cor 1}.
\end{itemize}

\subsection{Notations}
Throughout the whole paper, we denote $$\Omega{\eqdefa} \bigl\{x\in \R^2: \ |x|\geq 1 \bigr\}= \bigl\{(r,\theta): \ r\in[1,\oo), \theta \in \T \bigr\}.$$
For $1\leq p \leq \oo$, we denote
\begin{align*}
    &\qquad \qquad  \qquad \qquad \LpO{\eqdefa} L^p(\Omega; dx) = L^p\big((1,\oo) \times \T;  rdrd\theta\big), \\
    &\Lpr{\eqdefa} L^p\big((1,\oo); dr\big), \quad
    \LpT{\eqdefa} L^p\big((0,T);dt\big),\quad H^k_r {\eqdefa} H^k\big((1,\oo); dr\big), \quad  H^k_{r,0} {\eqdefa} H^k_0\big((1,\oo); dr\big).
    &
\end{align*}
By  $\braket{\cdot, \cdot}_\Ltr$, we mean the natural inner product in the Hilbert space $\Ltr$, i.e.,
\begin{equation*}
    \braket{f,g}_{\Ltr} {\eqdefa} \int_1^{\oo} f \bar{g} dr,
\end{equation*}
and $\braket{\cdot,\cdot}_\LtO$ is similarly defined.

By $w'$, we mean $\pa_r w$. And  $\mathbbm{1}_E(r)$ designates the  characteristic function  on $E\subset [1,\oo)$. We denote $g\lesssim f$ if there exists a constant $C>0$ so that $g\leq C f$, moreover we denote $g \ll f$ if this $C$ is universal and small enough, and we use $g\lesssim_\varepsilon f$ to emphasize dependence of $C$ on $\varepsilon$.
Throughtout this paper, we denote
\begin{equation*}
\kappa_k {\eqdefa}\nu^\f13|kB|^\f23, \quad \mu_k{\eqdefa}\max(\nu k^2, \kappa_k),\quad \cT_k \eqdefa-\nu\Big(\pa_r^2 - \frac{k^2-\f14}{r^2}\Big) +i\frac{kB}{r^2}, \quad D{\eqdefa}H^2_r\cap H^1_{r,0}.
\end{equation*}

\section{Fundamental energy estimates}\label{Fundamental energy estimates}

In this section, we first establish some fundamental energy estimates for the fully nonlinear equations  \eqref{velocity pertubation of the Taylor-Couette flow} and \eqref{vorticity pertubation of the Taylor-Couette flow}.


\begin{lemma}\label{vorticity conservation of L2}
{\sl Let   $W$  be  a smooth enough solution of  (\ref{vorticity pertubation of the Taylor-Couette flow}) on $[0,T].$ Then, for any \( T \in (0,T) \), one has
\begin{align}
\label{vorticity, basic energy conservation equality}\|W(t)\|_\LtO^2+2\nu\|\nabla W\|_{L^2_t(\LtO)}^2 =\|W(0)\|_\LtO^2.
\end{align}
}
\end{lemma}

\begin{proof} Due to $\dive U=0$ and \eqref{impermeable & navier slip boundary condition}, we get, by taking $\LtO$ inner product  of the equation (\ref{vorticity pertubation of the Taylor-Couette flow}) with $W$ and using integration by parts, that
\begin{align*}
    0&=\frac{1}{2} \frac{d}{dt} \|W(t)\|_\LtO^2 + \nu \|\nabla W\|_\LtO^2  + \f12 \int_{\pa \Omega} (V^\star + U)\cdot \mathbf{n} |W|^2 dS - \nu \int_{\pa \Omega} \nabla W \cdot \mathbf{n} W dS \\
    &=\frac{1}{2} \frac{d}{dt} \|W(t)\|_\LtO^2 + \nu \|\nabla W\|_\LtO^2.
 \end{align*}
By integrating the above inequality over $[0,t]$, we obtain \eqref{vorticity, basic energy conservation equality}.
\end{proof}

\begin{lemma}
{\sl Under the assumption that \( U \in H^2(\Omega) \), \( \dive U = 0 \), \( W = \curl U \), and  \( U \) satisfies the boundary condition \eqref{impermeable & navier slip boundary condition}, we have
\begin{equation}\label{2.4}
     - \int_\Omega \lap U\cdot U dx = \|W\|_\LtO^2.
\end{equation}}
\end{lemma}
\begin{proof} Let $\mathbf{n}=(\mathbf{n}^1, \mathbf{n}^2)$ be the outward normal vector to the domain $\Omega.$ We  get, by
using integration by parts, that
\begin{equation*}
    \begin{split}
\|W\|_\LtO^2 &= \int_\Omega (\pa_1 U^2 - \pa_2 U^1)^2 dx =\int_\Omega (\pa_1 U^2)^2 dx  + \int_\Omega(\pa_2 U^1)^2dx - 2\int_\Omega \pa_1 U^2 \pa_2 U^1 dx \\
& = \int_{\pa \Omega} U^2 \mathbf{n}^1\pa_1 U^2 dS - \int_\Omega U^2 \pa_1^2 U^2 dx +\int_{\pa \Omega} U^1 \mathbf{n}^2\pa_2 U^1 dS - \int_\Omega U^1 \pa_2^2 U^1 dx  \\
&\quad  - \int_{\pa \Omega} U^2\mathbf{n}^1 \pa_2 U^1 dS - \int_{\pa \Omega} U^1\mathbf{n}^2 \pa_1 U^2 dS + \int_\Omega U^2 \pa_2\pa_1 U^1 + U^1\pa_1 \pa_2 U^2 dx,
\end{split}
\end{equation*}
from which,   $\pa_1 U^1 +\pa_2 U^2=0$ and $(\pa_1 U^2 - \pa_2 U^1)|_{\pa \Omega}=0$, we infer
\begin{align*}
\|W\|_\LtO^2 &= \int_{\pa \Omega}  (U^2 \mathbf{n}^1- U^1 \mathbf{n}^2) (\pa_1 U^2 - \pa_2 U^1) dS - \int_\Omega \lap U\cdot U dx \\
&=  - \int_\Omega \lap U\cdot U dx.
    \end{align*}
This leads to \eqref{2.4}.
\end{proof}

\begin{lemma}\label{velocity, basic energy conservation equality}
  {\sl  Let $U$ be a smooth enough solution of \eqref{velocity pertubation of the Taylor-Couette flow} on $[0,T].$ Then for any $t\in (0,T)$,
  one has
\begin{equation}\label{S2eq1}
\begin{split}
&\|U(t)\|_\LtO^2  + 2\nu \|W\|_{L^2_t(\LtO)}^2 \\
&\qquad \qquad\leq \|U(0)\|_\LtO^2 + 8\pi |B| \sum_{k\in \Z\backslash \{0\}} |k|\big\|\frac{\varphi_k}{r^2}\big\|_{L^2_t(\Ltr)} \Big\|\frac{\pa_r \varphi_{-k}}{r}\Big\|_{L^2_t(\Ltr)}.
\end{split}
\end{equation}}
\end{lemma}

\begin{proof} Thanks to  \eqref{impermeable & navier slip boundary condition}, $V^\star\cdot \mathbf{n}|_{\pa \Omega}=0$ and \eqref{2.4},
we get,
    by taking $\LtO$ inner product of \eqref{velocity pertubation of the Taylor-Couette flow} with $U$, that
\begin{equation}
\begin{split}\label{2.3}
    \frac{1}{2}\frac{d}{dt}\|U\|_\LtO^2 +\nu \|W\|_\LtO^2 = - \int_\Omega (U\cdot \nabla )V^\star \cdot U dx.
\end{split}
\end{equation}
Yet by virtue of \eqref{S1eqq} and the fact: $\pa_r \mathbf{e_\theta}=0, \pa_\theta \mathbf{e_\theta}=-\mathbf{e_r}$, we write
    \begin{align*}
\int_\Omega (U\cdot \nabla )V^\star \cdot U dx &= \int_\Omega \big(U^r \pa_r + U^\theta \frac{\pa_\theta}{r} \big)\big(Ar \mathbf{e_\theta} + \frac{B}{r}\mathbf{e_\theta}\big) \cdot (U^r\mathbf{e_r}+U^\theta\mathbf{e_\theta}) dx \\
&= \int_\Omega \Big( \big(A-\frac{B}{r^2}\big)U^r \mathbf{e_\theta} -\big(A+\frac{B}{r^2}\big)U^\theta \mathbf{e_r} \Big) \cdot (U^r\mathbf{e_r}+U^\theta\mathbf{e_\theta}) dx\\
&= -2B \int_\Omega \frac{U^r U^\theta}{r^2} dx,
\end{align*}
from which, \eqref{Fourier decomposition}, \eqref{definition of w_k}, \eqref{expression of u} and $\varphi_{-k}=\overline{\varphi_{k}}$ (since $\varphi$ is real function), we infer
\begin{align*}
\int_\Omega (U\cdot \nabla )V^\star \cdot U dx &
 = -4\pi B \sum_{k\in \Z}\int_1^\oo \frac{\widehat{U}^r_k \widehat{U}^\theta_{-k}}{r^2} rdr= -4\pi B \sum_{k\in \Z}\int_1^\oo  \frac{u^r_k u^\theta_{-k}}{r^2} dr\\
& = 4\pi B \sum_{k\in \Z\backslash \{0\}} ik \int_1^\oo  \frac{\varphi_k \pa_r\varphi_{-k}}{r^3} dr - 4\pi B \sum_{k\in \Z\backslash \{0\}} ik \int_1^\oo  \frac{\varphi_k \varphi_{-k}}{2r^4} dr \\
& = 4\pi B \Re \sum_{k\in \Z\backslash \{0\}} ik \int_1^\oo  \frac{\varphi_k \pa_r\varphi_{-k}}{r^3} dr.
    \end{align*}
By inserting the above inequality into \eqref{2.3}, we find
\begin{equation*}
    \frac{1}{2}\frac{d}{dt}\|U\|_\LtO^2  + \nu \|W\|_\LtO^2 \leq 4\pi|B| \sum_{k\in \Z\backslash \{0\}} |k|\big\|\frac{\varphi_k}{r^2}\big\|_\Ltr \Big\|\frac{\pa_r \varphi_{-k}}{r}\Big\|_\Ltr.
\end{equation*}
By integrating the above inequality over $[0,t]$, we achieve \eqref{S2eq1}.
\end{proof}

Next, we study the mapping properties of the semigroup $e^{-\mathcal{T}_kt}$ in polynomial weighted spaces.

\begin{lemma}\label{Coercive estimates in general polynomial-weighted spaces}
{\sl Let  $\al \in \R$ and $r^\al w\in D.$ Then one has
 \begin{subequations} \label{S2eqw}
\begin{gather}
\label{Coercive estimates} \Re\braket{\mathcal{T}_kw,r^{2\alpha}w}_\Ltr \geq\nu(k^2-\alpha^2)\|r^{\alpha-1}w\|_{\Ltr}^2,\\
\label{weighted trivial w'}
\Re\braket{\mathcal{T}_k w,r^{2\al}w}_\Ltr =  \nu\|r^{\al}w'\|_\Ltr^2 +\nu\Big(k^2-\f14-2\al^2+\al\Big)\|r^{\alpha-1}w\|_\Ltr^2.
\end{gather}
\end{subequations}}
\end{lemma}

\begin{proof} We first get,
by using Lemma \ref{Coercive estimates prior}, that
  \begin{equation*}\label{L_k w, r al w inner product}
\Re\braket{\mathcal{T}_kw,r^{2\alpha}w}_\Ltr=\nu\Bigl(\|r^{\alpha}w'\|_\Ltr^2-\alpha(2\alpha-1)\|r^{\alpha-1}w\|_\Ltr^2\Bigr)+\nu\Big(k^2-\f14\Big)\|r^{\alpha-1}w\|_\Ltr^2,
\end{equation*}
which implies \eqref{weighted trivial w'}.

 By applying \eqref{weighted trivial w'} and  Lemma \ref{pair-w' with w}, we obtain
\begin{align*}
    \Re\langle \mathcal{T}_kw,r^{2\alpha}w\rangle_\Ltr\geq &\nu\Big((\alpha- 1/2)^2\|r^{\alpha-1}w\|_{\Ltr}^2-\alpha(2\alpha-1)\|r^{\alpha-1}w\|_{\Ltr}^2\Big)+\nu\Big(k^2-\f14\Big)\|r^{\alpha-1}w\|_{\Ltr}^2\\
=&\nu(k^2-\alpha^2)\|r^{\alpha-1}w\|_{\Ltr}^2.
\end{align*}
This completes the proof of \eqref{Coercive estimates}.
 \end{proof}

\begin{proposition}\label{decreasing energy estimates}
{\sl Let $w$ be a smooth enough solution of the following linear equation:
\begin{align}\label{equation with general shearing profile S}
\partial_t w-\nu\Big(\partial_r^2-\f{k^2-\f14}{r^2}\Big)w+ iS(r)w=0 \quad \textrm{with} \ w|_{r=1,\infty}=0,
\end{align}
where the shearing profile $S(r) \in C_b([1,\oo))$ is a real function. Then  as long as $|\al| \leq |k|,$ $\|r^\al w(t)\|_\Ltr$ is nonincreasing with respect to time $t\geq0.$}
\end{proposition}
\begin{proof}
By taking  the real part of the $\Ltr$ inner product of (\ref{equation with general shearing profile S}) with $r^{2\alpha}w$ and then using Lemma \ref{pair-w' with w}, we deduce that for $|\al| \leq |k|$
\begin{align*}
0&=\f12\frac{d}{dt}\|r^{\alpha}w(t)\|_\Ltr^2+\Re\Big\langle-\nu\Big(\partial_r^2-\f{k^2-\f14}{r^2}\Big)w+ iS(r)w,r^{2\alpha}w\Big\rangle_\Ltr
\\
&\geq\f12 \frac{d}{dt} \|r^{\alpha}w(t)\|_\Ltr^2+\nu(k^2-\alpha^2)\|r^{\alpha-1}w\|_\Ltr^2,
\end{align*}
which implies that $\|r^\al w(t)\|_\Ltr$ is nonincreasing with respect to  time $t\geq0$.
\end{proof}

The above proposition provides a basic polynomially weighted energy estimate for the
k-mode equation with a general shearing profile $S(r)$. However, the following proposition shows that even if $S(r)=0$ and $|\al| >|k|$, we may not have a uniform bound in time for $\|r^\al w\|_\Ltr$.

\begin{proposition}\label{prop lower bound of w_0}
{\sl   There exists $C>0$, so that for any increasing function $f(r) \geq 0$ and
  initial data $ 0\leq w_g(r)\in C^\oo_c(2,4),$ if $\nu t\geq 6,$ there holds
\begin{equation}\label{ineq lower bound of w_0}
\big\|rf(r)e^{- \mathcal{T}_0 t}w_g\big\|_\Ltr \geq C f(\sqrt{\nu t})\|w_g\|_\Llr.
\end{equation}}
\end{proposition}

\begin{proof}
We first observe that   $w{\eqdefa} e^{-\mathcal{T}_0t}w_g$ solves
\begin{align*}
\left\{
\begin{aligned}
    &\partial_t w-\nu\Big(\partial_r^2+\f{1}{4r^2}\Big)w= 0, \quad (t,r)\in \R^+ \times [1,+\oo),\\
   & w(t=0)=w_g,\quad w|_{r=1,\oo}=0.
\end{aligned}
\right.
\end{align*}
Then $\eta(x){\eqdefa} |x|^{-\f12}w(|x|)$ verifies
\begin{align*}
\left\{
\begin{aligned}
    &\partial_t \eta-\nu \lap \eta=\partial_t \eta-\nu(\partial_r^2+\f{1}{r}\partial_r) \eta= 0, \quad (t,x) \in \R^+\times \Omega,\\
    &\eta(0,x)=|x|^{-\f12}w_g(|x|),\quad \eta|_{\pa \Omega}=0.
\end{aligned}
\right.
\end{align*}
If we denote the heat kernel of Dirichlet Laplacian on $\Omega$ by $p_\Omega(t,x,y)$, then one has
\begin{equation}
    \eta(t,x)= \int_{|y|\geq 1} p_\Omega(\nu t,x,y)|y|^{-\f12}w_g(|y|)\, dy =\int_{2\leq |y|\leq 4} p_\Omega(\nu t,x,y)|y|^{-\f12}w_g(|y|)\, dy.
\end{equation}
By virtue of Theorem \ref{GS,ZQ03} below, we get for $|x|\geq 2, \nu t\geq 6,$
\begin{equation*}
    \begin{split}
\eta(t,x)\gtrsim \int_{2\leq |y|\leq 4} \frac{1}{\nu t}e^{-c_2\frac{|x-y|^2}{\nu t}} w_g(|y|)\, dy \gtrsim \frac{1}{\nu t} e^{-c_3 \frac{|x|^2}{\nu t}} \|w_g\|_\Llr,
    \end{split}
\end{equation*}
 which implies
\begin{equation*}
    \begin{split}
\big\|rf(r)e^{-\mathcal{T}_0 t}w_g\big\|_\Ltr^2& = \int_{|x|\geq 1} |x|^2f^2(|x|) |\eta(t,x)|^2 \, dx \geq f^2(\sqrt{\nu t}) \int_{|x|\geq \sqrt{\nu t}} |x|^2 |\eta(t,x)|^2 \, dx \\
&\gtrsim f^2(\sqrt{\nu t}) \int_{|x|\geq \sqrt{\nu t}} \frac{|x|^2}{(\nu t)^2} e^{-2c_3 \frac{|x|^2}{\nu t}} \, dx \|w_g\|_\Llr^2 \gtrsim f^2(\sqrt{\nu t})\|w_g\|_\Llr^2.
    \end{split}
\end{equation*}
This completes the proof of \eqref{ineq lower bound of w_0}.
\end{proof}

\begin{theorem}[\cite{GS,ZQ03}]\label{GS,ZQ03}
{\sl There exist some constants $c_1,c_2>0$, so that for $|x|,|y|\geq 2$, $t\geq 1$, there holds
\begin{equation}
    \frac{c_1}{t}e^{-c_2\frac{|x-y|^2}{t}} \leq p_\Omega(t,x,y) \leq \frac{1}{c_1t}e^{-\frac{|x-y|^2}{c_2t}}.
\end{equation}}\emph{}
\end{theorem}

However, in the following sections, we will show that, due to the special transport structure of Taylor-Couette flow, i.e., $S(r)=\frac{kB}{r^2}$, a uniform bound in time for $\|r^\al w\|_\Ltr$ can be obtained for $\nu/|kB|\ll(1+|\al|)^{-3}$. This is one of the key ideas of this paper.

\section{Resolvent estimates without enhanced dissipation weight}\label{Resolvent estimates without enhanced dissipation weight}

In order to a uniform operator norm estimate for the semigroup $e^{-t\mathcal{T}_k}$, we first investigate the
following fundamental resolvent equation:
\begin{align}\label{vorticity eqn}
& (\mathcal{T}_k -ikB\lambda)w= -\nu\Big(\partial_r^2-\f{k^2-\f14}{r^2}\Big)w+ikB(r^{-2}-\lambda)w=F,
\end{align}
with $w\in \Htr\cap\Hlro$, $\lambda\in\mathbb{R}$ and $|k|\geq 1$. By density argument, we may always assume without loss of generality that $\varphi, w,F \in \mathcal{C}^\oo_c([1,\oo))$ in this section, unless otherwise specified. To proceed, we introduce:
\begin{equation}\label{S3eq1}
     \|w\|_\hkz^2 {\eqdefa} \|\partial_r w\|_\Ltr^2 + (k^2 -\frac{1}{4}) \Big\|\frac{w}{r}\Big\|_\Ltr^2
\andf
     \|w\|_\hkf{\eqdefa} \inf_{w=\pa_rg_1+\frac{k}{r}g_2}  \bigl(\|g_1\|_\Ltr + \|g_2\|_\Ltr\bigr).
\end{equation}
Then for $w\in H^1_{r,0} $, $F\in H^{-1}_r $, $k\in\Z\backslash\{0\}$ and $\al\in \R$, we have
\begin{equation}\label{equivalent definition of weighted Hkf}
\begin{split}
   & |\braket{w,F}_\Ltr| \leq \|w\|_\hkz \|F\|_\hkf,\\
& \|r^{1+\al} w(s)\|_\hkf \approx \inf_{w=\pa_rF_1+\frac{k}{r}F_2} \bigl(\|r^{1+\al }F_1\|_\Ltr + \|r^{1+\al} F_2\|_\Ltr  \bigr).
\end{split}
\end{equation}

Below
we fix some decreasing function $\rho\in \mathcal{C}^\oo(\R)$ satisfying
$$\rho(z)=1, \ \text{for} \ z\leq -1, \qquad  \rho(z)=-1, \ \text{for} \ z\geq 1.$$

\begin{proposition}\label{prop 3.1}
  {\sl   For any $k\in\Z\backslash\{0\}$, $\al\in\R$, $\lambda\in\mathbb{R}$, if we assume that  $\frac{\nu}{|kB|}\ll (1+|\al|)^{-3}$, there exists constant $C>0$ independent of $\nu,k,B,\lambda,\al$ so that
\begin{equation}\label{ineq 3.2}
 \nu\|r^{\al}w'\|_\Ltr^2 + \mu_k \|r^{\al-1}w\|_\Ltr^2 \leq C\bigl( \big|\braket{F,r^{2\al}w}_\Ltr\big| + \big|\langle F,r^{2\al}\tilde{\rho} w\rangle_\Ltr\big|\mathbbm{1}_{(0,1)}(\lambda) \bigr),
\end{equation}
where $\mu_k=\max\big(\nu k^2, \nu^\f13|kB|^\f23\big)$ and  $\tilde{\rho} {\eqdefa} \rho\Big(\frac{r-\lambda^{-\f12}}{\nu^\f13|kB|^{-\f13}\lambda
^{-\f12}} \Big)$ for $\lambda\in(0,1)$.

Moreover, we have
\begin{subequations} \label{S3eq2}
\begin{gather}
\label{ineq 3.3}
    \nu^\f12 \mu_k^{\f12}\|r^\al w'\|_{\Ltr}+\mu_k \|r^{\al-1}w\|_{\Ltr} +|kB|\|r^{\al+1}(r^{-2}-\lambda)w\|_{\Ltr}\leq C \|r^{\al+1}F\|_{\Ltr},\\
\label{ineq 3.4}
     \nu \|r^\al w\|_\hkz + \nu^\f12 \mu_k^\f12\|r^{\al-1} w\|_\Ltr+|kB|\|r^{\al}(r^{-2}-\lambda)w \|_\hkf \leq C(1+|\al/k|) \|r^\al F\|_\hkf.
\end{gather}
\end{subequations}}
\end{proposition}
\begin{proof}
We first recall from \cite{LZZ-25} that \eqref{ineq 3.2} holds for $\al=0$, i.e., for $\nu\ll |kB|$, there holds
\begin{equation}\label{3.6,a}
    \nu\|w'\|_\Ltr^2 + \mu_k \|w/r\|_\Ltr^2 \leq C\bigl( \big|\braket{F,w}_\Ltr\big| + \big|\langle F,\tilde{\rho} w\rangle_\Ltr\big|\mathbbm{1}_{(0,1)}(\lambda) \bigr).
\end{equation}

 For arbitrary $\al\in \R$, we observe from  \eqref{vorticity eqn} that $\eta {\eqdefa} r^\al w$ and $G{\eqdefa} r^\al F$ satisfy
\begin{equation}\label{3.7,a}
    -\nu\Big(\partial_r^2-\f{k^2-\f14}{r^2}\Big)\eta+ikB\big(r^{-2}-\lambda\big)\eta=G +\nu(\al +\al^2) r^{-2}\eta -2\nu\al r^{-1}\pa_r \eta.
\end{equation}
Then we get, by applying \eqref{3.6,a} and setting $\tilde{\al}{\eqdefa}1+|\al|$, that
\begin{equation*}
    \begin{split}
        \nu\|\eta'\|_\Ltr^2 + \mu_k \|\eta/r\|_\Ltr^2 &\leq C\bigl( \big|\braket{G,\eta}_\Ltr\big| + \big|\langle G,\tilde{\rho} \eta\rangle_\Ltr\big|\mathbbm{1}_{(0,1)}(\lambda) \bigr) \\
        & \quad +C\nu \big|\braket{(\al +\al^2) r^{-2}\eta -2\al r^{-1}\pa_r \eta,\eta}_\Ltr\big| \\
        &\quad +C\nu \big|\langle (\al +\al^2) r^{-2}\eta -2\al r^{-1}\pa_r \eta,\tilde{\rho} \eta\rangle_\Ltr\big|\mathbbm{1}_{(0,1)}(\lambda) \\
        &\leq C\bigl( \big|\braket{G,\eta}_\Ltr\big| + \big|\langle G,\tilde{\rho} \eta\rangle_\Ltr\big|\mathbbm{1}_{(0,1)}(\lambda) \bigr) + C\nu \tilde{\al}^2 \|\eta/r\|_\Ltr^2 + \f12 \nu \|\eta'\|_\Ltr^2.
    \end{split}
\end{equation*}
Notice that $\nu\tilde{\al}^2 \ll \mu_k$ if $\frac{\nu}{|kB|}\ll \tilde{\al}^{-3}$, so that we deduce from the above inequality that
\begin{equation*}
    \nu\|\eta'\|_\Ltr^2 + \mu_k \|\eta/r\|_\Ltr^2 \leq C\bigl( \big|\braket{F,r^{2\al}w}_\Ltr\big| + \big|\langle F,r^{2\al}\tilde{\rho} w\rangle_\Ltr\big|\mathbbm{1}_{(0,1)}(\lambda) \bigr),
\end{equation*}
which together with the fact:
\begin{equation*}
   \nu\|r^{\al}w\|_\hkz^2 + \mu_k \|r^{\al-1}w\|_\Ltr^2 \approx \nu\|r^{\al}w'\|_\Ltr^2 + \mu_k \|r^{\al-1}w\|_\Ltr^2 \approx  \nu\|\eta'\|_\Ltr^2 + \mu_k \|\eta/r\|_\Ltr^2
\end{equation*}
ensures \eqref{ineq 3.2}.

Next let us turn to the proof of  \eqref{ineq 3.3} and \eqref{ineq 3.4}. We first observe from \eqref{equivalent definition of weighted Hkf} that
\begin{align}\label{3.10}
    \big|\langle F,r^{2\al}w\rangle_\Ltr\big|  \leq \min\big(\|r^{\al+1}F\|_\Ltr \|r^{\al-1}w\|_\Ltr ,\  \|r^{\al}F\|_\hkf \|r^{\al}w\|_\hkz \big),
\end{align}
and for $\lambda\in (0,1)$,
\begin{align*}
 \big|\langle F,r^{2\al}\tilde{\rho} w\rangle_\Ltr\big|  &\leq \min\bigl(\|r^{\al+1}F\|_\Ltr \|r^{\al-1}\tilde{\rho} w\|_\Ltr ,\  \|r^{\al}F\|_\hkf \|r^{\al} \tilde{\rho} w\|_\hkz \bigr)\\
    &\lesssim \min\bigl(\|r^{\al+1}F\|_\Ltr \|r^{\al-1}w\|_\Ltr ,\  \|r^{\al}F\|_\hkf (\|r^{\al} w\|_\hkz + \| \pa_r\tilde{\rho}  r^{\al}w\|_\Ltr  ) \bigr). \end{align*}
  Due to   $\frac{\nu}{|kB|}\ll (1+|\al|)^{-3}$,  $\nu^\f13|kB|^{-\f13}$ is sufficiently small, so that one has
  \begin{align*}
    \| \pa_r\tilde{\rho}  r^{\al}w\|_\Ltr \lesssim \frac{\lambda^{-\f12}}{\nu^\f13|kB|^{-\f13}\lambda^{-\f12}}\| r^{\al-1}w\|_\Ltr\lesssim \nu^{-\f13}|kB|^{\f13}\| r^{\al-1}w\|_\Ltr,
    \end{align*}
    as a result, it comes out
\begin{equation}\begin{split}\label{3.11}
  \big|\langle F,r^{2\al}\tilde{\rho} w\rangle_\Ltr\big|
      \lesssim  \min\Big(&\|r^{\al+1}F\|_\Ltr \|r^{\al-1}w\|_\Ltr ,\\
      &\quad \|r^{\al}F\|_\hkf \big(\|r^{\al} w\|_\hkz + \nu^{-\f13}|kB|^{\f13}\| r^{\al-1}w\|_\Ltr \big) \Big).
\end{split}\end{equation}
Then we deduce from \eqref{ineq 3.2}, \eqref{3.10} and \eqref{3.11} that
\begin{align*}
 \nu\|r^{\al}w'\|_\Ltr^2 + \mu_k \|r^{\al-1}w\|_\Ltr^2 \leq & C\|r^{\al+1}F\|_\Ltr \|r^{\al-1}w\|_\Ltr\leq C\mu_k^{-1}\|r^{\al+1}F\|_\Ltr^2
 +\f{\mu_k}2 \|r^{\al-1}w\|_\Ltr^2,
 \end{align*}
 and
 \begin{align*}
 \nu\|&r^{\al}w'\|_\Ltr^2 + \mu_k \|r^{\al-1}w\|_\Ltr^2\\ \leq & C\|r^{\al}F\|_\hkf\bigl( \|r^{\al}w\|_\hkz
 + \nu^{-\f13}|kB|^{\f13}\| r^{\al-1}w\|_\Ltr\bigr)\\
 \leq & C\bigl(\nu^{-1}+\nu^{-\f23}|kB|^{\f23}\mu_k^{-1}\bigr)\|r^{\al}F\|_\hkf^2+\f12\bigl({\nu}\|r^{\al}w\|_\hkz^2+\mu_k \|r^{\al-1}w\|_\Ltr^2\bigr),
 \end{align*}
 from which and the fact that $|kB|^{\f23}\mu_k^{-1}\leq \nu^{-\f13},$ we infer
 \begin{subequations} \label{S3eq4}
\begin{gather}
  \label{S3eq4a} \nu^\f12 \mu_k^{\f12}\|r^\al w'\|_{\Ltr}+\mu_k \|r^{\al-1}w\|_{\Ltr} \leq C \|r^{\al+1}F\|_{\Ltr},\\
 \label{S3eq4b} \nu \|r^\al w\|_\hkz + \nu^\f12 \mu_k^\f12\|r^{\al-1} w\|_\Ltr \leq C \|r^\al F\|_\hkf.
\end{gather}
\end{subequations}

Thanks to \eqref{S3eq4},
 it suffices to estimate $(r^{-2}-\lambda)w$ in order to complete the proof of \eqref{S3eq2}. Notice that
    \begin{align*}
        \|kBr^{\al}(r^{-2}-\lambda)w\|_\hkf&\leq \|r^\al F\|_\hkf + \nu \Big\|r^\al\Big(\partial_r^2-\f{k^2-\f14}{r^2}\Big)w\Big\|_\hkf\\
        &\lesssim \|r^\al F\|_\hkf + \nu\Big(1+|\frac{\al}{k}|\Big) \|r^{\al}\pa_r w\|_\Ltr + \nu |k| \|r^{\al-1}w\|_\Ltr\\
        &\lesssim \Big(1+|\frac{\al}{k}| + \frac{\nu^\f12 |k|}{\mu_k^\f12}\Big) \|r^\al F\|_\hkf \lesssim (1+|\al/k|)  \|r^\al F\|_\hkf,
    \end{align*}
from which and \eqref{S3eq4b}, we obtain \eqref{ineq 3.4}.

On the other hand, we get, by using integration by parts, that
\begin{align*}
\langle F&, r^{2\al+2}(r^{-2}-\lambda)w\rangle_\Ltr\\
=&\big\langle-\nu\big(\partial_r^2-(k^2-1/4)r^{-2}\big)w+ikB\big(r^{-2}-\lambda\big)w,r^{2\al+2}\big(r^{-2}-\lambda\big)w\big\rangle_\Ltr\\
=&\nu\big\langle w',\big(r^{2\al+2}(r^{-2}-\lambda)w\big)'\big\rangle_\Ltr+\nu(k^2-1/4)\big\langle w,r^{2\al}(r^{-2}-\lambda)w\big\rangle_\Ltr+ikB\|r^{\al+1}(r^{-2}-\lambda)w\|_\Ltr^2\\
=&\nu\big\langle w',r^{2\al+2}(r^{-2}-\lambda)w'\big\rangle_\Ltr+ \nu \big\langle w',\big( 2\al r^{2\al-1} -\lambda (2\al+2)r^{2\al+1}\big)w\big\rangle_\Ltr \\
&+\nu(k^2-1/4)\big\langle w,r^{2\al}(r^{-2}-\lambda)w\big\rangle_\Ltr+ikB\|r^{\al+1}(r^{-2}-\lambda)w\|_\Ltr^2.
\end{align*}
By taking imaginary part of the above inequality and  rewriting
$$2\al r^{2\al-1} -\lambda (2\al+2)r^{2\al+1} = 2(1+\al)r^{2\al+1}(r^{-2}-\lambda) -2 r^{2\al-1} ,$$
we obtain
\begin{align*}
& |kB|\|r^{\al+1}(r^{-2}-\lambda)w\|_\Ltr^2\\
&\leq \big|\langle F,r^{2\al+2}(r^{-2}-\lambda)w\rangle_\Ltr\big| +2\nu\tilde{\al}\big|\langle w',r^{2\al+1}(r^{-2}-\lambda) w\rangle_\Ltr\big| + 2\nu \big|\braket{w', r^{2\al-1}w}_\Ltr\big|\\
&\leq  \bigl(\|r^{\al+1}F\|_\Ltr+2\nu\tilde{\al}\|r^{\al}w'\|_\Ltr\bigr)\|r^{\al+1}(r^{-2}-\lambda)w\|_\Ltr +2\nu\|r^\al w'\|_\Ltr\|r^{\al-1} w\|_\Ltr.
\end{align*}
Observing from \eqref{ineq 3.3} that
\begin{align*}
    \nu^{\f23} |kB|^{\f13}\|r^\al w'\|_{\Ltr}+\nu^\f13|kB|^\f23  \|r^{\al-1}w\|_{\Ltr} \leq C \|r^{\al+1}F\|_{\Ltr}.
\end{align*}
Then
we get, by applying  Young's inequality and $(\frac{\nu}{|kB|})^\f13\ll\tilde{\al}^{-1},$ that
\begin{align*}
|kB|\|r^{\al+1}(r^{-2}-\lambda)w\|_\Ltr^2
\lesssim & \f{1}{|kB|}\|r^{\al+1}F\|_\Ltr^2+\f{\nu^2\tilde{\al}^2}{|kB|}\|r^{\al} w'\|_\Ltr^2+\nu\|r^\al w'\|_\Ltr\|r^{\al-1} w\|_\Ltr\\
\lesssim &\Big(\f{1}{|kB|} +\f{\nu^2\tilde{\al}^2}{|kB|} \nu^{-\f43}|kB|^{-\f23}
\Big)\|r^{\al+1}F\|_\Ltr^2
\lesssim\f{1}{|kB|}\|r^{\al+1}F\|_\Ltr^2,
\end{align*}
which together with \eqref{S3eq4b} ensures \eqref{ineq 3.3}.
This completes the proof of Proposition \ref{prop 3.1}.
\end{proof}

Below we present the resolvent estimate of $\varphi$:

\begin{proposition}\label{prop L2-phi resolvent estimate}
{\sl  Let $k\in \Z\backslash\{0\}, \varepsilon \in (0,2) $, $\lambda\in\mathbb{R}$, there exists constant $C_\varepsilon>0$ independent of $\nu,k,B,\lambda$, so that
\begin{align}\label{ineq L2-phi resolvent estimate}
    \mu_k^{\f12} |kB|^\f12 |k|^\f12 \Bigl( \Big\|\frac{\varphi'}{r^{1-\varepsilon}}\Big\|_\Ltr + |k|\Big\|\frac{\varphi}{r^{2-\varepsilon}}\Big\|_\Ltr \Big) \leq C_\varepsilon\|r^{2+\varepsilon}F\|_\Ltr.
\end{align}}
\end{proposition}

\begin{proof}
If $\lambda\leq 0$, it's easy to observe from  Proposition \ref{prop 3.1} that for $|\al|\leq 3$
\begin{equation*}
    \|r^{\al-1}w\|_\Ltr \leq \|r^{\al+1}(r^{-2}-\lambda)w\|_\Ltr \leq \frac{C}{|kB|} \|r^{\al+1}F\|_\Ltr,
\end{equation*}
which together with Lemma \ref{basic properties of ck} ensures
\begin{equation*}\begin{split}
 \Big\|\frac{\varphi'}{r^{1-\varepsilon}}\Big\|_\Ltr + |k|\Big\|\frac{\varphi}{r^{2-\varepsilon}}\Big\|_\Ltr &\lesssim_\varepsilon |k|^{-1}\|r^\varepsilon w\|_\Ltr \lesssim_\varepsilon |k|^{-1} \min\big(|kB|^{-1},\mu_k^{-1}\big)\|r^{2+\varepsilon}F\|_\Ltr \\
 &\lesssim_\varepsilon \mu_k^{-\f12}|kB|^{-\f12} |k |^{-1}\|r^{2+\varepsilon}F\|_\Ltr.
\end{split}\end{equation*}

If $\lambda>0$ and $\nu k^2 \geq |B|$, we deduce from Lemma \ref{basic properties of ck} that
\begin{equation*}\begin{split}
 \Big\|\frac{\varphi'}{r^{1-\varepsilon}}\Big\|_\Ltr + |k|\Big\|\frac{\varphi}{r^{2-\varepsilon}}\Big\|_\Ltr &\lesssim_\varepsilon |k|^{-1}\|r^\varepsilon w\|_\Ltr \lesssim_\varepsilon |k|^{-1} \mu_k^{-1}\|r^{2+\varepsilon}F\|_\Ltr \\
 &\lesssim_\varepsilon \mu_k^{-\f12}(\nu k^2)^{-\f12}|k |^{-1}\|r^{2+\varepsilon}F\|_\Ltr\\
 &\lesssim_\varepsilon \mu_k^{-\f12}|kB|^{-\f12} |k |^{-\f12}\|r^{2+\varepsilon}F\|_\Ltr.
\end{split}\end{equation*}

Therefore  it remains to consider \eqref{ineq L2-phi resolvent estimate} for the case: $\lambda>0$ and $\nu k^2 \leq |B|$ (i.e., $\mu_k=\nu^\f13|kB|^\f23$). Let us denote $\tilde{\delta}{\eqdefa}(\frac{\nu}{|kB|})^\f13\ll1 $, and
\begin{equation}\begin{split}\label{E Em Eo}
   & E{\eqdefa} \big\{r\geq 1: \ |1-\sqrt{\lambda}r|\leq \tilde{\delta}  \big\},\\
    &E_m{\eqdefa} \big\{r\geq 1: \ \tilde{\delta}\leq |1-\sqrt{\lambda}r|\leq \f12   \big\},\\
    & E_o{\eqdefa} \big\{r\geq 1: \ \f12 \leq |1-\sqrt{\lambda}r| \big\},
\end{split}\end{equation}
and  split $w= w(\mathbbm{1}_{E}+\mathbbm{1}_{E_m}+\mathbbm{1}_{E_o})$.
Then we get, by applying  Lemma \ref{basic properties of ck} that
\beno
\Big\|\frac{\varphi'}{r^{1-\varepsilon}}\Big\|_\Ltr + |k|\Big\|\frac{\varphi}{r^{2-\varepsilon}}\Big\|_\Ltr \lesssim_\varepsilon |k|^{-\f12} \big(\|r^{\varepsilon-\f12 }w\mathbbm{1}_E\|_\Llr+\|r^{\varepsilon-\f12 }w\mathbbm{1}_{E_m}\|_\Llr\big) + |k|^{-1}\|r^{\varepsilon} w\mathbbm{1}_{E_o}\|_\Ltr.
\eeno
It follows from  Lemmas \ref{Appendix A1-2} and \ref{Appendix A1-3} that
\begin{align*}
\|r^{\varepsilon-\f12 }w\mathbbm{1}_E\|_\Llr+\|r^{\varepsilon-\f12 }w\mathbbm{1}_{E_m}\|_\Llr
\lesssim &\|r^{-\f12}\mathbbm{1}_E\|_\Ltr\|r^{\varepsilon}w\|_\Ltr+ \Big\|r^{-\f12}\frac{\mathbbm{1}_{E_m}}{1-\lambda r^2}\Big\|_\Ltr \|r^{\varepsilon} (1-\lambda r^2) w\|_\Ltr\\
\lesssim &  \tilde{\delta}^\f12 \|r^{\varepsilon}w\|_\Ltr+ \tilde{\delta}^{-\f12}  \|r^{\varepsilon}  (1-\lambda r^2) w\|_\Ltr,
\end{align*}
 from which and Proposition
\ref{prop 3.1}, we infer
\begin{align*}
 & \Big\|\frac{\varphi'}{r^{1-\varepsilon}}\Big\|_\Ltr + |k|\Big\|\frac{\varphi}{r^{2-\varepsilon}}\Big\|_\Ltr\\
&\lesssim_\varepsilon  |k|^{-\f12} \Bigl(  \big(\frac{\nu}{|kB|}\big)^{\f16} \nu^{-\f13}|kB|^{-\f23} +\big(\frac{\nu}{|kB|}\big)^{-\f16} |kB|^{-1} +|k|^{-\f12} |kB|^{-1} \Bigr) \|r^{2+\varepsilon} F\|_\Ltr   \\
&\lesssim_\varepsilon |k|^{-\f12} \nu^{-\f16}|kB|^{-\f56}\|r^{2+\varepsilon} F\|_\Ltr = |k|^{-\f12} \mu_k^{-\f12}|kB|^{-\f12}\|r^{2+\varepsilon} F\|_\Ltr.
\end{align*}
This completes the proof of Proposition \ref{ineq L2-phi resolvent estimate}.
\end{proof}

To proceed further, we define sublinear operator $\cQ_{a,b}$, which will be frequently used in this paper, via
\begin{equation}\label{definition of Tab}
    \cQ_{a,b}[f](r) {\eqdefa}  \int_1^\oo (rs)^{-\f12} \min\Big(\big(\frac{r}{s}\big)^a, \big(\frac{s}{r}\big)^b\Big) |f|(s) ds.
\end{equation}
\begin{lemma}\label{lemma A.4}
{\sl Let $ \cQ_{a,b}[f]$ be defined by \eqref{definition of Tab}. Then we have
\begin{subequations} \label{S3eq5}
\begin{gather}
\label{ineq estimate of Tab, L2}
\|\cQ_{a,b}[f]\|_\Ltr \leq \Big(\frac{1}{a}+\frac{1}{b}\Big) \|f\|_\Ltr\quad \mbox{if}\ a,b> 0,\\
\label{ineq estimate of Tab, L1}
    \|r^\f12\cQ_{a,b}[f]\|_\Ltr \leq \Big(\frac{1}{2a+1}+\frac{1}{2b-1}\Big)^\f12 \|f\|_\Llr \quad \mbox{if}\ a> -\f12\andf b> \f12.
\end{gather}
\end{subequations}}
\end{lemma}
\begin{proof} In view of  (\ref{definition of Tab}), we  get, by using changes of variables:
$\tilde{r}= \log r$, $\tilde{s}=\log s$, that
\begin{equation*}
\begin{split}
 e^{\frac{\tilde{r}}{2}} \cQ_{a,b}[f](e^{\tilde{r}}) = &\int_{0}^{+\oo}
 \min\big(e^{a(\tilde{r}-\tilde{s})}, e^{b(\tilde{s}-\tilde{r})}\big)
 e^{\frac{\tilde{s}}{2}} |f|(e^{\tilde{s}}) \, d\tilde{s} ,\\
 e^{\tilde{r}} \cQ_{a,b}[f](e^{\tilde{r}}) =& \int_{0}^{+\oo} \min\big(e^{(a+\f12)(\tilde{r}-\tilde{s})}, e^{(b-\f12)(\tilde{s}-\tilde{r})}\big) e^{\tilde{s}} |f|(e^{\tilde{s}}) \, d\tilde{s}.
 \end{split}
\end{equation*}
Applying Young Inequality gives rise to
\begin{equation*}
\|\cQ_{a,b}[f]\|_\Ltr = \|e^{\frac{\tilde{r}}{2}} \cQ_{a,b}[f](e^{\tilde{r}})\|_{L^2(\R^+; d\tilde{r})} \leq \Big(\frac{1}{a}+\frac{1}{b}\Big) \|e^{\frac{\tilde{s}}{2}} |f|(e^{\tilde{s}})\|_{L^2(\R^+; d\tilde{s})} \leq \Big(\frac{1}{a}+\frac{1}{b}\Big) \|f\|_\Ltr ,
\end{equation*}
and
\begin{equation*}
\begin{split}
 \|r^\f12\cQ_{a,b}[f]\|_\Ltr &= \|e^{\tilde{r}} \cQ_{a,b}[f](e^{\tilde{r}})\|_{L^2(\R^+; d\tilde{r})} \leq \Big(\frac{1}{2a+1}+\frac{1}{2b-1}\Big)^\f12 \|e^{\tilde{s}} |f|(e^{\tilde{s}})\|_{L^1(\R^+; d\tilde{s})} \\
 &\leq \Big(\frac{1}{2a+1}+\frac{1}{2b-1}\Big)^\f12 \|f\|_\Llr.
\end{split}
\end{equation*}
This completes the proof of \eqref{S3eq5}.
\end{proof}

\begin{proposition}\label{resolvent estimate 4}
  {\sl Let $k\in \Z\backslash\{0\}, \varepsilon \in (0,2) $, $\lambda\in\mathbb{R}$, there exists constant $C_\varepsilon>0$ independent of $\nu,k,B,\lambda$, so that
\begin{align}\label{ineq H-1-phi resolvent estimate}
    \nu^\frac{1}{2} |kB|^\f12 |k|^\f12\Bigl( \Big\|\frac{\varphi'}{r^{1-\varepsilon}}\Big\|_\Ltr + |k|\Big\|\frac{\varphi}{r^{2-\varepsilon}}\Big\|_\Ltr \Bigr) \leq C_\varepsilon \|r^{1+\varepsilon}F\|_\hkf.
\end{align}}
\end{proposition}

\begin{proof}
 In the case of $\la \leq 0$, it's easy to observe from \eqref{vorticity eqn} that for  $|\al|\leq 3$,
\begin{align*}
|kB|\big|\langle(r^{-2}-\lambda)w, r^{2\al} w\rangle_\Ltr\big| \leq \big|\Im\langle F, r^{2\al}w\rangle_\Ltr\big|+  \nu \big|\Im\langle \pa_r w, \pa_r(r^{2\al})w\rangle_\Ltr\big|,
\end{align*}
which implies in this case that
\begin{align*}
|kB|\|r^{\al-1}w\|_{\Ltr}^2 \lesssim \big|\langle F,r^{2\al} w\rangle_\Ltr\big| +\nu \|r^{\al}\pa_rw \|_\Ltr \|r^{\al-1}w \|_\Ltr.
\end{align*}
Applying Young's inequality and Proposition \ref{prop 3.1} gives
\begin{equation*}
\begin{split}
   |kB| \|r^{\al-1}w\|_{\Ltr}^2 &\lesssim  \|r^{\al}F\|_\hkf \|r^{\al}w\|_\hkz + \nu\|r^{\al}\pa_rw \|_\Ltr \|r^{\al-1}w \|_\Ltr \\
   &\lesssim  (\nu^{-1}+\nu^{-\f12}\mu_k^{-\f12})\|r^{\al}F\|_\hkf^2\lesssim \nu^{-1} \|r^{\al}F\|_\hkf^2,
   \end{split}
\end{equation*}
 from which and Lemma \ref{basic properties of ck}, we infer
\begin{equation}\label{lmabda trival-0, H1}
    \begin{split}
        \Big\|\frac{\varphi'}{r^{1-\varepsilon}}\Big\|_\Ltr + |k|\Big\|\frac{\varphi}{r^{2-\varepsilon}}\Big\|_\Ltr \lesssim_\varepsilon |k|^{-1}  \|r^{\varepsilon} w \|_\Ltr \lesssim_\varepsilon |k|^{-1}\nu^{-\f12}|kB|^{-\f12} \|r^{1+\varepsilon}F \|_\hkf.
    \end{split}
\end{equation}

In the  case when $\la >0$ and $\nu k^2 \geq |B|$, we deduce from   Proposition \ref{prop 3.1} and Lemma \ref{basic estimate of K} that
\begin{equation}\label{S3eq6}
    \begin{split}
 \Big\|\frac{\varphi'}{r^{1-\varepsilon}}\Big\|_\Ltr + |k|\Big\|\frac{\varphi}{r^{2-\varepsilon}}\Big\|_\Ltr &\lesssim_\varepsilon |k|^{-1}  \|r^{\varepsilon} w \|_\Ltr \\
 &\lesssim_\varepsilon |k|^{-1} \nu^{-\f23} |kB|^{-\f13} \|r^{1+\varepsilon} F\|_\hkf \\
 &\lesssim_\varepsilon \nu^{-\f12} |kB|^{-\f12} |k|^{-\f12} \|r^{1+\varepsilon} F\|_\hkf.
    \end{split}
\end{equation}

We claim that if $\la >0$ and   $\nu k^2 \leq |B|,$ then
\begin{subequations} \label{S3eq9}
\begin{gather}
\label{S3eq9a}
\nu^\frac{1}{2} |kB|^\f12 |k|^\f32 \Big\|\frac{\varphi}{r^{2-\varepsilon}}\Big\|_\Ltr  \lesssim_\varepsilon \|r^{1+\varepsilon}F\|_\hkf,\\
 \nu^\frac{1}{2} |kB|^\f12 |k|^\f12 \Big\|\frac{\varphi'}{r^{1-\varepsilon}}\Big\|_\Ltr \leq C_\varepsilon \|r^{1+\varepsilon}F\|_\hkf. \label{S3eq9b}
\end{gather}
\end{subequations}

By summarizing \eqref{lmabda trival-0, H1}, \eqref{S3eq6} and \eqref{S3eq9}, we complete the proof of \eqref{ineq H-1-phi resolvent estimate}.
\end{proof}

It remains to prove \eqref{S3eq9}.

\begin{proof}[Proof of \eqref{S3eq9}]
In the  case when $\la >0$ and   $\nu k^2 \leq |B|,$
 we denote $r_0\eqdefa \lambda^{-\f12}$, $\tilde{\delta}\eqdefa (\frac{\nu}{|kB|})^\f13\ll 1$ and
  $\delta\eqdefa \tilde{\delta}r_0.$ Let $E, E_m, E_o$ be defined as in \eqref{E Em Eo}.  Since $\nu k^2 \leq |B|$, we have $|k|\tdelta \leq 1 $.  Let $\eta (r) \in \mathcal{C}_0^\oo(\R)$ be  supported in $(-2,2)$ and  equal to $1$ for $s \in [-1, 1],$ and we denote $\eta(\frac{r-r_0}{\delta})$ by $\eta_\delta.$
Then we write
\begin{equation}\label{S3eq7}
\begin{split}
-\frac{\varphi}{r^{2-\varepsilon}} &= \int_1^\oo  r^{-2+\varepsilon} K_k(r,s) w(s) ds\\
& = \int_1^\oo  r^{-2+\varepsilon} K_k(r,s)  \frac{1-\eta_\delta(s)}{s^{-2}-\lambda} (s^{-2}-\lambda)w(s) ds \\
&\quad +  \int_1^\oo  r^{-2+\varepsilon} K_k(r,s) w(s)\eta_\delta(s) ds\eqdefa  J_1+J_2,
\end{split}
\end{equation}
and
\begin{equation}\label{S3eq8}
    \begin{split}
-\frac{\varphi'}{r^{1-\varepsilon}} &= \int_1^\oo  r^{-1+\varepsilon} \pa_r K_k(r,s) w(s) ds\\
& =\int_1^\oo  r^{-1+\varepsilon} \pa_rK_k(r,s)  \frac{1-\eta_\delta(s)}{s^{-2}-\lambda} (s^{-2}-\lambda)w ds \\
&\quad +  \int_1^\oo  r^{-1+\varepsilon} \pa_rK_k(r,s) w(s)\eta_\delta(s) ds \eqdefa J_3+J_4.
    \end{split}
\end{equation}

By density argument, we only consider the case when $(s^{-2}-\lambda)w(s)=\pa_s F_1(s) + \frac{k}{s}F_2(s)$ with $F_1, F_2 \in \mathcal{C}^\oo_c([1,\oo))$.
Then  by virtue of \eqref{equivalent definition of weighted Hkf}, we have
\begin{equation}\label{eq 3.17}
    \|r^{1+\varepsilon} (s^{-2}-\lambda)w(s)\|_\hkf \approx  \inf_{F_1,F_2} \bigl(\|r^{1+\varepsilon }F_1\|_\Ltr + \|r^{1+\varepsilon} F_2\|_\Ltr  \bigr).
\end{equation}  Below we divide the proof into the following two steps:

\no{\bf Step 1.} The proof of  \eqref{S3eq9a}.

Due to $K(r,1)=0,$
we get, by using  integration by parts, that
\begin{equation*}
\begin{split}
    J_1 &= -  \int_1^\oo  r^{-2+\varepsilon} \pa_s \Bigl(K_k(r,s)  \frac{1-\eta_\delta(s)}{s^{-2}-\lambda}\Bigr) F_1(s) ds +   \int_1^\oo  r^{-2+\varepsilon} K_k(r,s)  \frac{1-\eta_\delta(s)}{s^{-2}-\lambda} \frac{k}{s}F_2(s)  ds \\
    &=:J_{11}+J_{12}+J_{13}+J_{14},
\end{split}
\end{equation*}
where
\begin{align*}
    &J_{11}{\eqdefa}- \int_1^\oo  r^{-2+\varepsilon} \pa_s K_k(r,s)  \frac{1-\eta_\delta(s)}{s^{-2}-\lambda} F_1(s) ds , \\ &J_{12}{\eqdefa}\delta^{-1}  \int_1^\oo  r^{-2+\varepsilon}  K_k(r,s)  \frac{\eta'_\delta(s)}{s^{-2}-\lambda} F_1(s) ds, \\
    &J_{13}{\eqdefa} -2 \int_1^\oo  r^{-2+\varepsilon}  K_k(r,s)  \frac{1-\eta_\delta(s)}{s^3(s^{-2}-\lambda)^2} F_1(s) ds, \\ &J_{14}{\eqdefa}\int_1^\oo  r^{-2+\varepsilon} K_k(r,s)  \frac{k(1-\eta_\delta(s))}{s(s^{-2}-\lambda)} F_2(s)  ds .
\end{align*}

We first get, by applying \eqref{estimate of cK} and $\frac{\mathbbm{1}_{E_o}}{|1-\lambda r^2|} \lesssim 1$, that
\begin{align*}
    |J_{11}|  &\lesssim  \int_1^\oo r^{-2+\varepsilon}\big(\frac{r}{s}\big)^{\f12} \min\big(\frac{r}{s},\frac{s}{r}\big)^{|k|}  s^{2}\frac{\mathbbm{1}_{E_m}+\mathbbm{1}_{E_o}}{|1-\lambda s^2|} |F_1(s)| ds\\
&\leq  r^{\f12} \int_1^\oo(rs)^{-\f12}   \big(\frac{s}{r}\big)^{\f32-\varepsilon} \min\big(\frac{r}{s},\frac{s}{r}\big)^{|k|}  \frac{\mathbbm{1}_{E_m}}{|1-\lambda s^2|}s^{\f12+\varepsilon} |F_1(s)| ds  \\
&\quad +   \int_1^\oo(rs)^{-\f12}   \big(\frac{s}{r}\big)^{1-\varepsilon} \min\big(\frac{r}{s},\frac{s}{r}\big)^{|k|}  \frac{\mathbbm{1}_{E_o}}{|1-\lambda s^2|}s^{1+\varepsilon} |F_1(s)| ds  \\
&\lesssim  r^{\f12} \cQ_{|k|-\frac{3}{2}+\varepsilon, |k|+\f32-\varepsilon} \Big[r^{\f12+\varepsilon}\frac{\mathbbm{1}_{E_m}}{1-\lambda r^2} F_1 \Big]+   \cQ_{|k|-1+\varepsilon, |k|+1-\varepsilon} [r^{1+\varepsilon} F_1 ],
\end{align*}
from which, Lemmas \ref{lemma A.4} and \ref{Appendix A1-3} and $|k|\leq \tilde{\delta}^{-1}$, we infer
\begin{equation*}
\begin{split}\label{estimate of J_11}
\|J_{11}\|_\Ltr &\lesssim_\varepsilon
  |k|^{-\f12}  \Big\|r^{-\f12} \frac{\mathbbm{1}_{E_m}}{1-\lambda r^2} \Big\|_\Ltr \|r^{1+\varepsilon}F_1\|_\Ltr+ |k|^{-1} \|r^{1+\varepsilon}F_1\|_\Ltr \\
&\lesssim_\varepsilon  |k|^{-\f12} \tilde{\delta}^{-\f12} \|r^{1+\varepsilon}F_1\|_\Ltr \lesssim_\varepsilon |k|^{-\f32} \tilde{\delta}^{-\f32} \|r^{1+\varepsilon}F_1\|_\Ltr.
\end{split}
\end{equation*}
Similarly, we get, by using Lemma \ref{basic properties of ck}, that
\begin{align*}
 \sum_{i=2}^4 \|J_{1i}\|_\Ltr
  &\lesssim \delta^{-1}\Big\|r^{-2+\varepsilon} \cK\Big[\frac{\eta'_\delta}{s^{-2}-\lambda} F_1\Big]\Big\|_\Ltr + \Big\|r^{-2+\varepsilon} \cK\Big[\frac{1-\eta_\delta}{s^3(s^{-2}-\lambda)^2} F_1\Big]\Big\|_\Ltr \\
    &\quad+ \Big\|r^{-2+\varepsilon} \cK\Big[ \frac{k(1-\eta_\delta)}{s(s^{-2}-\lambda)} F_2\Big]\Big\|_\Ltr\\
     &\lesssim \delta^{-1}|k|^{-\f32}\Big\|r^{-\f12+\varepsilon} \frac{\eta'_\delta}{r^{-2}-\lambda} F_1\Big\|_\Llr \\
     &\quad + |k|^{-\f32} \Big\|r^{-\f12+\varepsilon} \frac{\mathbbm{1}_{E_m}}{r^3(r^{-2}-\lambda)^2} F_1\Big\|_\Llr + |k|^{-2} \Big\|r^{\varepsilon} \frac{\mathbbm{1}_{E_o}}{r^3(r^{-2}-\lambda)^2} F_1\Big\|_\Ltr \\
   &\quad + |k|^{-\f32} \Big\|r^{-\f12+\varepsilon} \frac{k\mathbbm{1}_{E_m}}{r(r^{-2}-\lambda)} F_2\Big\|_\Llr + |k|^{-2} \Big\|r^{\varepsilon} \frac{k\mathbbm{1}_{E_o}}{r(r^{-2}-\lambda)} F_2\Big\|_\Ltr,
   \end{align*}
   which together with  Lemma \ref{Appendix A1-3} and $|k|\leq \tilde{\delta}^{-1}$ ensures that
   \begin{align*}
 \sum_{i=2}^4 \|J_{1i}\|_\Ltr
    &\lesssim_\varepsilon \delta^{-1}|k|^{-\f32} \Big\|r^{\f12}\frac{\mathbbm{1}_{\delta \leq |r-r_0| \leq 2\delta}}{1-\lambda r^2} \Big\|_\Ltr \|r^{1+\varepsilon} F_1\|_\Ltr  \\
    &\quad + |k|^{-\f32} \Big\|r^{-\f12} \frac{\mathbbm{1}_{E_m}}{(1-\lambda r^2)^2} \Big\|_\Ltr  \|r^{1+\varepsilon} F_1\|_\Ltr + |k|^{-2}  \Big\|r^{1+\varepsilon} \frac{\mathbbm{1}_{E_o}}{(1-\lambda r^2)^2}F_1\Big\|_\Ltr\\
&\quad +  |k|^{-\f12} \Big\|r^{-\f12} \frac{\mathbbm{1}_{E_m}}{1-\lambda r^2} \Big\|_\Ltr  \|r^{1+\varepsilon} F_2\|_\Ltr +|k|^{-1}  \Big\|r^{1+\varepsilon}\frac{\mathbbm{1}_{E_o}}{1-\lambda r^2} F_2\Big\|_\Ltr\\
&\lesssim_\varepsilon \big(\delta^{-1}|k|^{-\f32} r_0 \tilde{\delta}^{-\f12}  +|k|^{-\f32} \tilde{\delta}^{-\f32} +|k|^{-2}\big)
\|r^{1+\varepsilon} F_1\|_\Ltr \\
&\quad+ \big(|k|^{-\f12} \tilde{\delta}^{-\f12}+|k|^{-1}\big)\|r^{1+\varepsilon} F_2\|_\Ltr\\
&\lesssim_\varepsilon |k|^{-\f32} \tilde{\delta}^{-\f32} \bigl( \|r^{1+\varepsilon} F_1\|_\Ltr+ \|r^{1+\varepsilon} F_2\|_\Ltr   \bigr).
\end{align*}

By summarizing the above estimates, we obtain
\begin{equation*}\begin{split}
    \|J_1\|_\Ltr \lesssim_\varepsilon \nu^{-\f12}|kB|^{\f12}|k|^{-\f32} \bigl(  \|r^{1+\varepsilon }F_1\|_\Ltr + \|r^{1+\varepsilon} F_2\|_\Ltr \bigr),
    \end{split}
\end{equation*}
from which, \eqref{eq 3.17} and Proposition \ref{prop 3.1}, we infer
\begin{equation}\label{estimate of J_1}
    \|J_1\|_\Ltr\lesssim_\varepsilon \nu^{-\f12}|kB|^{\f12}|k|^{-\f32} \|r^{1+\varepsilon}(r^{-2}-\lambda) w\|_\hkf \lesssim_\varepsilon \nu^{-\f12}|kB|^{-\f12}|k|^{-\f32} \|r^{1+\varepsilon}F\|_\hkf .
\end{equation}

While we get, by using Proposition \ref{prop 3.1} and Lemmas  \ref{basic properties of ck} and \ref{Appendix A1-2}, that
\begin{equation}\label{estimate of J_2}
    \begin{split}
\|J_2\|_\Ltr &= \Big\|\frac{\cK[w\eta_\delta]}{r^{2-\varepsilon}}\Big\|_\Ltr \lesssim_\varepsilon |k|^{-\f32} \|r^{\varepsilon-\f12} w\eta_\delta\|_\Llr \\
 &\lesssim_\varepsilon |k|^{-\f32} \|r^{-\f12} \eta_\delta\|_\Ltr \|r^\varepsilon w\|_\Ltr  \lesssim_\varepsilon |k|^{-\f32}  \tilde{\delta}^{\f12} \nu^{-\f23} |kB|^{-\f13} \|r^{1+\varepsilon}F\|_\hkf \\
&\lesssim_\varepsilon \nu^{-\f12}|kB|^{-\f12} |k|^{-\f32}\|r^{1+\varepsilon} F\|_\hkf .
    \end{split}
\end{equation}

In view of \eqref{S3eq7}, by summing up
 \eqref{estimate of J_1} and \eqref{estimate of J_2}, we achieve \eqref{S3eq9a}.

\no{\bf Step 2.} The proof of  \eqref{S3eq9b}.

Similar to the estimation of $J_2$, we get, by using Proposition \ref{prop 3.1} and Lemmas   \ref{basic properties of ck} and \ref{Appendix A1-2}, that
\begin{equation}\label{estimate of J_4}
    \begin{split}
\|J_4\|_\Ltr &= \Big\|\frac{\cK'[w\eta_\delta]}{r^{1-\varepsilon}}\Big\|_\Ltr \lesssim_\varepsilon |k|^{-\f12} \|r^{\varepsilon-\f12} w\eta_\delta\|_\Llr \\
 &\lesssim_\varepsilon |k|^{-\f12} \|r^{-\f12} \eta_\delta\|_\Ltr \|r^\varepsilon w\|_\Ltr  \lesssim_\varepsilon |k|^{-\f12}  \tilde{\delta}^{\f12} \nu^{-\f23} |kB|^{-\f13} \|r^{1+\varepsilon}F\|_\hkf \\
&\lesssim_\varepsilon \nu^{-\f12}|kB|^{-\f12}|k|^{-\f12} \|r^{1+\varepsilon} F\|_\hkf .
    \end{split}
\end{equation}

To handle $J_3$ given by \eqref{S3eq8}, we first get,  by using integration by parts and  \eqref{1.17}, that
\begin{equation*}\label{J_3}
\begin{split}
   J_3(r) &=\int_1^\oo  r^{-1+\varepsilon}\pa_r K_k(r,s)  \frac{1-\eta_\delta(s)}{s^{-2}-\lambda} \pa_s F_1(s)  ds\\
    &\quad +   \int_1^\oo  r^{-1+\varepsilon}\pa_r K_k(r,s)  \frac{k(1-\eta_\delta(s))}{s(s^{-2}-\lambda)} F_2(s)  ds \\
    &\eqdefa J_{31}+J_{32}+J_{33}+J_{34} + J_{35},
\end{split}
\end{equation*}
where we define
\begin{align*}
   & J_{31}{\eqdefa}-  \int_1^\oo  r^{-1+\varepsilon} \widetilde{K}_k(r,s)  \frac{1-\eta_\delta(s)}{s^{-2}-\lambda} F_1(s) ds,\\
   &J_{32}{\eqdefa}\delta^{-1}\int_1^\oo  r^{-1+\varepsilon} \pa_r K_k(r,s)  \frac{\eta'_\delta(s)}{s^{-2}-\lambda} F_1(s) ds, \\
    & J_{33}{\eqdefa} -2  \int_1^\oo  r^{-1+\varepsilon}  \pa_r K_k(r,s)  \frac{1-\eta_\delta(s)}{s^3(s^{-2}-\lambda)^2} F_1(s) ds,\\
    & J_{34}{\eqdefa}\int_1^\oo  r^{-1+\varepsilon} \pa_r K_k(r,s)  \frac{1-\eta_\delta(s)}{s^{-2}-\lambda} \frac{k}{s}F_2(s)  ds ,\\
    &J_{35}{\eqdefa}-r^{1+\varepsilon} \frac{1-\eta_\delta(r)}{1-\lambda r^2} F_1(r).
\end{align*}

By virtue of \eqref{estimate of cK}, we get, by applying Lemma \ref{lemma A.4}, that
\begin{align*}
\|J_{31}\|_\Ltr &\lesssim |k|\Big\|r^{\f12} \cQ_{|k|-\frac{3}{2}+\varepsilon, |k|+\f32-\varepsilon} \Big[r^{\f12+\varepsilon}\frac{\mathbbm{1}_{E_m}}{1-\lambda r^2} F_1 \Big]\Big\|_\Ltr\\
&\quad +|k|\Big\| \cQ_{|k|-1+\varepsilon, |k|+1-\varepsilon} \Big[r^{1+\varepsilon} \frac{\mathbbm{1}_{E_o}}{1-\lambda r^2} F_1 \Big]\Big\|_\Ltr \\
&\lesssim_\varepsilon |k|^\f12 \Big\|r^{-\f12}\frac{\mathbbm{1}_{E_m}}{1-\lambda r^2} \Big\|_\Ltr \|r^{1+\varepsilon}F_1\|_\Ltr + \Big\|r^{1+\varepsilon}\frac{\mathbbm{1}_{E_o}}{1-\lambda r^2} F_1\Big\|_\Ltr,
\end{align*}
from which,  $|k|\leq \tilde{\delta}^{-1}$ and Lemma \ref{Appendix A1-3}, we infer
\begin{align*}
\|J_{31}\|_\Ltr \lesssim_\varepsilon |k|^{-\f12} \tilde{\delta}^{-\f32} \|r^{1+\varepsilon} F_1\|_\Ltr \lesssim_\varepsilon \nu^{-\f12}|kB|^{-\f12}  |k|^{-\f12}  \|r^{1+\varepsilon} F_1\|_\Ltr.
\end{align*}

Whereas we get, by a similar derivation of the estimate of $\sum_{i=2}^4 \|J_{1i}\|_\Ltr ,$ that
\begin{align*}
   \sum_{i=2}^4 \|J_{3i}\|_\Ltr
  &\lesssim \delta^{-1}\Big\|r^{-1+\varepsilon}\cK'\Big[\frac{\eta'_\delta}{s^{-2}-\lambda}F_1\Big]\Big\|_\Ltr + \Big\|r^{-1+\varepsilon}\cK'\Big[\frac{1-\eta_\delta}{s^3(s^{-2}-\lambda)^2}F_1\Big]\Big\|_\Ltr \\
   &\quad+  \Big\|r^{-1+\varepsilon}\cK'\Big[\frac{k(1-\eta_\delta)}{s(s^{-2}-\lambda)}F_2\Big]\Big\|_\Ltr\\
     &\lesssim \delta^{-1}|k|^{-\f12}\Big\|r^{-\f12+\varepsilon} \frac{\eta'_\delta}{r^{-2}-\lambda} F_1\Big\|_\Llr \\
     &\quad + |k|^{-\f12} \Big\|r^{-\f12+\varepsilon} \frac{\mathbbm{1}_{E_m}}{r^3(r^{-2}-\lambda)^2} F_1\Big\|_\Llr + |k|^{-1} \Big\|r^{\varepsilon} \frac{\mathbbm{1}_{E_o}}{r^3(r^{-2}-\lambda)^2} F_1\Big\|_\Ltr \\
   &\quad + |k|^{-\f12} \Big\|r^{-\f12+\varepsilon} \frac{k\mathbbm{1}_{E_m}}{r(r^{-2}-\lambda)} F_2\Big\|_\Llr + |k|^{-1} \Big\|r^{\varepsilon} \frac{k\mathbbm{1}_{E_o}}{r(r^{-2}-\lambda)} F_2\Big\|_\Ltr \\
   &\lesssim_\varepsilon \delta^{-1}|k|^{-\f12} \Big\|r^{\f12} \frac{\mathbbm{1}_{\delta \leq |r-r_0| \leq 2\delta}}{1-\lambda r^2}\Big\|_\Ltr \|r^{1+\varepsilon}F_1\|_\Ltr \\
   &\quad + |k|^{-\f12}\Big\|r^{-\f12} \frac{\mathbbm{1}_{E_m}}{(1-\lambda r^2)^2} \Big\|_\Ltr  \|r^{1+\varepsilon} F_1\|_\Ltr + |k|^{-1}\Big\|\frac{\mathbbm{1}_{E_o}}{(1-\lambda r^2)^2} r^{1+\varepsilon}F_1\Big\|_\Ltr \\
   &\quad + |k|^{\f12}\Big\|r^{-\f12} \frac{\mathbbm{1}_{E_m}}{1-\lambda r^2} \Big\|_\Ltr  \|r^{1+\varepsilon} F_2\|_\Ltr + \Big\|\frac{\mathbbm{1}_{E_o}}{1-\lambda r^2} r^{1+\varepsilon}F_2\Big\|_\Ltr \\
   &\lesssim_\varepsilon  \bigl(\delta^{-1}|k|^{-\f12}r_0 \tilde{\delta}^{-\f12} + |k|^{-\f12}\tilde{\delta}^{-\f32} + |k|^{-1} \bigr) \|r^{1+\varepsilon}F_1\|_\Ltr + \bigl(|k|^\f12 \tilde{\delta}^{-\f12} +1 \bigr)\|r^{1+\varepsilon}F_2\|_\Ltr\\
   &\lesssim_\varepsilon |k|^{-\f12} \tilde{\delta}^{-\f32} \bigl(\|r^{1+\varepsilon}F_1\|_\Ltr  +\|r^{1+\varepsilon}F_2\|_\Ltr \bigr).
\end{align*}

Finally for $J_{35}$, we have
\begin{equation*}\begin{split}\label{estimate of J_35}
    \|J_{35}\|_\Ltr &\lesssim   \Big\| \frac{1-\eta_\delta}{1-\lambda r^2}\Big\|_\Lor \|r^{1+\varepsilon}F_1\|_\Ltr \lesssim \tilde{\delta}^{-1} \|r^{1+\varepsilon}F_1\|_\Ltr \lesssim |k|^{-\f12}\tilde{\delta}^{-\f32} \|r^{1+\varepsilon}F_1\|_\Ltr.
\end{split}
\end{equation*}

In view of \eqref{S3eq8},
by summarizing the above estimates, and then applying \eqref{eq 3.17} and Proposition \ref{prop 3.1}, we  achieve \eqref{S3eq9b}.
\end{proof}

\section{Resolvent estimates with enhanced dissipation weight}\label{Resolvent estimates with time weight}
In this section,
we establish pointwise enhanced dissipation for the following linearized equations with $k\neq0$:
\begin{subequations}\label{eq 5.1}
    \begin{gather}
\label{eq 5.1a} \pa_t w_k +\mathcal{T}_k w_k = F_k, \quad w_k(t=0)=w_k(0),\quad w_k|_{r=1,\oo}=0,\\
\label{eq 5.1b} \Big(\pa_r^2 -\frac{k^2-\f14}{r^2}\Big)\varphi_k = w_k, \quad \varphi_k|_{r=1,\oo}=0.
    \end{gather}
\end{subequations}
Since we shall deal with $L^2$ type estimates, for simplicity,
 we always set $w_k(0)=0$ and assume that $\varphi_k, w_k,F \in \mathcal{C}^\oo_c([0,\oo)\times [1,\oo)).$

 Recall that
$ \kappa_k {\eqdefa} \nu^\frac{1}{3}|kB|^\f23$ and $ \La_k(t,r){\eqdefa}\log\big(\hc r^2 + \kappa_k t\big),$
we denote
\beq \label{S5eq1} w_\star {\eqdefa} \La_k(t,r) w_k, \quad  F_\star {\eqdefa} \La_k(t,r) F_k, \quad \varphi_\star{\eqdefa}\La_k(t,r) \varphi_k.
 \eeq
 Then in view of \eqref{eq 5.1}, we find
\begin{align}\label{eq *}
    \left\{
    \begin{aligned}
        & \partial_t w_\star -\nu\Big(\partial_r^2 -\frac{k^2-\frac{1}{4}}{r^2}\Big)w_\star + \frac{ikB}{r^2} w_\star \\
       &\qquad = F_\star + \frac{\pa_t \La_k}{\La_k} w_\star -\nu \frac{\pa_r^2 \La_k}{\La_k} w_\star + 2\nu \Big(\frac{\pa_r\La_k}{\La_k}\Big)^2 w_\star -2\nu \frac{\pa_r\La_k}{\La_k} \pa_r w_\star  \eqdef F_\star + R^{w_\star},\\
        & \Big(\partial_r^2 -\frac{k^2-\frac{1}{4}}{r^2}\Big)\varphi_\star  = w_\star + \frac{\pa_r^2 \La_k}{\La_k} \varphi_\star - 2 \Big(\frac{\pa_r\La_k}{\La_k}\Big)^2 \varphi_\star +2 \frac{\pa_r\La_k}{\La_k} \pa_r \varphi_\star\eqdef w_\star + R^{\varphi_\star},\\
        & w_\star(0)=0,\quad  w_\star|_{r=1,+\oo}=\varphi_\star|_{r=1,+\oo}=0.
    \end{aligned}
    \right.
\end{align}

The key idea to handle the equations \eqref{eq *} is to view $R^{w_\star} $ and $R^{\varphi_\star}$ as  small perturbations and using the
resolvent estimates we obtained in Section \ref{Resolvent estimates without enhanced dissipation weight}.

\begin{proposition}\label{space-time resolvent L2-L2}
 {\sl  For any $k\in \Z\backslash\{0\} $, $\al \in[0,4]$, $T>0$, there exist constants $C>0$ and $\hat{C}$ sufficiently large, which are independent of $\nu,k,B,\al,T$, so that
\begin{equation}\label{L2-L2, space-time weighted resolvent}
\begin{split}
\nu^\f12 \mu_k^{\f12}\|r^\al \partial_r w_\star\|_{\LtT (\Ltr)} +\mu_k \|r^{\al-1} w_\star\|_{\LtT (\Ltr)} \leq C \|r^{\al+1}F_\star\|_{\LtT (\Ltr)}.
\end{split}
\end{equation}}
\end{proposition}

\begin{proof}
Without loss of generality, we set $T=\oo$ and define
\begin{equation*}
    \hat{f}(\lambda){\eqdefa}\int_0^\oo e^{-it\lambda} f(t) dt.
\end{equation*}
We may extend $f(t)=0$ for $t\leq 0$ in case $f(0)=0.$

Then we get, by taking Fourier transformation to the $w_\star$ equation in \eqref{eq *} with respect to the time variable, that
\begin{equation}
\begin{split}
      \label{4.5}  &-\nu\Big(\pa_r^2 -\frac{k^2-\f14}{r^2}\Big) \widehat{w}_\star(kB\lambda,r) + ikB(r^{-2}-\lambda) \widehat{w}_\star(kB\lambda,r) \\
      &\qquad \qquad \qquad \qquad \qquad \qquad \qquad \qquad= \hat{f}_\star(kB\lambda,r) + \widehat{R^{w_\star}}(kB\lambda,r),
\end{split}
\end{equation}
from which and  Proposition \ref{prop 3.1}, we infer
\begin{align*}
   & \nu \mu_k \|r^{\al} \partial_r \widehat{w}_\star(kB\lambda,r)\|_{\Ltr}^2 +\mu_k^2\|r^{\al-1}\widehat{w}_\star(kB\lambda,r) \|_{\Ltr}^2 \\
    &\qquad \qquad \qquad \qquad  \leq C \Big( \|r^{\al+1} \widehat{F}_\star(kB\lambda,r)\|_{\Ltr}^2  + \|r^{\al+1} \widehat{R^{w_\star}}(kB\lambda,r)\|_{\Ltr}^2 \Big).
\end{align*}
By integrating the above inequality over $\R$ with respect to $\lambda$ variable, and then using Plancherel's equality, we obtain
\begin{align}\label{4.5,a}
    \nu \mu_k \|r^{\al} \partial_r w_\star\|_{L^2_T (\Ltr)}^2 +\mu_k^2\|r^{\al-1}w_\star \|_{L^2_T (\Ltr)}^2  \leq C \Bigl( \|r^{\al+1} F_\star\|_{L^2_T (\Ltr)}^2  + \|r^{\al+1} R^{w_\star}\|_{L^2_T (\Ltr)}^2 \Bigr).
\end{align}
Yet it follows from Lemma \ref{Basic calculation for La_k} that
\begin{equation}\begin{split}\label{4.6,a}
   &\|r^{\al+1} R^{w_\star}\|_{L^2_T (\Ltr)}^2\\
   &\leq C\Bigl( \frac{\nu^\f13|kB|^\f23}{\hc\log\hc}  \|r^{\al-1}w_\star \|_{L^2_T (\Ltr)}  + \frac{\nu}{\log\hc}\|r^{\al} \partial_r w_\star\|_{L^2_T (\Ltr)} + \frac{\nu}{\log\hc} \|r^{\al-1}w_\star \|_{L^2_T (\Ltr)} \Bigr)^2\\
    &\leq \frac{C}{\log^2\hc} \Bigl(   \nu^2 \|r^{\al} \partial_r w_\star\|_{L^2_T (\Ltr)}^2 +(\nu^\f13 |kB|^\f23)^2\|r^{\al-1}w_\star \|_{L^2_T (\Ltr)}^2 \Bigr).
\end{split}\end{equation}
By inserting \eqref{4.6,a} into \eqref{4.5,a}, and   taking $\hc$ to be large enough, we arrive at \eqref{L2-L2, space-time weighted resolvent}.
\end{proof}

\begin{proposition}\label{space-time resolvent,H-1 to H1}
 {\sl For any $k\in \Z\backslash\{0\}$, $\al \in[0,4]$, $T>0$, there exist constants $C>0$ and $\hat{C}$ sufficiently large, which are independent of $\nu,k,B,\al,T$, so that
\begin{equation}\label{H-1 to H1, space-time weighted resolvent}
\begin{split}
\nu \|r^{\al}  w_\star\|_{\LtT (\hkz)} &+\nu^\f12 \mu_k^\f12\|r^{\al-1}w_\star \|_{\LtT (\Ltr)} \leq C \|r^{\al}F_\star\|_{\LtT (\hkf)}.
\end{split}
\end{equation}}
\end{proposition}
\begin{proof}
 Thanks to  (\ref{3.10}-\ref{3.11}) and Proposition \ref{prop 3.1}, we deduce from \eqref{4.5} that
  \begin{equation*}
  \begin{split}
      &\nu \|r^{\al}\widehat{w}_\star(kB\lambda,r)\|_{L^2_\lambda( \hkz)}^2 + \mu_k \|r^{\al-1}\widehat{w}_\star(kB\lambda,r)\|_{L^2_\lambda (\Ltr)}^2 \\
      &\leq C \int_\R \Big|\braket{r^{\al}\widehat{F}_\star, r^{\al}\widehat{w}_\star }_\Ltr\Big| (kB\lambda)d\lambda + C \int_\R \Big|\braket{r^{\al}\widehat{F}_\star, r^{\al}\tilde{\rho} \widehat{w}_\star }_\Ltr\Big| (kB\lambda) \mathbbm{1}_{(0,1)}\big(\lambda\big) d\lambda\\
      &\quad + C \int_\R \Big|\braket{r^{\al}\widehat{R^{w_\star}}, r^{\al}\widehat{w}_\star }_\Ltr\Big|(kB\lambda) d\lambda + C \int_\R \Big|\braket{r^{\al}\widehat{R^{w_\star}}, r^{\al}\tilde{\rho} \widehat{w}_\star }_\Ltr\Big|(kB\lambda) \mathbbm{1}_{(0,1)}\big(\lambda\big) d\lambda\\
      & \leq C \nu^{-1}\|r^\al \widehat{F}_\star(kB\lambda,r)\|_{L^2_\lambda(\hkf)}^2 \\
      &\quad + \frac{1}{4} \Bigl( \nu \| r^\al \widehat{w}_\star(kB\lambda,r)\|_{L^2_\lambda( \hkz)}^2 + \nu^\f13|kB|^\f23 \|r^{\al-1}\widehat{w}_\star(kB\lambda,r)\|_{L^2_\lambda (\Ltr)}^2 \Bigr)\\
      &\quad + C\|r^{\al+1}\widehat{R^{w_\star}}(kB\lambda,r)\|_{L^2_\lambda (\Ltr)} \|r^{\al-1}\widehat{w}_\star(kB\lambda,r)\|_{L^2_\lambda (\Ltr)},
\end{split}
  \end{equation*}
  from which and  and  Plancherel's equality, we infer
\begin{equation}\label{4.9,a}
    \begin{split}
\f34\big( \nu \|r^{\al}w_\star\|_{L^2_T( \hkz)}^2 + \mu_k \|r^{\al-1}w_\star\|_{L^2_T (\Ltr)}^2 \big)
  \leq & C\nu^{-1}\|r^\al F_\star\|_{L^2_T(\hkf)}^2\\
  & + C\|r^{\al+1}R^{w_\star}\|_{L^2_T (\Ltr)} \|r^{\al-1}w_\star\|_{L^2_T (\Ltr)}.
    \end{split}
\end{equation}
Whereas we get, by using \eqref{4.6,a} and taking $\hc$ to be large enough, that
    \begin{align*}
&C\|r^{\al+1}R^{w_\star}\|_{L^2_T (\Ltr)} \|r^{\al-1}w_\star\|_{L^2_T (\Ltr)} \\
&\leq \frac{C}{\log\hc}\|r^{\al-1}w_\star\|_{L^2_T (\Ltr)} \bigl( \mu_k \|r^{\al-1}w_\star \|_{L^2_T (\Ltr)}  + \nu \|r^{\al} \partial_r w_\star\|_{L^2_T(\Ltr)} \bigr)\\
&\leq \frac{1}{4} \bigl( \nu \|r^\al w_\star\|_{L^2_T(\hkz)}^2 + \mu_k \|r^{\al-1}w_\star\|_{L^2_T(\Ltr)}^2 \bigr).
    \end{align*}
By inserting the above inequality into \eqref{4.9,a}, we achieve
 \eqref{H-1 to H1, space-time weighted resolvent}.
\end{proof}

\begin{proposition}\label{space-time resolvent L2 to phi}
  {\sl For any $k\in \Z\backslash\{0\}$, $\varepsilon \in(0,2)$, $T>0$, there exist constants $C_\varepsilon>0$ and $\hat{C}_0\gg1$ depending only on $\varepsilon$, so that for $\hc\geq \hc_0$,
\begin{subequations} \label{ineq spcae-time resovent L2 to phi}
\begin{gather} \label{ineq spcae-time resovent L2 to phia}
    \mu_k^{\f12} |kB|^\f12 |k|^\f12\Bigl( \Big\|\frac{\partial_r\varphi_\star}{r^{1-\varepsilon}}\Big\|_{\LtT (\Ltr)} + |k|\Big\|\frac{\varphi_\star}{r^{2-\varepsilon}}\Big\|_{\LtT (\Ltr)} \Bigr) \leq C_\varepsilon \|r^{2+\varepsilon}F_\star\|_{\LtT (\Ltr)},\\
  \label{ineq spcae-time resovent L2 to phib}   \nu^\frac{1}{2} |kB|^\f12 |k|^\f12 \Bigl( \Big\|\frac{\partial_r\varphi_\star}{r^{1-\varepsilon}}\Big\|_{\LtT (\Ltr)} + |k|\Big\|\frac{\varphi_\star}{r^{2-\varepsilon}}\Big\|_{\LtT (\Ltr)} \Bigr) \leq C_\varepsilon \|r^{1+\varepsilon}F_\star\|_{\LtT (\hkf)}.
\end{gather}
\end{subequations}}
\end{proposition}
\begin{proof} By virtue of \eqref{expression of phi_k} and $\varphi_\star$ equation of \eqref{eq *},
we  write
     \begin{align*}
         -\varphi_\star = \cK[w_\star] + \cK[R^{\varphi_\star}],
     \end{align*}
which implies
\begin{equation}\label{4.8}
    \begin{split}
   \Big\| \frac{\partial_r\varphi_\star}{r^{1-\varepsilon}}\Big\|_{\LtT (\Ltr)} + |k|\Big\| \frac{\varphi_\star}{r^{2-\varepsilon}}\Big\|_{\LtT (\Ltr)} &\leq  \Big\| \frac{\partial_r\cK[w_\star]}{r^{1-\varepsilon}}\Big\|_{\LtT (\Ltr)} + |k|\Big\| \frac{\cK[w_\star]}{r^{2-\varepsilon}}\Big\|_{\LtT (\Ltr)} \\
   &\quad + \Big\|\frac{ \pa_r\cK[R^{\varphi_\star}]}{r^{1-\varepsilon}} \Big\|_{\LtT (\Ltr)} +|k| \Big\|\frac{ \cK[R^{\varphi_\star}]}{r^{2-\varepsilon}}\Big\|_{\LtT (\Ltr)}.
    \end{split}
\end{equation}

Below for simplicity, we skip the standard transformation between $L^2_\lambda (\Ltr)$ and $\LtT (\Ltr)$. We first,
 get, by using Proposition \ref{prop L2-phi resolvent estimate}, \eqref{4.6,a} and Proposition \ref{space-time resolvent L2-L2}, that
\begin{align*}
    &\Big\| \frac{\partial_r\cK[w_\star]}{r^{1-\varepsilon}}\Big\|_{\LtT (\Ltr)} + |k|\Big\| \frac{\cK[w_\star]}{r^{2-\varepsilon}}\Big\|_{\LtT (\Ltr)} \\
    &\lesssim_\varepsilon \frac{1}{\mu_k^{\f12} |kB|^\f12 |k|^\f12}\bigl(\|r^{2+\varepsilon}F_\star\|_{\LtT (\Ltr)}+\|r^{2+\varepsilon}R^{w_\star}\|_{\LtT (\Ltr)}\bigr) \\
    &\lesssim_\varepsilon \frac{1}{\mu_k^{\f12} |kB|^\f12 |k|^\f12}\bigl(\|r^{2+\varepsilon}F_\star\|_{\LtT (\Ltr)} + \nu^\f13 |kB|^\f23 \|r^\varepsilon w_\star\|_{\LtT (\Ltr)}+ \nu \|r^{1+\varepsilon}\pa_r w_\star\|_{\LtT (\Ltr)}\bigr)\\
    &\lesssim_\varepsilon \frac{1}{\mu_k^{\f12} |kB|^\f12 |k|^\f12}\|r^{2+\varepsilon}F_\star\|_{\LtT (\Ltr)}.
\end{align*}
While we get, by applying   Propositions \ref{prop L2-phi resolvent estimate},
\ref{resolvent estimate 4} and \ref{space-time resolvent,H-1 to H1}, that
\begin{align*}
    &\Big\| \frac{\partial_r\cK[w_\star]}{r^{1-\varepsilon}}\Big\|_{\LtT (\Ltr)} + |k|\Big\| \frac{\cK[w_\star]}{r^{2-\varepsilon}}\Big\|_{\LtT (\Ltr)}  \\
    &\lesssim_\varepsilon \frac{1}{\nu^\frac{1}{2} |kB|^\f12 |k|^\f12}\|r^{1+\varepsilon}F_\star\|_{\LtT (\hkf)} + \frac{1}{\mu_k^\f12 |kB|^\f12 |k|^\f12}\|r^{2+\varepsilon}R^{w_\star}\|_{\LtT (\Ltr)}  \\ &\lesssim_\varepsilon  \frac{1}{\nu^\frac{1}{2} |kB|^\f12 |k|^\f12} \|r^{1+\varepsilon}F_\star\|_{L^2_T(\hkf)}\\
    &\quad +\frac{1}{\nu^\frac{1}{6} |kB|^\f56 |k|^\f12} \bigl( \nu^\f13 |kB|^\f23 \|r^\varepsilon w_\star\|_{\LtT (\Ltr)}+ \nu \|r^{1+\varepsilon}\pa_r w_\star\|_{\LtT (\Ltr)} \bigr) \\
    &\lesssim_\varepsilon \frac{1}{\nu^\frac{1}{2} |kB|^\f12 |k|^\f12} \|r^{1+\varepsilon}F_\star\|_{L^2_T(\hkf)}.
\end{align*}
And it follows from Lemmas \ref{basic properties of ck} and \ref{Basic calculation for La_k} that
\begin{align*}
   \Big\|\frac{ \pa_r\cK[R^{\varphi_\star}] }{r^{1-\varepsilon}}\Big\|_{\LtT (\Ltr)} +|k| \Big\|\frac{ \cK[R^{\varphi_\star}]}{r^{2-\varepsilon}}\Big\|_{ \LtT (\Ltr)} &\lesssim_\varepsilon |k|^{-1} \|r^\varepsilon R^{\varphi_\star}\|_{ \LtT (\Ltr)} \\
   &\lesssim_\varepsilon  \frac{|k|^{-1}}{\log\hc} \Bigl(  \Big\| \frac{\partial_r\varphi_\star}{r^{1-\varepsilon}}\Big\|_{\LtT (\Ltr)} + |k|\Big\| \frac{\varphi_\star}{r^{2-\varepsilon}}\Big\|_{\LtT (\Ltr)} \Bigr).
\end{align*}

By substituting the above inequalities into \eqref{4.8} and then taking $\hc$  to be sufficiently large,  we  achieve \eqref{ineq spcae-time resovent L2 to phi}.
\end{proof}

\section{Space-time estimates of the vorticity in nonzero modes}\label{Space-time estimates of the vorticity in nonzero modes}

In this section, we   primarily investigate the enhanced dissipation and integrated inviscid damping estimates concerning \eqref{eq 5.1}.
 Below we split the solution of  \eqref{eq 5.1} $w_k=\wkh+\wki$ and $\varphi_k=\pkh+\pki,$ where $(\wkh,\pkh)$ and $(\wki,\pki)$ satisfy  respectively the following two equations:
 \begin{equation}\label{homogenous eq}
    \begin{cases}
         \partial_t \wkh - \nu\big(\partial_r^2-\frac{k^2-\frac{1}{4}}{r^2}\big)\wkh+\frac{ikB}{r^2}\wkh = 0, \quad (t,r)\in \R^+ \times [1,+\oo),\\
\wkh(t=0) =w_k(0),  \quad \wkh|_{r=1,\infty}=0,\\
 \big(\partial_r^2-\frac{k^2-\frac{1}{4}}{r^2}\big)\pkh = \wkh,\quad \pkh|_{r=1,\infty}=0,
    \end{cases}
\end{equation}
and
\begin{equation}\label{inhomogenous eq}
    \begin{cases}
         \partial_t \wki - \nu\big(\partial_r^2-\frac{k^2-\frac{1}{4}}{r^2}\big)\wki+\frac{ikB}{r^2}\wki = f_1+f_2, \quad (t,r)\in \R^+ \times [1,+\oo),\\
         \wki(t=0) =0, \quad  \wki|_{r=1,\infty}=0,\\
         \big(\partial_r^2-\frac{k^2-\frac{1}{4}}{r^2}\big)\pki = \wki, \quad\pki|_{r=1,\infty}=0,
    \end{cases}
\end{equation}
where we write $F_k=f_1+f_2$.


The main result of Section \ref{Space-time estimates of the vorticity in nonzero modes} states as follows:

\begin{proposition}\label{nonlinear space-time estimates}
{\sl Let $k\in \mathbb{Z}\backslash \{0 \}$, $\varepsilon\in(0,2)$ let $w_k$ be the solution to \eqref{eq 5.1} and the
 energy functionals $M_k(0)$ and $ E_k(T)$ are given respectively by \eqref{S1eq1} and \eqref{S1eq2}. Then
   there exist constants $C_\varepsilon>0, \hc_0\gg1$  depending only on $\varepsilon$, so that for $\hc\geq \hc_0$ there holds
    \begin{equation}\label{ineq, nonlinear space-time estimates}
    \begin{split}
      E_k(T)\leq&  C_\varepsilon \Big(M_k(0)+\mu_k^{-\f12} \|r^{2+\varepsilon}f_1\La_k\|_{\LtT (\Ltr)} + \nu^{-\f12}\|r^{1+\varepsilon}f_2\La_k\|_{\LtT (\hkf)}\Big).
    \end{split}
    \end{equation}}
\end{proposition}
\begin{proof}
    Proposition \ref{nonlinear space-time estimates} follows Propositions \ref{the inhomogeneous equation with zero initial data} and \ref{estimate of homogeneous eqs} below.
\end{proof}

\subsection{Space-time estimates of the inhomogeneous equation \eqref{inhomogenous eq}}\label{Space-time estimates of the inhomogeneous equation inhomogenous eq}

\begin{proposition}\label{the inhomogeneous equation with zero initial data}
  {\sl Let $k\in \mathbb{Z}\backslash \{0 \}$, $\varepsilon\in(0,2)$ and let $(\wki,\pki)$ solve \eqref{inhomogenous eq}. Then there exist constants $C_\varepsilon>0, \hc_0\gg1$ depending only on $\varepsilon$, so that for $\hc\geq \hc_0$,
\begin{equation}\label{ineq inhomogenous equations}
 \begin{split}
    &\|r^{1+\varepsilon}\wki \La_k\|_{\LoT(\Ltr)}+\nu^\frac{1}{2}\|r^{1+\varepsilon}\wki'\La_k\|_{\LtT (\Ltr) }+\mu_k^\frac{1}{2} \|r^{\varepsilon}\wki \La_k\|_{\LtT (\Ltr)}\\
    &\quad +|kB|^{\f12} |k|^\f12 \Bigl( \Big\|\frac{\pki'}{r^{1-\varepsilon}}\La_k\Big\|_{\LtT (\Ltr)}+|k|\Big\|\frac{\pki}{r^{2-\varepsilon}}\La_k\Big\|_{\LtT (\Ltr)} \Bigr)\\
    &\leq C_\varepsilon \Bigl(\mu_k^{-\f12} \|r^{2+\varepsilon}f_1\La_k\|_{\LtT (\Ltr)} + \nu^{-\f12}\|r^{1+\varepsilon}f_2\La_k\|_{\LtT (\hkf)}\Bigr).
 \end{split}
\end{equation}}
\end{proposition}
\begin{proof}
We first get, by applying Propositions \ref{space-time resolvent L2-L2}-\ref{space-time resolvent L2 to phi} and Lemma \ref{Basic calculation for La_k}, that
\begin{equation}
\begin{split}\label{ineq 6.5}
     &\nu^\frac{1}{2}\|r^{1+\varepsilon}\wki'\La_k\|_{\LtT (\Ltr) }+\mu_k^\frac{1}{2} \|r^{\varepsilon}\wki \La_k\|_{\LtT (\Ltr)}\\
    &\quad +|kB|^{\f12} |k|^\f12 \Bigl( \Big\|\frac{\pki'}{r^{1-\varepsilon}}\La_k\Big\|_{\LtT (\Ltr)}+|k|\Big\|\frac{\pki}{r^{2-\varepsilon}}\La_k\Big\|_{\LtT (\Ltr)} \Bigr)\\
    &\lesssim_\varepsilon \mu_k^{-\f12} \|r^{2+\varepsilon}f_1\La_k\|_{\LtT (\Ltr)} + \nu^{-\f12}\|r^{1+\varepsilon}f_2\La_k\|_{\LtT (\hkf)}.
\end{split}
\end{equation}

Whereas we get, by taking
the real part of the $\Ltr$ inner product of \eqref{inhomogenous eq} with $r^{2+2\varepsilon}\La_k^2 \wki$, and then using integration by parts and Lemma \ref{Basic calculation for La_k}, that
\begin{equation*}
    \begin{split}
    \frac{d}{dt}\|r^{1+\varepsilon}\wki \La_k\|_{ \Ltr}^2 &\lesssim  \big|\braket{\wki,r^{2+2\varepsilon}\wki \pa_t \La_k^2 }_\Ltr\big| + \nu \big| \braket{ \wki, r^{2+2\varepsilon}\wki \pa_r^2 \La_k^2 }_\Ltr\big| \\
    &\quad + \|r^{2+\varepsilon}f_1\La_k\|_\Ltr \|r^{\varepsilon}\wki \La_k\|_{\Ltr} + \|r^{1+\varepsilon}f_2\La_k\|_\hkf \|r^{1+\varepsilon}\wki \La_k\|_{\hkz}\\
    &\lesssim \nu\|r^{1+\varepsilon}\wki'\La_k\|_{ \Ltr }^2 + \mu_k \|r^{\varepsilon}\wki \La_k\|_{\Ltr}^2 \\
    &\quad + \mu_k^{-1} \|r^{2+\varepsilon}f_1\La_k\|_{ \Ltr}^2+ \nu^{-1}\|r^{1+\varepsilon}f_2\La_k\|_{ \hkf}^2.
    \end{split}
\end{equation*}
By integrating the above inequality over $[0,T]$ and  using \eqref{ineq 6.5}, we achieve \eqref{ineq inhomogenous equations}.
\end{proof}

\subsection{Space-time estimates of the homogeneous equation \eqref{homogenous eq}}\label{Space-time estimates of the homogeneous equation homogenous eq}In this subsection, we present the space-time estimates of the solution to the homogeneous linear equation \eqref{homogenous eq}.
The main idea will be based on a decomposition of $\wkh$, which was already used in \cite{LZZ-25} to derive  linear enhanced dissipation around TC flow. For reader's convenience, we shall explain the motivation of the decomposition. Indeed if $\nu=0$ in  \eqref{homogenous eq}, the
 corresponding solution of  \eqref{homogenous eq} can be expressed as a rotation of the initial data, i.e. $e^{-i\f{kB}{r^2}t}w_k(0)$.
 This motivates us to consider the solution $\wkh$ may be in the vicinity of a solution similar to $e^{-i\f{kB}{r^2}t}w_k(0)$.
And we decompose the solution $\wkh$ into
\begin{align}\label{S6eq1}
\wkh=\wkh^{(1)}+\wkh^{(2)} \quad \text{ with }\ \wkh^{(1)}{\eqdefa}e^{-i\f{kB}{r^2}t}w_k^{(1)}.
\end{align}
Then in view of \eqref{homogenous eq}, we find
\begin{align*}
&e^{-i\f{kB}{r^2}t}\partial_tw_k^{(1)}-\nu\Big(\partial_r^2-\f{k^2-\frac{1}{4}}{r^2}\Big)\big(e^{-i\f{kB}{r^2}t}w_k^{(1)}\big)\\
&=-\Big(\partial_t\wkh^{(2)}-\nu\Big(\partial_r^2-\f{k^2-\f14}{r^2}\Big)\wkh^{(2)}+\f{ikB}{r^2}\wkh^{(2)}\Big).
\end{align*}
It is easy to observe that
\begin{align*}
    &-\nu\Big(\partial_r^2-\f{k^2-\frac{1}{4}}{r^2}\Big)\big(e^{-i\f{kB}{r^2}t}w_k^{(1)}\big)\\
    &=-e^{-i\f{kB}{r^2}t}\nu\bigg(\Big(\partial_r^2-\f{k^2-\frac{1}{4}}{r^2}\Big)w_k^{(1)}-\Big|\f{2kBt}{r^3}\Big|^2w_k^{(1)}
    +\f{4ikBt}{r^3}\partial_rw_k^{(1)}
    -\f{6ikBt}{r^4}w_k^{(1)}\bigg).
\end{align*}
As a consequence, we obtain
\begin{align*}
&e^{-i\f{kB}{r^2}t}\bigg(\partial_tw_k^{(1)}-\nu\Big(\partial_r^2-\f{k^2-\f14}{r^2}\Big)w_k^{(1)}+\nu\Big|\f{2kBt}{r^3}\Big|^2w_k^{(1)}\bigg)\\
&=-\bigg(\partial_t\wkh^{(2)}-\nu\Big(\partial_r^2-\f{k^2-\f14}{r^2}\Big)\wkh^{(2)}+\f{ikB}{r^2}\wkh^{(2)}\bigg)\\
&\quad+e^{-i\f{kB}{r^2}t}\nu\Big(\f{4ikBt}{r^3}\partial_rw_k^{(1)}-\f{6ikBt}{r^4}w_k^{(1)}\Big),
\end{align*}
where the  term
\begin{align*}
    -\nu\bigg(\Big(\partial_r^2-\f{k^2-\frac{1}{4}}{r^2}\Big)w_k^{(1)}-\Big|\f{2kBt}{r^3}\Big|^2w_k^{(1)}\bigg)
\end{align*}
is what we refer to as the enhanced dissipation term.

The computation motivates us to construct  $w_k^{(1)}$ and $\wkh^{(2)}$ in \eqref{S6eq1} through
\begin{align}
\left\{
    \begin{aligned}
\label{w-1}&\partial_tw_k^{(1)}-\nu\Big(\partial_r^2-\f{k^2+\Theta^2}{r^2}\Big)w_k^{(1)}+\nu\Big|\f{2kBt}{r^3}\Big|^2w_k^{(1)}=0,\\
&w_k^{(1)}|_{t=0}=w_k(0),\quad w_k^{(1)}|_{r=1,+\oo}=0,\\
& \Big(\partial_r^2-\f{k^2-\frac{1}{4}}{r^2}\Big)\varphi_{k,H}^{(1)} = e^{-i\f{kB}{r^2}t}w_k^{(1)},\quad\varphi_{k,H}^{(1)}|_{r=1,+\oo}=0,
\end{aligned}
\right.
\end{align}
and
\begin{align}
\left\{
    \begin{aligned}
\label{w-2}&\partial_t\wkh^{(2)}-\nu\Big(\partial_r^2-\f{k^2-\f14}{r^2}\Big)\wkh^{(2)}+\f{ikB}{r^2}\wkh^{(2)}\\
&\qquad =e^{-i\f{kB}{r^2}t}\nu\Big(i\f{4kBt}{r^3}\partial_rw_k^{(1)}-i\f{6kBt}{r^4}w_k^{(1)}+ \frac{4\cdot\Theta^2+1}{4r^2}w_k^{(1)}\Big),\\
&\wkh^{(2)}|_{t=0}=0,\quad \wkh^{(2)}|_{r=1,+\oo}=0,\\
& \Big(\partial_r^2-\f{k^2-\f14}{r^2}\Big)\varphi_{k,H}^{(2)} = \wkh^{(2)},\quad\varphi_{k,H}^{(2)}|_{r=1,+\oo}=0.
\end{aligned}
\right.
\end{align}
Here $\Theta\gg 1 $ is a fixed constant, which we shall take
 $\Theta\geq 10^9$ in the rest of this section.

\begin{proposition}\label{estimate of homogeneous eqs}
{\sl Let $(\wkh,\pkh)$ solve the homogeneous equation \eqref{homogenous eq}.  Then
 for any $k\in\Z\backslash\{0\}$, $\varepsilon \in (0,2)$, $T>0,$ there exist a constant $C_{\varepsilon}>0$ depending only on $\varepsilon$, so that
    \begin{equation}\label{S6eq3}
    \begin{split}
     & \|r^{1+\varepsilon}\wkh\La_k\|_{\LoT(\Ltr)}+ \nu^\f12 \|r^{1+\varepsilon}\pa_r\wkh\La_k\|_{\LtT (\Ltr)} +  \mu_k^\frac{1}{2} \|r^{\varepsilon} \wkh\La_k\|_{\LtT (\Ltr)} \\
     &\qquad + |kB|^{\f12} |k|^\f12 \Bigl(\Big\| \frac{\pa_r\pkh}{r^{1-\varepsilon}}\La_k\Big\|_{\LtT (\Ltr)} + |k|\Big\| \frac{\pkh}{r^{2-\varepsilon}}\La_k\Big\|_{\LtT (\Ltr)}\Bigr) \leq C_\varepsilon M_k(0).
    \end{split}
    \end{equation}}
\end{proposition}
\begin{proof} In view of \eqref{w-2}, we get,
by applying Proposition \ref{the inhomogeneous equation with zero initial data} and Propositions \ref{estimate of homogeneous eqs without log loss} and \ref{prop estimate of pa_r w^(1)} below, that
\begin{align*}
& \big\|r^{1+\varepsilon} \wkh^{(2)}\La_k\big\|_{\LoT(\Ltr)}+ \nu^\f12 \big\|r^{1+\varepsilon} \pa_r\wkh^{(2)}\La_k\big\|_{\LtT (\Ltr)}+ \mu_k^\frac{1}{2} \big\|r^{\varepsilon} \wkh^{(2)}\La_k\big\|_{\LtT (\Ltr)} \\
     &\quad + |kB|^{\f12} |k|^\f12 \Bigl(\Big\|  \frac{\pa_r\pkh^{(2)}}{r^{1-\varepsilon}}\La_k\Big\|_{\LtT (\Ltr)} + |k|\Big\| \frac{\pkh^{(2)}}{r^{2-\varepsilon}}\La_k\Big\|_{\LtT (\Ltr)}\Bigr) \\
& \lesssim_\varepsilon \mu_k^{-\f12} \nu \Bigl( \Big\|\frac{kBt}{r^{1-\varepsilon}} \pa_r w_k^{(1)}\La_k\Big\|_{\LtT (\Ltr)} +\Big\|\frac{kBt}{r^{2-\varepsilon}}  w_k^{(1)}\La_k\Big\|_{\LtT (\Ltr)} + \big\|r^{\varepsilon} w_k^{(1)}\La_k\big\|_{\LtT (\Ltr)} \Bigr) \\
&\lesssim_\varepsilon \big((\nu/\mu_k)^\f12 +(\nu/\mu_k)\big) \big( \|r^{1+\varepsilon} w_k(0)\La_k(0)\|_{\Ltr}+\|r^{2+\varepsilon} \partial_r w_k(0)\La_k(0)\|_{\Ltr} \big)\\
&\lesssim_\varepsilon  \|r^{2+\varepsilon} w_k(0)\La_k(0)\|_{\Ltr}+\|r^{3+\varepsilon} \partial_r w_k(0)\La_k(0)\|_{\Ltr},
    \end{align*}
which together with \eqref{eqs-with log weight}, \eqref{eqs-with log weight,derivative} and \eqref{S6eq2} ensures \eqref{S6eq3}.
 We thus complete the proof of Proposition \ref{estimate of homogeneous eqs}.
\end{proof}

\begin{proposition}\label{estimate of homogeneous eqs without log loss}
{\sl Let $k \in \Z\backslash\{0\}$, $0\leq \alpha \leq 4$, $\hc\gg1 $ and  $w_k^{(1)}$ solve \eqref{w-1}. Then there exists a constant $C>0,$
which is independent of $\nu,k,B,\hc,T,\alpha$, so that
\begin{equation}
    \label{eqs-with log weight}\begin{split}
&\big\|r^{\al}w_k^{(1)}\La_k\big\|_{\LoT(\Ltr)}^2 + \nu \big\|r^{\al}\partial_r w_k^{(1)}\La_k \big\|_{\LtT (\Ltr)}^2+ \mu_k \big\|r^{\al-1}w_k^{(1)}\La_k\big\|_{\LtT (\Ltr)}^2\\
&+ \kappa_k\Big\|\frac{\kappa_k t}{r^{3-\al}} w_k^{(1)}\La_k\Big\|_{\LtT (\Ltr)}^2
\leq C\|r^{\alpha} w_k(0)\La_k(0)\|_{\Ltr}^2.
\end{split}
\end{equation}}
\end{proposition}
\begin{proof}For the brevity of notations, we only present the proof of \eqref{eqs-with log weight} for the case $\alpha=1$, and we divide our proof  into the following six steps:

\no {\bf Step 1.} The change of unknown:  $\eta{\eqdefa}w_k^{(1)}\La_k$.

Let $\eta\eqdefa w_k^{(1)}\La_k.$ Then in view of \eqref{w-1}, we find
\begin{align}
\label{eta}\left\{
    \begin{aligned}
&\partial_t\eta-\nu\Big(\partial_r^2- \f{k^2+\Theta^2}{r^2}\Big)\eta+\nu\Big|\f{2kBt}{r^3}\Big|^2\eta= R_1^\eta+R_2^\eta,\\
&\eta|_{t=0}=\eta(0)=w_k(0)\La_k(0),\quad \eta|_{r=1,+\oo}=0,
\end{aligned}
\right.
\end{align}
where we denote $R_1^\eta{\eqdefa}-2\nu \frac{\partial_r\La_k}{\La_k}\partial_r\eta +\big(2\nu|\frac{\partial_r\La_k}{\La_k}|^2-\nu\frac{\partial_r^2\La_k}{{{\La_k}}}\big)\eta$ and $R_2^\eta{\eqdefa}\frac{\partial_t \La_k}{\La_k}{{\eta}}$.

It is easy to observe from Lemma \ref{Basic calculation for La_k} that
\begin{align*}
    \|r\partial_r w_k^{(1)}\La_k \|_{\LtT (\Ltr)}+  \|w_k^{(1)}\La_k\|_{\LtT (\Ltr)}\approx\|r\partial_r\eta \|_{\LtT (\Ltr)}+\|\eta\|_{\LtT (\Ltr)}.
\end{align*}
So that in order prove \eqref{eqs-with log weight}, it remains  to prove
\begin{equation}\label{equivalent estimate of wk1Mk}
    \begin{split}
&\|r\eta\|_{\LoT (\Ltr)}^2 + \nu \|r\partial_r\eta \|_{\LtT (\Ltr)}^2+ \kappa_k\Big\|\frac{\kappa_k t}{r^2} \eta\Big\|_{\LtT (\Ltr)}^2+ \mu_k \|\eta\|_{\LtT (\Ltr)}^2
\leq C\|r \eta(0)\|_{\Ltr}^2.
\end{split}
\end{equation}

\no{\bf Step 2.}  Dyadic decomposition.

We first fix some function $\chi(r)\in \mathcal{C}^2_c[\f23,2]$  (see the construction in \cite{LZZ-25}) and denote $\chi_j{\eqdefa} \chi(\frac{r}{2^j})$ so that
\begin{equation}\label{S6eq3}
\begin{cases}
    0\leq\chi_j\leq 1, \quad  \chi_0(1)=1, \quad \sum\limits_{j\geq 0 }\chi_j(r) = 1, \quad \f12 \leq \sum\limits_{j\geq 0 } \chi_j^2(r) \leq 1, \ \forall\ r\geq 1,\\
    \operatorname{supp}_r\chi_j \subset [2^{j+1}/3,2^{j+1}],\quad \chi_j\chi_\ell=0 \ \forall\ |j-\ell|\geq 2, \quad \|\chi\|_{C^2} \leq 4\cdot 10^6.
\end{cases}
\end{equation}
Then it is obvious to observe that
\[
\|r^\alpha \eta\|_{\Ltr}^2 \approx \sum_{j \geq 0} 2^{2\alpha j} \|\eta_j\|_{\Ltr}^2,
\]
where  \(\eta_j {\eqdefa} \chi_j \eta\), which  satisfies
\begin{align}\label{eq eta_j}
\left\{
    \begin{aligned}
&\partial_t\eta_j-\nu\Big(\partial_r^2-\f{k^2+\Theta^2}{r^2}\Big)\eta_j+\nu\Big|\f{2kBt}{r^3}\Big|^2\eta_j= \chi_j(R_1^\eta+R_2^\eta)-\nu (2^{-2j}\chi_j''\eta+2^{1-j}\chi_j'\partial_r \eta), \\
&\eta_j|_{t=0}=\eta_j(0),\quad \eta_j|_{r=\max(1,\frac{2^{j+1}}{3})}=\eta_j|_{r= 2^{j+1}}=0, \quad (t,r) \in \R^+\times [\max(1,2^{j+1}/3), 2^{j+1}],
\end{aligned}
\right.
\end{align} here and in what follows, we always denote $\chi_j'{\eqdefa} \chi'(\frac{r}{2^j})$ and $\chi_j''{\eqdefa} \chi''(\frac{r}{2^j})$.
Noticing that $\eta(r) = \sum\limits_{\ell\geq 0} \eta_\ell(r)$ for $r\geq 1$ and $\chi_j\chi_\ell=0$ if $ |j-\ell|\geq 2$, we also have
$$\chi_j''\eta = \sum\limits_{|\ell-j|\leq1, \ell\geq0}\chi_j''\eta_\ell, \qquad \chi_j'\partial_r \eta= \sum\limits_{|\ell-j|\leq1, \ell\geq0} \chi_j'\partial_r \eta_\ell .$$

Next we introduce the following energy functional:
\beq\label{S6eq5}
\begin{split}
 \|\eta\|_{Y_T}^2 {\eqdefa}& \f12 \sum_{j\geq0}  \|r\eta_j\|_{\LoT(\Ltr)}^2
 +\sum_{j\geq0}\|\eta_j\|_{X_T}^2,
 \with\\
\|\eta_j\|_{X_T}^2{\eqdefa} \frac{1}{2} \|r\eta_j(T)\|^2&+\nu\|r\partial_r \eta_j\|_{\LtT (\Ltr)}^2 + \nu(k^2+\Theta^2)\|\eta_j\|_{\LtT (\Ltr)}^2+ 4\kappa_k^3 \|\frac{\tau}{r^2}\eta_j\|_{\LtT (\Ltr)}^2.\end{split}
\eeq


 \no{\bf Step 3.} The control of \(\kappa_k\|\eta\|_{\LtT (\Ltr)}^2\) by \(\|\eta\|_{Y_T}^2\).

 Let
\begin{enumerate}
    \item \(D_1{\eqdefa} \bigl\{(\tau,r) \in (0,T) \times (1,+\infty) : r^2 \leq \kappa_k \tau \ \bigr\}\);
    \item \(D_2 {\eqdefa} \bigl\{(\tau,r) \in (0,T) \times (1,+\infty) : 0 \leq \kappa_k \tau \leq r^2 \ \bigr\}\).
\end{enumerate}
Here $D_1$ is the region where enhanced dissipation occurs. Then we decompose
\begin{equation*}\label{step3 begin}
\begin{split}
    \kappa_k\|\eta\|_{\LtT (\Ltr)}^2 =  \kappa_k\Big(\iint_{(\tau,r)\in D_1}+\iint_{(\tau,r)\in D_2}\Big) |\eta|^2 dr d\tau\eqdefa \mathcal{I}_1+\mathcal{I}_2.
\end{split}
\end{equation*}

Since $\sum\limits_{j\geq 0} \chi_j \geq \f12$,  one has
\begin{equation*}
    \mathcal{I}_1 \leq  \kappa_k \Big\|\frac{\kappa_k \tau}{r^2}\eta\Big\|_{\LtT (\Ltr)}^2 \leq 2\kappa_k \sum_{j\geq 0} \Big\|\frac{\kappa_k \tau}{r^2}\eta_j\Big\|_{\LtT (\Ltr)}^2 \leq \f12 \|\eta\|_{Y_T}^2.
\end{equation*}
Whereas it follows from a direct computation that
\begin{align*}
    \mathcal{I}_2 &\leq 2 \kappa_k \sum_{j\geq 0} \iint_{(\tau,r)\in D_2}|\eta_j|^2 drd\tau\\
    &\leq 2 \kappa_k \sum_{j\geq 0} \int_0^{\min(T,2^{2j+2}/\kappa_k)} \int_1^\oo \frac{9}{2^{2j+2}}|r\eta_j|^2 drd\tau\\
    &\leq 18 \sum_{j\geq 0}   \|r\eta_j\|_{\LoT(\Ltr)}^2 \leq 36 \|\eta\|_{Y_T}^2.
\end{align*}
As a result, it comes out
\begin{equation}\label{step 3 final}
\kappa_k\|\eta\|_{\LtT (\Ltr)}^2\lesssim \|\eta\|_{Y_T}^2.
\end{equation}

 \no{\bf Step 4.} The estimation of \(\|\eta\|_{Y_T}\).

We get, by first taking the real part of \(\Ltr\) inner product of \eqref{eq eta_j} with  \(r^2\eta_j\) and then using integration by parts and \eqref{L_k w, r al w inner product}, that
\begin{equation*}
    \begin{split}
\frac{d}{dt}\|\eta_j(t)\|_{X_t}^2 &\leq \big|\braket{r^2\chi_j R_1^\eta,\eta_j}_\Ltr\big| +\big|\braket{r^2\chi_j R_2^\eta,\eta_j}_\Ltr\big| + \nu\|\eta_j\|_{\Ltr}^2 \\
&\quad +\nu\sum_{|\ell-j|\leq1, \ell\geq0}(2^{-2j}\|r^2\chi_j''\|_{\Lor}\|\eta_\ell\|_{\Ltr}+2^{1-j}\|r\chi_j'\|_{\Lor}\|r\partial_r \eta_\ell\|_{\Ltr}) \|\eta_j\|_{\Ltr}.
    \end{split}
\end{equation*}
Observing that $\text{supp}_r\chi_j \in [2^{j+1}/3, 2^{j+1}] $ and $\|\chi_j''\|_{L^\oo},\|\chi_j'\|_{L^\oo} \leq 4\cdot10^6,$  we find
\begin{equation*}
    \begin{split}
    \frac{d}{dt}\|\eta_j(t)\|_{X_t}^2 \leq &\big|\braket{r^2\chi_j R_1^\eta,\eta_j}_\Ltr\big|+\big|\braket{r^2\chi_j R_2^\eta,\eta_j}_\Ltr\big| + \nu\|\eta_j\|_{\Ltr}^2 \\
&+16\cdot 10^6\nu \sum_{|\ell-j|\leq1, \ell\geq0}(\|\eta_\ell\|_{\Ltr}+\|r\partial_r \eta_\ell\|_{\Ltr}) \|\eta_j\|_{\Ltr},\\
    \end{split}
\end{equation*}
Then by integrating the above inequality  over $[0,T]$ and using the definition of $X_T$ given by \eqref{S6eq5}, we obtain
\begin{equation*}
    \begin{split}
&\f12 \|r\eta_j\|_{\LoT(\Ltr)}^2+\|\eta_j\|_{X_T}^2 \leq 2\times \Bigl(\frac{1}{2}\|r\eta_j(0)\|_{\Ltr}^2 +\int_0^T\big|\braket{r^2\chi_j R_1^\eta,\eta_j}_\Ltr\big| d\tau\\
     &   \ + \int_0^T\big|\braket{r^2\chi_j R_2^\eta,\eta_j}_\Ltr\big| d\tau   + \frac{1}{\Theta^2}\|\eta_j\|_{X_T}^2+ 16\cdot 10^6(1/\Theta+1/\Theta^2)  \|\eta_j\|_{X_T}\sum_{|\ell-j|\leq1, \ell\geq0} \|\eta_\ell\|_{X_T} \Bigr).\\
    \end{split}
\end{equation*}
By summarizing the above inequalities for $j\geq 0$ and  using Cauchy Inequality, we achieve
\begin{equation}\label{7.19}
\begin{split}
  \|\eta\|_{Y_T}^2
    &\leq \sum_{j\geq0}\|r\eta_j(0)\|_\Ltr^2  +2\int_0^T \sum_{j\geq0}\big|\braket{r^2\chi_j R_1^\eta,\eta_j}_\Ltr\big| d\tau \\
    &\quad +2\int_0^T \sum_{j\geq0}\big|\braket{r^2\chi_j R_2^\eta,\eta_j}_\Ltr\big| d\tau \\
    &\quad + \frac{2}{\Theta^2} \|\eta\|_{Y_T}^2 + 16\cdot 10^6(1/\Theta+1/\Theta^2) \big( \|\eta\|_{Y_T}^2 + 3\sum_{j\geq0}\sum_{|\ell-j|\leq1, \ell\geq0} \|\eta_\ell\|_{X_T}^2\big)\\
    &= \sum_{j\geq0}\|r\eta_j(0)\|_\Ltr^2  +\big(2/\Theta^2+16\cdot10^7(1/\Theta+1/\Theta^2)\big) \|\eta\|_{Y_T}^2\\
    &\quad+2\int_0^T \sum_{j\geq0}\big|\braket{r^2\chi_j R_1^\eta,\eta_j}_\Ltr\big| d\tau +2\int_0^T \sum_{j\geq0}\big|\braket{r^2\chi_j R_2^\eta,\eta_j}_\Ltr\big| d\tau.
\end{split}
\end{equation}

\no{\bf Step 5.} The estimation of terms in \eqref{7.19} involving $R_i^\eta$.

We first observe that
\begin{equation}\label{6.18,a}
    \nu \|r\pa_r \eta\|_{\LtT (\Ltr)}^2 +  \nu(k^2+\Theta^2)\|\eta\|_{\LtT (\Ltr)}^2 \lesssim \|\eta\|_{Y_T}^2.
\end{equation}
While it follows from  Lemma \ref{Basic calculation for La_k} that
\begin{equation*}
    \begin{split}
&\sum_{i=1}^2\int_0^T \sum_{j\geq0}\big|\braket{r^2\chi_j R_i^\eta,\eta_j}_\Ltr \big| d\tau \\
&\lesssim \frac{\nu}{\log\hc} (\|r\pa_r \eta\|_{\LtT (\Ltr)} + \|\eta\|_{\LtT (\Ltr)})\|\eta\|_{\LtT (\Ltr)} + \f{\kappa_k}{\hc\log\hc} \|\eta\|_{\LtT (\Ltr)}^2 \\
&\lesssim   \frac{\nu}{\log\hc} \|r\pa_r \eta\|_{\LtT (\Ltr)}^2 +  \frac{\nu+\kappa_k}{\log\hc}\|\eta\|_{\LtT (\Ltr)}^2 .
    \end{split}
\end{equation*}
So that by taking $\hc \gg1$ and using \eqref{step 3 final} and \eqref{6.18,a}, we obtain
\begin{equation}\label{estimate of R}
    \begin{split}
&\sum_{i=1}^{2}\int_0^T \sum_{j\geq0}\big|\braket{r^2\chi_j R_i^\eta,\eta_j}_\Ltr \big| d\tau \leq \f18 \|\eta\|_{Y_T}^2.
    \end{split}
\end{equation}

\no {\bf Step 6.} The end of the proof.

Due to $\Theta \geq 10^9$,
by inserting \eqref{estimate of R} into \eqref{7.19}, we achieve
\begin{align*}
    \|\eta\|_{Y_T}^2 &\leq \sum_{j\geq0}\|r\eta_j(0)\|_\Ltr^2 +\big(1/4+2/\Theta^2+16\cdot10^7(1/\Theta+1/\Theta^2) \big)  \|\eta\|_{Y_T}^2\\
   &\leq \|r\eta(0)\|_\Ltr^2 +\f12 \|\eta\|_{Y_T}^2,
    \end{align*}
which together with \eqref{step 3 final} and \eqref{6.18,a} ensures (\ref{eqs-with log weight}).
We thus complete the proof of Proposition \ref{estimate of homogeneous eqs without log loss}.
\end{proof}

\begin{proposition}\label{prop estimate of pa_r w^(1)}
{\sl Let $k\in \Z\backslash\{0\}$, $\alpha \in[1,5]$, $\hc\gg1$ and   $w_k^{(1)}$ solve  \eqref{w-1}.  Then there exists  constant $C>0$ independent of $\nu,k,B,\hc,T,\alpha$, so that
\begin{equation}
    \label{eqs-with log weight,derivative}\begin{split}
&\big\|r^{\al} \pa_rw_k^{(1)}\La_k\big\|_{\LoT(\Ltr)}^2 + \nu \big\|r^{\al}\partial^2_r w_k^{(1)}\La_k \big\|_{\LtT (\Ltr)}^2+ \mu_k  \big\|r^{\al-1}\pa_r w_k^{(1)}\La_k\big\|_{\LtT (\Ltr)}^2\\
&+ \kappa_k \Big\|\frac{\kappa_k t}{r^{3-\al}} \pa_r w_k^{(1)}\La_k\Big\|_{\LtT (\Ltr)}^2
\leq C\big(\|r^{\alpha} \pa_r w_k(0)\La_k(0)\|_{\Ltr}^2 + \|r^{\alpha-1} w_k(0)\La_k(0)\|_{\Ltr}^2\big).
\end{split}
\end{equation}}
\end{proposition}

\begin{proof}
    We inherit the notations used in the proof of Proposition \ref{estimate of homogeneous eqs without log loss}, and only present
    the detailed proof for the case $\alpha=2.$ We first get,  by taking the real part of the $\Ltr$ inner product of \eqref{eq eta_j}
     with $-\pa_r (r^4 \pa_r\eta_j),$ that
\begin{equation}\label{6.28}
    \begin{split}
&\frac{1}{2}\frac{d}{dt} \|r^2 \pa_r \eta_j\|_{\Ltr}^2 + \nu\Re\Big\langle-\Big(\pa_r^2 -\frac{k^2+\Theta^2}{r^2}\Big)\eta_j + \frac{4|kB|^2 t^2}{r^6} \eta_j, -\pa_r (r^4 \pa_r\eta_j)\Big\rangle_\Ltr \\
&\qquad \qquad \qquad = \Re\braket{\chi_j (R^\eta_1+R^\eta_2) -\nu(2^{-2j}\chi_j''\eta + 2^{1-j}\chi_j'\pa_r \eta) , -\pa_r (r^4 \pa_r\eta_j)}_\Ltr.
    \end{split}
\end{equation}
By using integration by parts, one has
\begin{align*}
    &\Re\Big\langle-\nu\Big(\pa_r^2 -\frac{k^2+\Theta^2}{r^2}\Big)\eta_j +\nu \frac{4|kB|^2 t^2}{r^6} \eta_j, \  -\pa_r (r^4 \pa_r\eta_j)\Big\rangle_\Ltr \\
    &= \nu \|r^2\pa_r^2 \eta_j\|_\Ltr^2 + \nu \Re\braket{\pa_r^2 \eta_j, 4r^3 \pa_r \eta_j}_\Ltr \\
    &\quad+ \nu(k^2+\Theta^2) \|r\pa_r \eta_j\|_\Ltr^2 -2 \nu(k^2+\Theta^2) \Re\braket{r\eta_j,\pa_r \eta_j}_\Ltr\\
    &\quad +4\nu |kBt|^2  \|r^{-1}\pa_r \eta_j\|_\Ltr^2 -24\nu |kBt|^2 \Re\braket{r^{-2}\eta_j,r^{-1}\pa_r \eta_j}_\Ltr \\
    &\geq \f\nu2 \bigl(\|r^2\pa_r^2 \eta_j\|_\Ltr^2+(k^2+\Theta^2) \|r\pa_r \eta_j\|_\Ltr^2  \bigr) + 2\nu |kBt|^2  \|r^{-1}\pa_r \eta_j\|_\Ltr^2 \\
    &\quad - C\bigl(  \nu  \|r\pa_r \eta_j\|_\Ltr^2  +\nu(k^2+\Theta^2) \| \eta_j\|_\Ltr^2  +\nu |kBt|^2\|r^{-2}\eta_j\|_\Ltr^2  \bigr).
\end{align*}
By using the fact: $|\frac{\pa_t\Lambda_k}{\Lambda_k}| + r|\pa_r\frac{\pa_t\Lambda_k}{\Lambda_k} | \lesssim \frac{\kappa_k}{\log\hc} \frac{1}{r^2}$, we obtain
    \begin{align*}
&\Re\braket{\chi_j R^\eta_2, -\pa_r (r^4 \pa_r\eta_j)}_\Ltr \\
&= \Big\langle\frac{\pa_t \La_k}{\La_k} \eta_j, -\pa_r (r^4 \pa_r\eta_j)\Big\rangle_\Ltr
 = \bigg\|\sqrt{\frac{\pa_t \La_k}{\La_k}} r^2\pa_r \eta_j\bigg\|_\Ltr^2 + \Big\langle\pa_r \Big( \frac{\pa_t \La_k}{\La_k} \Big) \eta_j, r^4 \pa_r \eta_j\Big\rangle\\
 &\lesssim \frac{\kappa_k}{\log\hc} \bigl (\|r\pa_r \eta_j\|_\Ltr^2 + \| \eta_j\|_\Ltr^2\bigr).
    \end{align*}
It follows from  Lemma \ref{Basic calculation for La_k} that
\begin{align*}
  & \Re \braket{\chi_j R^\eta_1-\nu(2^{-2j}\chi_j''\eta + 2^{1-j}\chi_j'\pa_r \eta), -\pa_r (r^4 \pa_r\eta_j)}_\Ltr \\
   &\lesssim \bigl(\|\chi_j r^2R^\eta_1\|_\Ltr +\nu \sum\limits_{|\ell-j|\leq1, \ell\geq0} (\|\eta_\ell\|_\Ltr +\|r\pa_r \eta_\ell\|_\Ltr ) \bigr) \times (\|r^2\pa_r^2 \eta_j\|_\Ltr + \|r\pa_r \eta_j\|_\Ltr) \\
   &\leq C\nu \sum\limits_{|\ell-j|\leq1, \ell\geq0} \big(\|\eta_\ell\|_\Ltr^2 +\|r\pa_r \eta_\ell\|_\Ltr^2 \big) + \frac{\nu}{4}\big(\|r^2\pa_r^2 \eta_j\|_\Ltr^2 + \|r\pa_r \eta_j\|_\Ltr ^2 \big).
\end{align*}
By substituting the above estimates into \eqref{6.28} and integrating the resulting inequality over $[0,T]$,  we find
\begin{align*}
  &\|r^2\pa_r \eta_j\|_{\LoT(\Ltr)}^2 +   \nu\|r^2\partial_r^2 \eta_j\|_{\LtT (\Ltr)}^2 + \nu(k^2+\Theta^2)\|r\partial_r\eta_j\|_{\LtT (\Ltr)}^2+ 4\kappa_k \Big\|\frac{\kappa_k t}{r}\partial_r\eta_j\Big\|_{\LtT (\Ltr)}^2\\
  &\lesssim \|r^2\pa_r\eta_j(0)\|_\Ltr^2+ \frac{\kappa_k}{\log \hc} \big(\|r  \pa_r \eta_j\|_{\LtT(\Ltr)}^2 + \| \eta_j\|_{\LtT (\Ltr)}^2 \big)\\
  &\quad + \sum\limits_{|\ell-j|\leq1, \ell\geq0} \Big( \nu  \|r\pa_r \eta_\ell\|_{\LtT (\Ltr)}^2  +\nu(k^2+\Theta^2) \| \eta_\ell\|_{\LtT (\Ltr)}^2  + \kappa_k\Big\|\frac{\kappa_k t}{r^2}\eta_\ell\Big\|_{\LtT (\Ltr)}^2  \Big).
\end{align*}
By summarizing the above inequalities for $j\geq0$, and using  \eqref{step 3 final} and \eqref{6.18,a}, we achieve
\begin{align*}
    \|r\pa_r \eta\|_{Y_T}^2  &\lesssim \|r^2\pa_r \eta(0)\|_\Ltr^2+ \frac{\kappa_k}{\log \hc}  \big(\|r\pa_r \eta\|_{\LtT(\Ltr)}^2 +\| \eta\|_{\LtT(\Ltr)}^2\big) + \|\eta\|_{Y_T}^2 \\
    &\lesssim \|r^2\pa_r \eta(0)\|_\Ltr^2+ \frac{1}{\log \hc}   \|r\pa_r \eta\|_{Y_T}^2 + \|\eta\|_{Y_T}^2.
\end{align*}
 So that by choosing $\hc \gg1$ and then  applying Proposition \ref{estimate of homogeneous eqs without log loss}, we obtain
  \eqref{eqs-with log weight,derivative}. This completes the proof of Propsition \ref{prop estimate of pa_r w^(1)}.
\end{proof}

The next Proposition concerns  the space-time estimates for the stream function $\pkh$, which is also  referred to as the inviscid damping estimate.
\begin{proposition}\label{estimate of homogeneous eqs-inviscid damping}
{\sl Let $k\in \Z\backslash\{0\}$ and $\pkh$ solve the homogeneous equation \eqref{homogenous eq}, we denote
$\pkh^{(1)}\eqdefa e^{-i\f{kB}{r^2}t}\pkh.$ Then there exist constants $C_\varepsilon>0, \hc_0\gg1$  depending only on $\varepsilon$, so that for $\hc\geq \hc_0$,
    \begin{equation}\label{S6eq2}
    \begin{split}
     & |kB|^{\f12} |k|^\f12 \Big(\Big\|r^{-1+\varepsilon}\pa_r\pkh^{(1)}\La_k \Big\|_{\LtT (\Ltr)} + |k|\Big\| r^{-2+\varepsilon}\pkh^{(1)}\La_k \Big\|_{\LtT (\Ltr)}\Big) \leq  C_{\varepsilon}M_k(0).
    \end{split}
    \end{equation}}
\end{proposition}
\begin{proof}Without loss of generality, we assume that $w_k^{(1)}\in \mathcal{C}^\oo_c([1,\oo))$.
Observing that  $\frac{d}{ds}e^{-i\frac{kB}{s^2}t} = \frac{2ikBt}{s^3}e^{-i\frac{kB}{s^2}t},$
we get, by using integration by parts, that
\begin{equation}
    \begin{split}
-r^{-2+\varepsilon} \La_k(t,r) \pkh^{(1)} &= \int_1^\oo r^{-2+\varepsilon} \La_k(t,r) K_k(r,s) e^{-i\frac{kB}{s^2}t} w_k^{(1)} ds\\
&= \frac{1}{2ikBt} \int_1^\oo r^{-2+\varepsilon} \La_k(t,r) K_k(r,s)  s^3 w_k^{(1)} d\big(e^{-i\frac{kB}{s^2}t}\big)\\
& =-  \frac{1}{2ikBt} \int_1^\oo r^{-2+\varepsilon} \La_k(t,r) \pa_s \bigl(K_k(r,s)  s^3 w_k^{(1)}\bigr) e^{-i\frac{kB}{s^2}t} ds,
    \end{split}
\end{equation}
and by \eqref{1.17} that
\begin{equation}
    \begin{split}
-r^{-1+\varepsilon} \La_k(t,r) \pa_r \pkh^{(1)} &= \int_1^\oo r^{-1+\varepsilon} \La_k(t,r) \pa_r K_k(r,s) e^{-i\frac{kB}{s^2}t} w_k^{(1)} ds\\
&= \frac{1}{2ikBt} \int_1^\oo r^{-1+\varepsilon} \La_k(t,r) \pa_r K_k(r,s)  s^3 w_k^{(1)} d\big(e^{-i\frac{kB}{s^2}t}\big)\\
&= - \frac{1}{2ikBt} r^{-1+\varepsilon} \La_k(t,r)  r^3 w_k^{(1)}(r) e^{-i\frac{kB}{r^2}t}\\
& \quad -  \frac{1}{2ikBt} \int_1^\oo r^{-1+\varepsilon} \La_k(t,r) \widetilde{K}_k(r,s) s^3 w_k^{(1)} e^{-i\frac{kB}{s^2}t} ds\\
&\quad - \frac{1}{2ikBt} \int_1^\oo r^{-1+\varepsilon} \La_k(t,r) \pa_r K_k(r,s) \pa_s\bigl( s^3 w_k^{(1)}\bigr) e^{-i\frac{kB}{s^2}t} ds.
    \end{split}
\end{equation}
Then we deduce from  \eqref{B.10} and  \eqref{estimate of cK}, that
\begin{align*}
&\big|r^{-2+\varepsilon} \La_k(t,r) \pkh^{(1)}\big| + |k|^{-1} \big|r^{-1+\varepsilon} \La_k(t,r) \pa_r \pkh^{(1)}\big| -\frac{1}{2|k||kB|t} r^{2+\varepsilon} \La_k(t,r) \big|w_k^{(1)}\big|\\
&\leq \frac{C}{|kB|t}  \int_1^\oo r^{-2+\varepsilon} \La_k(t,r) \big(\frac{r}{s}\big)^\f12 \min\big(\frac{r}{s},\frac{s}{r}\big)^{|k|} s^3 \Big(\big|w_k^{(1)}\big|+ s\big|\pa_s w_k^{(1)}\big|\Big) ds\\
& = \frac{C}{|kB|t}  \int_1^\oo (rs)^{-\f12} \big(\frac{s}{r}\big)^{1-\varepsilon} \frac{\La_k(t,r)}{\La_k(t,s)} \min\big(\frac{r}{s},\frac{s}{r}\big)^{|k|}  s^{2+\varepsilon}\La_k(t,s)\Big(\big|w_k^{(1)}\big|+ s\big|\pa_s w_k^{(1)}\big|\Big) ds\\
&\leq \frac{C}{|kB|t}  \int_1^\oo (rs)^{-\f12} \big(\frac{s}{r}\big)^{1-\varepsilon} \max\Big(1,\big(\frac{r}{s}\big)^{\frac{2-\varepsilon}{2}}\Big) \min\big(\frac{r}{s},\frac{s}{r}\big)^{|k|}  s^{2+\varepsilon}\La_k(t,s)\Big(\big|w_k^{(1)}\big|+ s\big|\pa_s w_k^{(1)}\big|\Big)ds \\
&\leq \frac{C}{|kB|t} \cQ_{|k|-1+\varepsilon, |k|-\frac{\varepsilon}{2}}\Big[ s^{2+\varepsilon}\La_k(t,s)\Big(\big|w_k^{(1)}\big|+ s\big|\pa_s w_k^{(1)}\big|\Big)\Big],
\end{align*}
from which,  \eqref{ineq estimate of Tab, L2} and Propositions \ref{estimate of homogeneous eqs without log loss}
 and \ref{prop estimate of pa_r w^(1)}, we infer
\begin{align*}
   &|k|^{-1}\big\| r^{-1+\varepsilon}  \pa_r \pkh^{(1)}\La_k\big\|_\Ltr +  \big\|r^{-2+\varepsilon}  \pkh^{(1)}\La_k\big\|_\Ltr \\
   &\leq \frac{C_{\varepsilon}}{|kBt|} |k|^{-1} \Bigl(   \big\|r^{2+\varepsilon} w_k^{(1)}\La_k\big\|_{\LoT(\Ltr)}+\big\|r^{3+\varepsilon}  \partial_r w_k^{(1)}\La_k\big\|_{\LoT\Ltr}  \Bigr) \\
   &\leq \frac{C_{\varepsilon}}{|kBt|} |k|^{-1} \Bigl(   \big\|r^{2+\varepsilon}  w_k^{(1)}(0)\La_k(0)\big\|_{ \Ltr}+\big\|r^{3+\varepsilon}  \partial_r w_k^{(1)}(0)\La_k(0)\big\|_{\Ltr}  \Bigr).
\end{align*}
Whereas we deduce from  \eqref{basic estimate of K with space-time weight} and Proposition \ref{estimate of homogeneous eqs without log loss}
that
\begin{align*}
&|k|\big\|r^{-1+\varepsilon} \pa_r \pkh^{(1)}\La_k\big\|_\Ltr + |k|^2\big\|r^{-2+\varepsilon} \pkh^{(1)}\La_k\big\|_\Ltr \\
&\leq C_{\varepsilon} \big\|r^{\varepsilon}  w_k^{(1)}\La_k\big\|_{\LoT(\Ltr)} \leq C_{\varepsilon} \big\|r^{\varepsilon} w_k^{(1)}(0)\La_k(0)\big\|_\Ltr .
\end{align*}
As a consequence, we arrive at
\begin{equation*}
\begin{split}
    &\Big\|r^{-1+\varepsilon}\pa_r\pkh^{(1)}\La_k(t) \Big\|_{\Ltr}^2 + |k|^2\Big\| r^{-2+\varepsilon}\pkh^{(1)}\La_k(t) \Big\|_{\Ltr}^2 \\
    &\leq C_{\varepsilon} \min\Big(\frac{1}{|k|^2},\frac{1}{|kBt|^2}\Big)\Bigl(   \big\|r^{2+\varepsilon} w_k^{(1)}(0)\La_k(0)\big\|_{ \Ltr}^2+\big\|r^{3+\varepsilon} \partial_r w_k^{(1)}(0)\La_k(0)\big\|_{\Ltr}^2  \Bigr).
    \end{split}
\end{equation*}
By integrating the above inequality over $[0,T]$, we  complete the proof of Proposition \ref{estimate of homogeneous eqs-inviscid damping}.
\end{proof}

\section{Proof of Theorem \ref{thm1} and Theorem \ref{cor 1}.}\label{Sect7}

In this section, we present the proof of our main theorems. We begin by revisiting the fully nonlinear perturbation equations \eqref{nonlinear equation of w_0} and \eqref{nonlinear equation of w_k}
\begin{align}
\label{fully nonlinear system}\left\{
\begin{aligned}
&\partial_t w_0+\mathcal{T}_0w_0=r^{-\f12}\partial_r\Big(\sum_{\ell \in\mathbb{Z}\backslash\{0\}}i\ell r^{-1}\varphi_\ell w_{-\ell}\Big), \\
&\partial_t w_k+\mathcal{T}_kw_k = f_1^{(1)} + f_1^{(2)}+f_1^{(3)}+f_2^{(1)}+ f_2^{(2)},
\end{aligned}
\right.
\end{align}
where
\beq\label{S7eq1}
\begin{split}
    &f_1^{(1)}{\eqdefa} -ikr^{-1}\sum_{\ell \in\Z\setminus \{ 0, k\}}\partial_r\big(r^{-\f12}\varphi_\ell\big)w_{k-\ell},  \quad
    f_1^{(2)}{\eqdefa} -ikr^{-\f32} u^\theta_0 w_k,\\
    &f_1^{(3)}{\eqdefa}  -ikr^{-1}\partial_r\big(r^{-\f12}\varphi_k\big)w_{0},
    \quad f_2^{(1)} {\eqdefa} ir^{-\f12}\sum_{\ell \in\Z\setminus\{0, k\}}\ell \partial_r \big( r^{-1}\varphi_\ell w_{k-\ell}\big),\\
   & f_2^{(2)} {\eqdefa} ir^{-\f12} \partial_r \big( kr^{-1}\varphi_kw_0\big).
\end{split}
\eeq

\begin{lemma}\label{lemma F_0}
   {\sl  Let $w$ be a smooth enough solution to \eqref{1.10}, and $\mathcal{M}(0), E_0(T)$ and $\mathcal{E}(T)$
    be the energy functionals defined respectively by \eqref{S1eq1} and
    \eqref{S1eq2}. Then there exists a constant $C$ independent of $\varepsilon, \nu, B,T$, such that
\begin{equation}\label{S7eq1}
    \begin{split}
 E_0(T) \leq C  \mathcal{M}(0) + \frac{1}{2}  \mathcal{E}(T) +   C|\nu B|^{-\f12} \mathcal{E}^2(T).
    \end{split}
\end{equation}}
\end{lemma}
\begin{proof}
We first deduce from  Lemma \ref{velocity, basic energy conservation equality} that
\begin{equation*}\begin{split}
    &\|u^\theta_0\|_{\LoT(\Ltr)}^2 + 2\nu\|w_0\|_{\LtT(\Ltr)}^2 \\
    &\leq \|U\|_{\LoT(\LtO)}^2 + 2\nu \|W\|_{\LtT(\LtO)}^2 \\
    &\leq 2\Big(\|U(0)\|_\LtO^2 +8\pi |B|\sum_{k\in \Z\backslash\{0\}} |k|\Big\|\frac{\varphi_k}{r^2}\Big\|_{\LtT(\Ltr)} \Big\|\frac{\pa_r\varphi_{-k}}{r}\Big\|_{\LtT (\Ltr)} \Big) \\
&\leq 2\|U(0)\|_\LtO^2 +16\pi \sum_{k\in \Z\backslash\{0\}} E_k(T) E_{-k}(T) ,
    \end{split}
\end{equation*}
which implies
\begin{equation}\label{6.8,a}
    \|u^\theta_0\|_{\LoT(\Ltr)} + \sqrt{2}\nu^{\f12}\|w_0\|_{\LtT(\Ltr)} \leq \sqrt{2}\times \Big( \sqrt{2}\|U(0)\|_\LtO +\sqrt{16\pi} \Big(\sum_{k\neq0} E_k(T) E_{-k}(T)\Big)^\f12\Big).
\end{equation}

On the other hand, we get, by taking $\Ltr$ inner product of \eqref{nonlinear equation of w_0} with $r^2 w_0$ and  using integration by parts, that
\begin{align*}
 &\frac{1}{2}\frac{d}{dt}\big\|rw_0(t)\big\|_\Ltr^2 + \nu\big\|r \pa_r w_0\big\|_\Ltr^2 -\frac{5\nu}{4} \big\| w_0\big\|_\Ltr^2\\
 & \leq \Big|\big\langle\sum_{k\in\mathbb{Z}\backslash\{0\}}ikr^{-\f12}\varphi_kw_{-k}, \ r^{-\f12}\pa_r\big(r^{\f32}w_0\big)\big\rangle_\Ltr\Big|\\
 &\leq   \Big\|r^{-\f12}\sum_{k\in\Z\backslash\{0\}} k\varphi_kw_{-k} \Big\|_{\Ltr} \big(\big\|r \pa_r  w_0\big\|_\Ltr+ 3/2\big\|  w_0\big\|_\Ltr\big)\\
& \leq \nu^{-1}\Big\|r^{-\f12}\sum_{k\in\Z\backslash\{0\}} k\varphi_k w_{-k} \Big\|_{\Ltr}^2 + \frac{\nu}{2} \big\|r \pa_r  w_0\big\|_\Ltr^2 +\frac{9\nu}{8} \big\|  w_0\big\|_\Ltr^2.
\end{align*}
By integrating the above inequality over $[0,T]$, one has
\begin{equation*}
    \begin{split}
 &\big\|rw_0\big\|_{\LoT(\Ltr)}^2 + \nu\big\|r \pa_r w_0\big\|_{\LtT(\Ltr)}^2 \\
& \leq 2\times\Big( \|rw_0(0)\big\|_{\Ltr}^2+ 2\nu^{-1}\Big\|r^{-\f12}\sum_{k\in\Z\backslash\{0\}} k\varphi_k w_{-k} \Big\|_{\LtT(\Ltr)}^2 +\frac{19\nu}{4} \big\|  w_0\big\|_{\LtT(\Ltr)}^2\Big),
 \end{split}
\end{equation*}
from which and the fact:
\begin{equation*}
    x+y \leq \sqrt{2}(a+b+c),\quad \text{if} \ x^2+y^2 \leq a^2 +b^2+c^2, \quad \text{for }\ x,y,a,b,c\geq 0,
\end{equation*}
we infer
\begin{align*}
&\big\|rw_0\big\|_{\LoT(\Ltr)} + \nu^\f12 \big\|r \pa_r w_0\big\|_{\LtT(\Ltr)} \\
& \leq \sqrt{2}\times\Big( \sqrt{2}\|rw_0(0)\big\|_{\Ltr}+ 2\nu^{-\f12}\Big\|r^{-\f12}\sum_{k\neq0} k\varphi_k w_{-k} \Big\|_{\LtT(\Ltr)} +\sqrt{\frac{19}{2}} \nu^{\f12} \big\|  w_0\big\|_{\LtT(\Ltr)}\Big)\\
& \leq 2 \|rw_0(0)\big\|_{\Ltr}+ 2\sqrt{2}\nu^{-\f12}\sum_{k\neq 0}\big|k|\|r^{-\f32} \varphi_k\big\|_{\LtT(\Lor)} \big\|r w_{-k} \big\|_{\LoT(\Ltr)} +\sqrt{19} \nu^{\f12} \big\|  w_0\big\|_{\LtT(\Ltr)}.
\end{align*}
Then along the same line to the derivation of \eqref{6.8,a}, we deduce that
\begin{equation}\label{6.9,a}
    \begin{split}
    \|rw_0\|_{\LoT( \Ltr)} +&\nu^\f12 \|r \pa_r w_0\|_{\LtT (\Ltr)}
       \leq   C\|rw_0(0)\|_\Ltr + C\|U(0)\|_\LtO +\\
       &  C|\nu B|^{-\f12} \sum_{k\neq 0} E_k(T)E_{-k}(T) +\sqrt{19}\times\sqrt{16\pi} \Big(\sum_{k\neq0} E_k(T) E_{-k}(T)\Big)^\f12.
    \end{split}
\end{equation}

In view of \eqref{S1eq1},
 we get, by summing up \eqref{6.8,a} with \eqref{6.9,a}, that
\begin{align*}
E_0(T)  \leq& C\big(\|U(0)\|_\LtO+\|rw_0(0)\|_\Ltr\big)  + C|\nu B|^{-\f12} \sum_{k\neq0} E_k(T)E_{-k}(T)\\
&\quad + (\sqrt{2}+\sqrt{19})\sqrt{16\pi} \Big(\sum_{k\neq0} E_k(T) E_{-k}(T)\Big)^\f12\\
   \leq& C  \mathcal{M}(0) + C|\nu B|^{-\f12} \mathcal{E}^2(T) + \frac{(\sqrt{2}+\sqrt{19})\sqrt{16\pi}}{100}\mathcal{E}(T)\\
   \leq& C  \mathcal{M}(0) + \frac{1}{2}  \mathcal{E}(T) +   C|\nu B|^{-\f12} \mathcal{E}^2(T),
\end{align*}
which leads to \eqref{S7eq1}.
\end{proof}

\begin{lemma}\label{basic property of k}
{\sl For any $k, \ell \in\mathbb{Z}$ and any $ \alpha\in (0, 1]$, there hold
\begin{align}
\label{9.2-1}|k|^{\alpha}\leq|\ell|^{\alpha}+|k-\ell|^{\alpha}, \quad \log(|k|+|\ell|+e) \leq \log(|k|+e)\log(|\ell|+e).
\end{align}
For any $\ell \in\mathbb{Z}$ and $\ell \neq 0,k$, there holds
\begin{align}
\label{9.2-2}|k|\leq 2|\ell| |k-\ell|.
\end{align}
Moreover,  for any $\hat{C}\geq e$,one has
\begin{align}
\label{9.2-3}\La_k \leq \La_\ell\La_{k-\ell}.
\end{align}}
\end{lemma}
\begin{proof}
\eqref{9.2-2} and the first inequality in \eqref{9.2-1} are obvious. Notice that
\begin{equation*}
    \log(a+x+y)\leq \log(a+x)\log(a+y),\quad \text{for} \ x,y\geq 0,\  a\geq e,
\end{equation*}
which implies the second inequality in \eqref{9.2-1} and for $\hc\geq e$,
\begin{align*}
    \Lambda_k=&\log\big(\hat{C}r^2 + \nu^\f13|kB|^\f23 t\big)\leq\log\Big(\hat{C}r^2 +  \nu^\f13|B|^\f23|\ell|^{\f23}t+\nu^\f13|B|^\f23|k-\ell|^{\f23} t\Big)\\
    \leq&\log\big(\hat{C}r^2 +  \nu^\f13|B|^\f23|\ell|^{\f23}t\big)\log\big(\hat{C}r^2 +  \nu^\f13|B|^\f23|k-\ell|^{\f23} t\big)= \La_\ell\La_{k-\ell},
\end{align*}
which leads to \eqref{9.2-3}. We thus complete the proof of Lemma \ref{basic property of k}.
\end{proof}

\begin{lemma}\label{f1 f2} {\sl Let $k\in \mathbb{Z}\backslash\{0 \}$ and $f_i^{(j)}, i=1,2,$ $j=1,2,3,$ be given by \eqref{S7eq1}.
 Then there exists a constant $C_{\varepsilon}$  depending only on $\varepsilon$, such that
 \beq \label{S7eq2}
\begin{split}
&\mu_k^{-\f12} \Big(\big\|r^{2+\varepsilon} f_1^{(1)} \La_k\big\|_{\LtT (\Ltr)} +\big\|r^{2+\varepsilon} f_1^{(2)} \La_k\big\|_{\LtT (\Ltr)}+ \big\|r^{2+\varepsilon} f_1^{(3)} \La_k\big\|_{\LtT (\Ltr)}\Big)\\
&\quad+\nu^{-\f12} \Big(\big\|r^{1+\varepsilon} f_2^{(1)} \La_k\big\|_{\LtT (\hkf)}+\big\|r^{1+\varepsilon} f_2^{(2)} \La_k\big\|_{\LtT(\hkf)} \Big)\\
&\leq C_\varepsilon |\nu B|^{-\f12} \sum_{\ell \in\mathbb{Z} } E_\ell(T) E_{k-\ell}(T).
\end{split}\eeq}
\end{lemma}
\begin{proof} Recall that
$\mu_k=\max\bigl(\nu|k|^2, \nu^{\f13}|kB|^{\f23}\bigr),$ we deduce from Lemma  \ref{basic property of k} that
\begin{align*}
  \La_k \lesssim \La_\ell\La_{k-\ell} \ (\forall k,\ell \in \mathbb{Z}),\quad |k|^\f12\lesssim |\ell|^\f12 |k-\ell|^\f12 \text{ (for any $\ell\neq 0,k$)},\quad  \mu_k^{-\f12} \leq |k|^{-\f12} |\nu B|^{-\f14},
\end{align*}
from which, \eqref{S7eq1} and  \eqref{basic estimate of K with space-time weight}, we infer
\begin{align*}
    \mu_k^{-\f12} \big\|r^{2+\varepsilon} f_1^{(1)} \La_k\big\|_{\LtT (\Ltr)} &\leq C |k|\mu_k^{-\f12} \sum_{\ell\in\Z\setminus\{0,k\}} {\|r\partial_r(r^{-\f12}\varphi_\ell) \La_\ell\|_{\LoT(\Lor)} } \|r^{\varepsilon}w_{k-\ell}\La_{k-\ell}\|_{\LtT (\Ltr)} \\
    & \leq  C |k|\mu_k^{-\f12} \sum_{\ell\in\Z\setminus\{0,k\}}|\ell|^{-\f12} \|r w_\ell\La_\ell\|_{\LoT(\Ltr)} \|r^{\varepsilon}w_{k-\ell}\La_{k-\ell}\|_{\LtT (\Ltr)} \\
    &\leq C|k|\mu_k^{-\f12} \sum_{\ell \neq k,0}|\ell|^{-\f12}\mu^{-\f12}_{k-\ell} E_\ell(T)  E_{k-\ell}(T) \\
    &\leq C |\nu B|^{-\f12} |k|^\f12 \sum_{\ell \neq k,0}|\ell|^{-\f12}|k-\ell|^{-\f12} E_\ell(T)  E_{k-\ell}(T) \\
    &\leq C |\nu B|^{-\f12} \sum_{\ell\in\Z\setminus\{0,k\}}E_\ell(T)  E_{k-\ell}(T).
\end{align*}

While it follows from \eqref{w_0, u_0} and Lemma \ref{Appendix A1-1} that
\begin{equation}\label{ineq 8.2.a}
\big\|r^\f12u^\theta_0\big\|_{\LoT(\Lor)} \lesssim \big\|r \pa_r u^{\theta}_0\big\|_{\LoT(\Ltr)} + \big\|u^\theta_0\big\|_{\LoT(\Ltr)} \lesssim   \big\|r w_0\big\|_{\LoT(\Ltr)} + \big\|u^\theta_0\big\|_{\LoT(\Ltr)}.
\end{equation}
Then we deduce from \eqref{S7eq1} that
\begin{equation*}
    \begin{split}
\mu_k^{-\f12}\big\|r^{2+\varepsilon} f_1^{(2)} \La_k\big\|_{\LtT (\Ltr)} &\leq C|k|\mu_k^{-\f12} \|r^\f12 u^\theta_0\|_{\LoT(\Lor)} \|r^{\varepsilon}w_k  \La_k \|_{\LtT (\Ltr)} \\
&\leq C |k|\mu_k^{-\f12} E_0(T) \mu_k^{-\f12} E_k(T) \leq C |\nu B|^{-\f12} E_0(T) E_k(T).
    \end{split}
\end{equation*}

For $f_1^{(3)}$, we get, by using \eqref{basic estimate of K with space-time weight}, that
\begin{align*}
    \mu_k^{-\f12} \big\|r^{2+\varepsilon} f_1^{(3)} \La_k\big\|_{\LtT (\Ltr) }&\lesssim |k|\mu_k^{-\f12}\|r^{\varepsilon}\partial_r (r^{-\f12}\varphi_k)\La_k\|_{\LtT (\Lor)} \|rw_0\|_{\LoT(\Ltr)}\\
   &\lesssim_\varepsilon |k|\mu_k^{-\f12} {|k|^{-\f12}\|r^{\varepsilon}w_k \La_k \|_{\LtT (\Ltr)} }\|r w_0\|_{\LoT(\Ltr)} \\
   &\lesssim |k|^\f12 \mu_k^{-1}  E_0(T) E_k(T) \lesssim_\varepsilon  |\nu B|^{-\f12} E_0(T) E_k(T).
\end{align*}

Notice that
\begin{align*}
   r^{\f12 +\varepsilon}\partial_r F \La_k = \partial_r(r^{\f12 +\varepsilon} F \La_k) - F\partial_r(r^{\f12 +\varepsilon} \La_k),\quad \partial_r(r^{\f12 +\varepsilon}\La_k) \leq C r^{-\f12+\varepsilon} \La_k,
\end{align*}
where the constant $C$ is independent of $k$. Then we get, by applying \eqref{ineq 8.2}, that
\begin{align*}
\nu^{-\f12} &\Big(\|r^{1+\varepsilon} f_2^{(1)} \La_k\|_{\LtT(\hkf)}+\big\|r^{1+\varepsilon} f_2^{(2)} \La_k\big\|_{\LtT(\hkf)} \Big) \\
\leq& \nu^{-\f12} \sum_{\ell\in\Z\setminus\{0,k\}} |\ell| \Big(\|\partial_r(r^{-\f12+\varepsilon} \varphi_\ell w_{k-\ell}\La_k)\|_{\LtT(\hkf)} + \|\varphi_\ell w_{k-\ell} r^{-1}\partial_r(r^{\f12+\varepsilon}\La_k) \|_{\LtT(\hkf)} \Big)\\
&+ \nu^{-\f12}|k|\Big(\|\partial_r(r^{-\f12+\varepsilon} \varphi_k w_{0}\La_k)\|_{\LtT(\hkf)} + \|\varphi_k w_{0} r^{-1}\partial_r(r^{\f12+\varepsilon}\La_k) \|_{\LtT(\hkf)} \Big)\\
\lesssim& \nu^{-\f12} \sum_{\ell\in\Z\setminus\{0,k\}} |\ell| \|r^{-\f12+\varepsilon} \varphi_\ell w_{k-\ell}\La_k\|_{\LtT(\Ltr)} +\nu^{-\f12}  |k| \|r^{-\f12+\varepsilon} \varphi_k w_{0}\La_k\|_{\LtT(\Ltr)} \\
\lesssim &\nu^{-\f12} \sum_{\ell\in\Z\setminus\{0,k\}} |\ell|\|r^{-\f32}\varphi_\ell \La_\ell\|_{\LtT (\Lor)}  \|r^{1+\varepsilon}w_{k-\ell}\La_{k-\ell}\|_{\LoT(\Ltr)} \\
&+ \nu^{-\f12}|k| \|r^{\varepsilon-\f32}\varphi_k\La_k\|_{\LtT (\Lor)} \|rw_0\|_{\LoT(\Ltr)}\\
\lesssim& |\nu B|^{-\f12} \sum_{\ell \in\mathbb{Z} } E_\ell(T) E_{k-\ell}(T).
\end{align*}
    This completes the proof of Lemma \ref{f1 f2}.
\end{proof}

Now
we are in a position to complete   the proof of Theorem \ref{thm1} and Theorem \ref{cor 1}.

\begin{proof}[\textbf{Proof of Theorem \ref{thm1}}]
 Due to  the local well-posedness  established in Theorem \ref{thm0},
here we focus on the {\it a priori} estimates. We first deduce from Proposition \ref{nonlinear space-time estimates} and Lemma \ref{f1 f2} that for $k\neq 0$,
\begin{align*}
 E_k(T) &\lesssim_\varepsilon      M_k(0)+   \mu_k^{-\f12}\Big(\big\|r^{2+\varepsilon} f_1^{(1)} \La_k\big\|_{\LtT (\Ltr)} +\big\|r^{2+\varepsilon} f_1^{(2)} \La_k\big\|_{\LtT (\Ltr)} + \big\|r^{2+\varepsilon} f_1^{(3)} \La_k\big\|_{\LtT (\Ltr)}\Big) \\
    & \quad+  \nu^{-\f12} \Big(\|r^{1+\varepsilon} f_2^{(1)} \La_k\|_{\LtT(\hkf)}+\big\|r^{1+\varepsilon} f_2^{(2)} \La_k\big\|_{\LtT(\hkf)} \Big) \\
    &\lesssim_\varepsilon   M_k(0)+|\nu B|^{-\f12} \sum_{\ell \in\mathbb{Z} }  E_\ell(T)   E_{k-\ell}(T).
\end{align*}
While it follows from  Lemma \ref{lemma F_0} that
\begin{equation*}
    E_0(T) \leq C  \mathcal{M}(0) + \frac{1}{2}  \mathcal{E}(T) +   C|\nu B|^{-\f12} \mathcal{E}^2(T).
\end{equation*}
By summing up the above inequalities for $k\in\Z$ and applying Young's inequality, we obtain
\begin{equation}\label{S7eq3}
    \mathcal{E}(T)  \leq C_\varepsilon  \mathcal{M}(0)+ \frac{1}{2} \mathcal{E}(T) + C_\varepsilon  |\nu B|^{-\f12} \mathcal{E}^2(T).
\end{equation}
Therefore, as long as
\begin{align*}
 \mathcal{M}(0) \leq c_1(\varepsilon)\nu^{\f12} |B|^{\f12}  \text{ for some sufficiently small $c_1(\varepsilon)$ depending only on $\varepsilon$},
\end{align*}
 we deduce from \eqref{S7eq3} that
\begin{equation*}
    \mathcal{E}(T) \leq C_2(\varepsilon) \mathcal{M}(0) \text{ for any $T>0$ and some $C_2(\varepsilon) $ depending only on $\varepsilon$}.
\end{equation*}
This completes the proof of Theorem \ref{thm1}.
\end{proof}

\begin{proof}[\textbf{Proof of Theorem \ref{cor 1}}]
    By Theorem \ref{thm1}, we have
    \begin{equation*}
\log\big(\hc +\nu^\f13 |kB|^\f23 t\big)
 \|r^{1+\varepsilon} w_k\|_\Ltr \leq \|r^{1+\varepsilon} w_k \La_k\|_\Ltr \lesssim_\varepsilon \mathcal{M}(0),
    \end{equation*}
    which leads to \eqref{ineq 1.19}.

While we observe from Lemma \ref{vorticity conservation of L2}  that $\|W(t)\|_\LtO$ is decreasing with respect to time $t.$    Furthermore, Theorem \ref{thm1} implies $\|W\|_{\LtT(\LtO)}<\oo$, so that we deduce that
\begin{equation*}
    \lim_{t\rightarrow\oo} \bigl(t^\f12 \|W(t)\|_\LtO\bigr)\leq  \sqrt{2}\lim_{t\rightarrow\oo}\Bigl(\int_{t/2}^{t}\|W(s)\|_\LtO^2 \,ds\Bigr)^{\f12}=0.
\end{equation*}
This completes the proof of Theorem \ref{cor 1}.
\end{proof}

\appendix

\section{Local wellposedness of \eqref{velocity pertubation of the Taylor-Couette flow}}\label{Appc}

 In this section, we present the result concerning the local well-posedness of the system \eqref{velocity pertubation of the Taylor-Couette flow}.

\begin{lemma}
    Let $\varepsilon\in(0,2)$ and $T_0\geq 0$. There holds
\begin{equation}\label{H1 estimate}
    \|U\|_{H^1(\Omega)} \lesssim \|U\|_{\mathscr{M}(T_0)}\eqdefa\|rw_0\|_\Ltr + \|u^\theta_0\|_\Ltr+\sum_{k\neq0}  \|r^{1+\varepsilon} w_k \La_k(T_0,\cdot)\|_\Ltr.
\end{equation}
\end{lemma}
\begin{proof}
  In view of   \eqref{expression of phi_k} and $u^r_0=0$, we get, by
   using Lemma \ref{basic properties of ck}, that
\begin{equation*}
    \begin{split}
\|U\|_{\LtO}^2 &\lesssim \|u^\theta_0\|_\Ltr^2 + \sum_{k\neq0} \big(\|u^\theta_k\|_\Ltr^2+\|u^r_k\|_\Ltr^2\big) \\
&\lesssim \|u^\theta_0\|_\Ltr^2 + \sum_{k\neq0} \Big(\|\pa_r \varphi_k\|_\Ltr^2+\Big\|\frac{k}{r}\varphi_k\Big\|_\Ltr^2\Big)\quad
[\mbox{by using}\ \eqref{expression of u}]\\
&\lesssim \|u^\theta_0\|_\Ltr^2 + \sum_{k\neq0} |k|^{-2}\|rw_k\|_\Ltr^2 \lesssim \|U\|_{\mathscr{M}(T_0)}^2.
    \end{split}
\end{equation*}
Along the same line, we deduce that
\begin{equation*}
    \begin{split}
&\|\nabla U\|_{\LtO}^2 \lesssim  \|\pa_r (U^r \mathbf{e}_r + U^\theta \mathbf{e}_\theta) \|_\LtO^2 +\|r^{-1}\pa_\theta(U^r \mathbf{e}_r + U^\theta \mathbf{e}_\theta) \|_\LtO^2\\
    &\lesssim \|\pa_r U^{r} \|_\LtO^2 +\|\pa_r U^{\theta} \|_\LtO^2  \\
    &\quad +\|r^{-1}\pa_\theta U^{r} \|_\LtO^2 +\|r^{-1}\pa_\theta U^{\theta} \|_\LtO^2 +\|r^{-1} U^r \|_\LtO^2+ \|r^{-1} U^{\theta} \|_\LtO^2\\
    &\lesssim \|\pa_r u^\theta_0\|_\Ltr^2 + \|r^{-1}u^\theta_0\|_\Ltr^2 +\sum_{k\neq0}\big(\|\pa_r u^\theta_k\|_\Ltr^2 + \|kr^{-1}u^\theta_k\|_\Ltr^2\big) + \sum_{k\neq0}\big(\|\pa_r u^r_k\|_\Ltr^2 + \|kr^{-1}u^r_k\|_\Ltr^2\big)\\
    &\lesssim \|w_0\|_\Ltr^2 +\|u^\theta_0\|_\Ltr^2 + \sum_{k\neq0} \big( \|\pa_r^2 \varphi_k\|_\Ltr^2 +\|kr^{-1}\pa_r \varphi_k\|_\Ltr^2 +\|k^2r^{-2} \varphi_k\|_\Ltr^2  \big)\\
    &\lesssim \|w_0\|_\Ltr^2 +\|u^\theta_0\|_\Ltr^2 + \sum_{k\neq0} \big( \|w_k\|_\Ltr^2 +\|kr^{-1}\pa_r \varphi_k\|_\Ltr^2 +\|k^2r^{-2} \varphi_k\|_\Ltr^2  \big)\\
    &\lesssim \|w_0\|_\Ltr^2 +\|u^\theta_0\|_\Ltr^2 + \sum_{k\neq0} \big( \|w_k\|_\Ltr^2 +\|k\pa_r \varphi_k\|_\Ltr^2 +\|k^2r^{-1} \varphi_k\|_\Ltr^2  \big)\\
    &\lesssim \|w_0\|_\Ltr^2 +\|u^\theta_0\|_\Ltr^2 + \sum_{k\neq0} \big( \|w_k\|_\Ltr^2 +\|rw_k\|_\Ltr^2  \big)\lesssim \|U\|_{\mathscr{M}(T_0)}^2,
    \end{split}
\end{equation*}
which completes the proof of \eqref{H1 estimate}.
\end{proof}

Now the main result states as follows:

\begin{theorem}\label{thm0}
  {\sl  Let $0<\varepsilon<2$, $\hc>\max\big(e,e^{\frac{4}{2-\varepsilon}}\big)$ and $T_0\geq 0$. Then under the assumption that
    \begin{align}\label{S1eq4}
\|U(0)\|_{\mathscr{M}(T_0)}=\|rw_0(0)\|_\Ltr + \|u^\theta_0(0)\|_\Ltr+\sum_{k\neq0}  \|r^{1+\varepsilon} w_k(0) \La_k(T_0,\cdot)\|_\Ltr <\oo,
    \end{align}
     there exists $T>0$ depending on $\varepsilon$, $A,B$, $\nu$, $T_0$ and $\|U(0)\|_{\mathscr{M}(T_0)}$, so that the
     system \eqref{velocity pertubation of the Taylor-Couette flow} has a unique solution $U\in C\big([0,T];\mathscr{M}(T_0)\big)$.
Moreover, we have
\begin{equation}\label{S1eq5}
    \begin{split}
   \|U\|_{N_{T}}\eqdefa &\|rw_0\|_{\LoT(\Ltr)} + \| u^\theta_0\|_{\LoT(\Ltr)} + \nu^\f12 \|w_0\|_{\LtT (\Ltr)} + \sum_{k\neq0}   \|r^{1+\varepsilon}w_k \La_k(T_0,r)\|_{\LoT(\Ltr)} \\
     & +  \nu^\f12 \sum_{k\neq0}   \Big(\|r^{1+\varepsilon}\pa_r w_k \La_k(T_0,r)\|_{\LtT(\Ltr)}+ |k|  \|r^{\varepsilon}w_k \La_k(T_0,r)\|_{\LtT (\Ltr)}\Big)  <\oo.
    \end{split}
\end{equation}}
\end{theorem}

The proof of the above theorem is through standard energy estimate and iteration scheme. We omit the details here.

\section{Basic Calculus}\label{Sappa}

In this section, we present some inequalities which are based on calculus.

\begin{lemma}\label{Coercive estimates prior}
{\sl Let $k \in \R$, $w\in \mathcal{C}^\oo_c([1,\oo))$ with $w(1)=0$,  and  $A(r)\in \mathcal{C}^2\big((1,\oo)\big)$ be any real-valued function. Then, one has
  \begin{align*}
&\Re\big\langle \mathcal{T}_kw,A^2(r)w\big\rangle_\Ltr=\nu\|A(r)w'\|_{\Ltr}^2-\f\nu2 \big\langle\big(A^2(r)\big)'', |w|^2 \big\rangle_\Ltr  +\nu\Big(k^2-\f14\Big)\|A(r)r^{-1}w\|_{\Ltr}^2.
\end{align*}}
\end{lemma}
\begin{proof} Indeed we get, by using integration by parts, that
 \begin{align*}
\Re\big\langle \mathcal{T}_kw,A^2(r)w\big\rangle_\Ltr
=&\nu \|A(r)w'\|_{\Ltr}^2+\nu\Re\big\langle w',(A^2(r))'w\big\rangle_\Ltr+\nu\Big(k^2-\f14\Big)\|A(r)r^{-1}w\|_{\Ltr}^2\\
=&\nu\|A(r)w'\|_{\Ltr}^2-\f\nu2 \big\langle\big(A^2(r)\big)'', |w|^2 \big\rangle_\Ltr  +\nu\Big(k^2-\f14\Big)\|A(r)r^{-1}w\|_{\Ltr}^2.
\end{align*}
This completes the proof of Lemma \ref{Coercive estimates prior}.
\end{proof}

\begin{lemma}\label{pair-w' with w}
{\sl For any $\alpha \in \R$, $w\in \mathcal{C}^\oo_c([1,\oo))$ with $w(1)=0$, we have
\begin{gather}
    \|r^{\alpha}w'\|_{\Ltr}^2\geq (\alpha-1/2)^2\|r^{\alpha-1}w\|_{\Ltr}^2,  \label{5.2} 
\end{gather}}
\end{lemma}
\begin{proof} It is easy to observe that for any $\alpha, \beta \in \R$,
\begin{align*}
  0\leq   \|r^{\beta}(r^{\alpha-\beta}w)'\|_{\Ltr}^2=&\big\langle r^{\alpha}w'+(\alpha-\beta)r^{\alpha-1}w,r^{\alpha}w'+(\alpha-\beta)r^{\alpha-1}w\big\rangle_\Ltr\\
    =&\|r^{\alpha}w'\|_{\Ltr}^2+(\alpha-\beta)^2\|r^{\alpha-1}w\|_{\Ltr}^2+(\alpha-\beta)\int_1^{\infty}r^{2\alpha-1}d|w|^2\\
     =&\|r^{\alpha}w'\|_{\Ltr}^2+(\alpha-\beta)^2\|r^{\alpha-1}w\|_{\Ltr}^2-(2\alpha-1)(\alpha-\beta)\int_1^{\infty}r^{2\alpha-2}|w|^2dr\\
     =&\|r^{\alpha}w'\|_{\Ltr}^2-\big((\alpha-1/2)^2-(\beta-1/2)^2\big)\|r^{\alpha-1}w\|_{\Ltr}^2.
\end{align*}
Taking $\beta=\f12$ in the above inequality leads to  \eqref{5.2}.
\end{proof}

\begin{lemma}\label{Appendix A1-1}
{\sl Let  $\al \in \R$, $\hat{C}, a \geq 0$, $w\in \mathcal{C}^1([1,\oo))$, $A(r)=r^\al$ or $r^\al \log(\hat{C}r^2 + a).$ We assume that $\|A(r)r^{-\f12}w\|_{\Ltr} + \|A(r)r^{\f12}w'\|_{\Ltr} <\oo, $ then there exists a constant $C>0$ independent of $\al,\hc,a$, such that
\begin{align}\label{A.3,a}
&\|A(r)w\|_{\Lor}^2 \leq C \|A(r)r^{-\f12}w\|_{\Ltr}\bigl(\|A(r)r^{\f12}w'\|_{\Ltr} +(|\al|+1)\|A(r)r^{-\f12}w\|_{\Ltr}\bigr).
\end{align}
Moreover, there holds
\begin{equation}\label{A.4,a}
 \lim_{r\rightarrow+\oo} |A(r)w(r)| =0.
\end{equation}}
\end{lemma}
\begin{proof}
Let \( \chi(z) \in \mathcal{C}^\infty(\mathbb{R}_+) \) be a smooth function such that \( \chi(z) = 1 \) for \( z \in [0, 1] \) and \( \chi(z) = 0 \) for \( z \geq 2 \). We denote \( \chi_R(r) {\eqdefa} \chi\left( \frac{r-1}{R} \right) \in \mathcal{C}^\infty_c([1, \infty)) \) for \( R \geq 1 \). Observing that \( A'(r) \lesssim (|\alpha| + 1) \frac{A(r)}{r} \) and \( \partial_r \chi_R = \frac{1}{R} \chi'\left( \frac{r-1}{R} \right) \),
we infer
\begin{align*}
|A(r)\chi_Rw(r)|^2=&-\int_r^\oo\partial_s\big(A^2(s)\chi_R^2|w(s)|^2\big)ds\\
=&-\int_r^\oo A^2(s)\chi_R^2\big(w'(s)\overline{w}(s)+w(s)\overline{w}'(s)\big)ds\\
&- \int_r^\oo 2A(s)A'(s)\chi_R^2|w|^2 ds - \int_{R+1}^{2R+1} 2A(s)^2\chi_R \pa_s\chi_R|w|^2 ds \\
\lesssim &\|A(r)r^{-\f12}w\|_{\Ltr}\|A(r)r^{\f12}w'\|_{\Ltr} +(|\al|+1)\|A(r)r^{-\f12}w\|_{\Ltr}^2,
\end{align*}
which together with the fact
\begin{equation*}
  \|A(r)w\|_{\Lor}^2 =\sup_{R\geq 1} \|A(r)\chi_Rw\|_{\Lor}^2,
\end{equation*}
ensures \eqref{A.3,a}.

While we get, by applying \eqref{A.3,a}, that
\begin{equation*}
\begin{split}
    \|A(r)(1-\chi_R)w\|_{\Lor}^2
    &\lesssim  \|A(r)r^{-\f12}(1-\chi_R)w\|_{\Ltr}\bigl(\big\|A(r)r^{\f12}\big((1-\chi_R)w\big)'\big\|_{\Ltr}  \\ &\quad+(|\al|+1)\|A(r)r^{-\f12}(1-\chi_R)w\|_{\Ltr}\bigr)\\
    &\lesssim  \|A(r)r^{-\f12}(1-\chi_R)w\|_{\Ltr}\bigl(\|A(r)r^{\f12}(1-\chi_R)w'\|_{\Ltr} \\
    &\quad +\|A(r)r^{\f12}\pa_r\chi_Rw\|_{\Ltr}  +(|\al|+1)\|A(r)r^{-\f12}(1-\chi_R)w\|_{\Ltr}\bigr)\\
    &\lesssim  \|A(r)r^{-\f12}(1-\chi_R)w\|_{\Ltr}\bigl(\|A(r)r^{\f12}(1-\chi_R)w'\|_{\Ltr} \\
    &\quad +\|A(r)r^{-\f12}(1-\chi_{R/2})w\|_{\Ltr}  +(|\al|+1)\|A(r)r^{-\f12}(1-\chi_R)w\|_{\Ltr}\bigr).
\end{split}
\end{equation*}
Taking  \( R \to +\infty \) in the above inequality leads to \eqref{A.4,a}. We thus complete the
proof of Lemma \ref{Appendix A1-1}.
\end{proof}

\begin{lemma}\label{Appendix A1-2}
{\sl Let $\lambda>0$, $0<\tilde{\delta}\ll 1$ and $ E{\eqdefa} \big\{r\geq 1: \ |1-\sqrt{\lambda}r|\leq \tilde{\delta}  \big\}$, then there holds
\begin{align*}
\|r^{-\f12}\mathbbm{1}_E\|_\Ltr \lesssim
\tilde{\delta}^{\f12}.
\end{align*}}
\end{lemma}
\begin{proof} Let $r_0{\eqdefa} \lambda^{-\f12}$, then we get, by a direct computation, that
\begin{align*}
&\|r^{-\f12}\mathbbm{1}_E\|_\Ltr^2 = \int_{E\cap[1,\oo)} \frac{dr}{r} \leq   \int_{r_0-\tilde{\delta}r_0}^{r_0+\tilde{\delta}r_0} \frac{dr}{r} =\log \f{r_0+\tilde{\delta}r_0}{r_0-\tilde{\delta}r_0}=\log\Big(1+\f{2\tilde{\delta}}{1-\tilde{\delta}}\Big)
\leq\f{2\tilde{\delta}}{1-\tilde{\delta}}\lesssim \tilde{\delta}.
\end{align*}
This finishes the proof of Lemma \ref{Appendix A1-2}.
\end{proof}

\begin{lemma}\label{Appendix A1-3}
{\sl Let $\tilde{\delta} \leq \f14$, $\lambda>0$, $\beta=1$ or $2$, and
  let $E_m \eqdefa \big\{r\geq 1: \ \tilde{\delta}\leq |1-\sqrt{\lambda} r| \leq \f12 \big\}$. Then there exists a constant $C$ independent of $\tilde{\delta}$ and $\lambda$  such that
  \begin{align}
      \Big\|r^{-\frac{1}{2}} \frac{\mathbbm{1}_{E_m}(r)}{(1-\lambda r^2)^{\beta}} \Big\|_\Ltr \leq C \tilde{\delta}^{-\beta+\f12},\label{5.4}
  \end{align}}
\end{lemma}

\begin{proof}
By using the changes of variable: $\tilde{r}= \sqrt{\lambda} r$, we find for $\beta=1$ or $2$,
\begin{equation*}
    \begin{split}
\Big\|r^{-\frac{1}{2}} \frac{\mathbbm{1}_{E_m}}{(1-\lambda r^2)^\beta} \Big\|_\Ltr^2
&= \int_1^\oo r^{-1}\frac{\mathbbm{1}_{E_m}}{(1-\lambda r^2)^{2\beta} } dr \leq \int_{\tilde{\delta} \leq |1-\tilde{r}| \leq \f12} \tilde{r}^{-1}\frac{d\tilde{r}}{(1-\tilde{r}^2)^{2\beta}} \\
&\lesssim  \int_{\tilde{\delta} \leq |1-\tilde{r}| \leq \f12} \frac{d\tilde{r}}{(1-\tilde{r})^{2\beta}} \lesssim \tilde{\delta}^{-2\beta+1},
    \end{split}
\end{equation*}
which leads to \eqref{5.4}, and we complete the proof of Lemma \ref{Appendix A1-3}.
\end{proof}

\section{Basic properties of $\mathcal{K}_k$ and $\La_k$}

In this section, we shall present some basic properties of the operator: $\cK,$ which is given by
\eqref{expression of phi_k}, and the enhanced dissipation weight: $\La_k=\log\bigl(\hat{C}r^2+\kappa_kt\bigr).$

\begin{lemma}[Basic properties of $\mathcal{K}_k$]\label{basic properties of ck}
   {\sl  For any $|k| \geq 1$, $a>0$, $0<\varepsilon < 2$ and $\hc \geq e^{\frac{4}{2-\varepsilon}}$, there exists a  constant $C_\varepsilon>0,$ which is  independent of $k,a,\hc$, such that
\begin{align}
         &|k|\Big\|r^{\varepsilon-1}\cK'[w]\Big\|_\Ltr +|k|^2 \Big\|r^{\varepsilon-2}\cK[w]\Big\|_\Ltr \label{basic estimate of K}  + |k|^\f12 \Big\|r^{\varepsilon-\f12}\cK'[w] \Big\|_\Lor \\
         &\quad + |k|^\f32 \Big\|r^{\varepsilon-\f32}\cK[w] \Big\|_\Lor \leq C_\varepsilon \|r^\varepsilon w\|_\Ltr,\nonumber\\
         &|k|\Big\|r^{\varepsilon-1}\cK'[w] \Big\|_\Ltr +|k|^2 \Big\|r^{\varepsilon-2}\cK[w]\Big\|_\Ltr  \leq C_\varepsilon  |k|^\f12 \|r^{\varepsilon-\f12} w\|_\Llr,  \label{basic estimate of K, L1 type} \\
    &|k| \Big\|r^{\varepsilon-1}\log(\hat{C}r^2+a)\cK'[w]\Big\|_\Ltr +|k|^2 \Big\|r^{\varepsilon-2}\log(\hat{C}r^2+a)\cK[w]\Big\|_\Ltr   \label{basic estimate of K with space-time weight}\\
     &\quad + |k|^\f12 \Big\|r^{\varepsilon-\f12}\log(\hat{C}r^2+a)\cK'[w]\Big\|_\Lor + |k|^\f32 \Big\|r^{\varepsilon-\f32}\log(\hat{C}r^2+a) \cK[w]\Big\|_\Lor  \nonumber\\
     & \leq  C_\varepsilon  \|r^\varepsilon \log(\hat{C}r^2+a) w\|_\Ltr,\nonumber\\
    &|k| \Big\|r^{\varepsilon-1}\log(\hat{C}r^2+a)\cK'[w]\Big\|_\Ltr +|k|^2 \Big\|r^{\varepsilon-2}\log(\hat{C}r^2+a)\cK[w]\Big\|_\Ltr \label{basic estimate of K with space-time weight, L1 type} \\
     & \leq  C_\varepsilon  |k|^\f12 \|r^{\varepsilon-\f12} \log(\hat{C}r^2+a) w\|_\Llr \nonumber.
\end{align}}
\end{lemma}

\begin{proof}
We first handle the estimates of  the second term in \eqref{basic estimate of K} and \eqref{basic estimate of K with space-time weight}. Recall from \eqref{expression of phi_k} that
$$
\cK[w] = \int_1^\oo K_k(r,s)  w(s) ds,\quad
\cK'[w] = \int_1^\oo \partial_r K_k(r,s)   w(s)ds.
$$
Then for the operator $\cQ_{a,b}[f]$ given by \eqref{definition of Tab}, we get,  by using \eqref{estimate of cK}, that
\begin{equation*}
    \begin{split}
\Big|\frac{\cK[w]}{r^{2-\varepsilon}}\Big| &\leq \int_1^\oo r^{\varepsilon-2} s^{-\varepsilon} |K_k(r,s)| s^\varepsilon |w|(s)ds\\
&\lesssim |k|^{-1} \int_1^\oo (rs)^{-\f12} \min\Big(\big(\frac{r}{s}\big)^{|k|+\varepsilon-1}, \big(\frac{s}{r}\big)^{|k|+1-\varepsilon}\Big) s^\varepsilon |w|(s)ds \\
&\lesssim |k|^{-1} \cQ_{|k|+\varepsilon-1,|k|+1-\varepsilon} [s^\varepsilon w],
    \end{split}
\end{equation*}
which together with \eqref{ineq estimate of Tab, L2} ensures that
\begin{equation}\label{B.9}
    \Big\|\frac{\cK[w]}{r^{2-\varepsilon}}\Big\|_\Ltr \leq \frac{C}{|k|} \Big(\frac{1}{|k|+\varepsilon-1}+ \frac{1}{|k|-\varepsilon+1}\Big) \|r^\varepsilon w\|_\Ltr  \leq \frac{C_\varepsilon}{|k|^2} \|r^\varepsilon w\|_\Ltr.
\end{equation}

While for $f(r)= \log(\hc r^2 +a)r^{\frac{\varepsilon-2}{2}}$, it is straightforward to verify that for $\hc \geq e^{\frac{4}{2-\varepsilon}}$,
\begin{align*}
    f'(r)&= \frac{2\hc r}{\hc r^2 +a} r^{\frac{\varepsilon-2}{2}} - \frac{2-\varepsilon}{2}r^{\frac{\varepsilon-4}{2}}\log(\hc r^2+a)\leq r^{\frac{\varepsilon-4}{2}} \Big(2- \frac{2-\varepsilon}{2} \log\hc\Big)\leq 0,
\end{align*}
then by separately considering the cases: \( r \geq s \) and \( r \leq s \), to  conclude  that for any \( a \geq 0 \) and \( \hat{C} \geq e^{\frac{4}{2 - \varepsilon}} \), there holds
\begin{equation}\label{B.10}
\frac{\log(\hc r^2 +a)}{\log(\hc s^2 +a)}
\leq \max\Big(1,\big(\frac{r}{s}\big)^{\frac{2-\varepsilon}{2}}\Big).
\end{equation}
By virtue of (\ref{estimate of cK}) and (\ref{B.10}), one has
\begin{equation*}
    \begin{split}
&\Big|\frac{\log(\hat{C}r^2+a)}{r^{2-\varepsilon}}\cK[w]\Big|\\
 &\leq\frac{C}{|k|}  \int_1^\oo r^{\varepsilon-2} s^{-\varepsilon}(rs)^{\f12}  \frac{\log(\hc r^2 +a)}{\log(\hc s^2 +a)}     \min\big(\frac{r}{s},\frac{s}{r}\big)^{|k|}    s^\varepsilon \log(\hat{C}s^2+a)|w|(s)ds\\
 &\leq \frac{C}{|k|}  \int_1^\oo (rs)^{-\f12} \big(\frac{r}{s}\big)^{\varepsilon-1}   \max\Big(1,\big(\frac{r}{s}\big)^{\frac{2-\varepsilon}{2}}\Big) \min\big(\frac{r}{s},\frac{s}{r}\big)^{|k|}    s^\varepsilon \log(\hat{C}s^2+a)|w|(s)ds\\
 &\leq \frac{C}{|k|} \int_1^\oo (rs)^{-\f12} \min\Big(\big(\frac{r}{s}\big)^{|k|-1+\varepsilon}, \big(\frac{s}{r}\big)^{|k|-\frac{\varepsilon}{2}}\Big)  s^\varepsilon \log(\hat{C}s^2+a)|w|(s)ds \\
& \leq\frac{C}{|k|} \cQ_{|k|-1+\varepsilon,|k|-\frac{\varepsilon}{2}} [r^\varepsilon \log(\hat{C}r^2+a)w],
    \end{split}
\end{equation*}
from which and (\ref{ineq estimate of Tab, L2}), we infer
\begin{equation}\label{B.11}
\begin{split}
    \Big\|\frac{\log(\hat{C}r^2+a)}{r^{2-\varepsilon}}\cK[w]\Big\|_{\Ltr} \leq& \frac{C}{|k|} \Big(\frac{1}{|k|-\frac{\varepsilon}{2}}+ \frac{1}{|k|+\varepsilon-1}\Big) \|r^\varepsilon \log(\hat{C}r^2+a)w\|_\Ltr\\
    \leq& \frac{C_\varepsilon}{|k|^2} \|r^\varepsilon \log(\hat{C}r^2+a)w\|_\Ltr.
\end{split}
\end{equation}

So far, we have completed  the estimates of the second term in \eqref{basic estimate of K} and \eqref{basic estimate of K with space-time weight}. By virtue of the estimates \eqref{estimate of cK}, \eqref{B.10}, and \eqref{ineq estimate of Tab, L2}, the estimates for the first term in \eqref{basic estimate of K} and \eqref{basic estimate of K with space-time weight} follows along the same line. For the third term in \eqref{basic estimate of K}, we get, by applying Lemma \ref{Appendix A1-1} with \( A(r) = r^{\varepsilon - \frac{1}{2}} \) or \( A(r) = r^{\varepsilon - \frac{1}{2}} \log(\hat{C} r^2 + a) \), that
\begin{equation*}
    \begin{split}
&\|A(r) \cK'[w]\|_\Lor  \lesssim \|A(r)r^\f12\cK''[w]\|_\Ltr^\f12 \|A(r)r^{-\f12}\cK'[w]\|_\Ltr^\f12 +  \|A(r)r^{-\f12} \cK'[w]\|_\Ltr \\
& \lesssim  \Big(\|A(r)r^\f12 w\|_\Ltr + |k|^2 \|A(r)r^{-\f32} \cK[w]\|_\Ltr \Big)^\f12 \|A(r)r^{-\f12}\cK'[w]\|_\Ltr^\f12 +  \|A(r)r^{-\f12} \cK'[w]\|_\Ltr \\
&\lesssim_\varepsilon |k|^{-\f12} \|A(r)r^\f12 w\|_\Ltr.
    \end{split}
\end{equation*}
The estimates of fourth term in \eqref{basic estimate of K} and \eqref{basic estimate of K with space-time weight} follow along  the same lines.

Finally, for \eqref{basic estimate of K, L1 type} and \eqref{basic estimate of K with space-time weight, L1 type}, we only present the  estimate to
 typical term: \( |k|^2 \left\| \frac{\mathcal{K}[w]}{r^{2 - \varepsilon}} \right\|_{\Ltr }\). We first observe that
\begin{equation*}
    \begin{split}
   \Big|\frac{\cK[w]}{r^{2-\varepsilon}}\Big| &\leq \int_1^\oo r^{\varepsilon-2} s^{-\varepsilon+\f12} |K_k(r,s)| s^{\varepsilon-\f12} |w|(s)ds\\
      &\lesssim |k|^{-1} r^{\f12} \int_1^\oo r^{\varepsilon-\f52} s^{-\varepsilon+\f12} (rs)^{\f12} \min\big(\frac{r}{s},\frac{s}{r}\big)^{|k|} s^{\varepsilon-\f12} |w|(s)ds\\
& =|k|^{-1}r^{\f12} \int_1^\oo (rs)^{-\f12} \big(\frac{s}{r}\big)^{\f32-\varepsilon }\min\big(\frac{r}{s},\frac{s}{r}\big)^{|k|} s^{\varepsilon-\f12} |w|(s)  ds\\
&= |k|^{-1}r^{\f12} \int_1^\oo (rs)^{-\f12}  \min\Big(\big(\frac{r}{s}\big)^{|k|-\f32+\varepsilon},\big(\frac{s}{r}\big)^{|k|+\frac{3}{2}-\varepsilon}\Big) s^{\varepsilon-\f12}|w(s)|ds \\
&= |k|^{-1}r^{\f12} \cQ_{|k|-\f32+\varepsilon, |k|+\f32-\varepsilon}[r^{\varepsilon-\f12}w],
    \end{split}
\end{equation*}
which together with \eqref{ineq estimate of Tab, L1} ensures that
\begin{equation*}
    \Big\|\frac{\cK[w]}{r^{2-\varepsilon}}\Big\|_\Ltr \leq \frac{C}{|k|} \Big(\frac{1}{2(|k|+\varepsilon-\f32)+1}+ \frac{1}{2(|k|-\varepsilon+\f32)-1}\Big)^\f12 \|r^{\varepsilon-\f12} w\|_\Llr  \leq \frac{C_\varepsilon}{|k|^\f32} \|r^{\varepsilon-\f12} w\|_\Llr.
\end{equation*}
We thus complete the proof of Proposition \ref{basic properties of ck}.
\end{proof}

\begin{lemma}\label{Basic calculation for La_k}
{\sl For $\beta\in\R$, $ r\geq 1$, $\hc \geq e$ and  $\La_k=\log\bigl(\hat{C}r^2+\kappa_kt\bigr),$
 we have
\begin{subequations} \label{Sappeq9}
\begin{gather}\label{A.16}
\Big|\frac{\partial_t\La_k}{\La_k}\Big|= \f{\kappa_k }{(\hc r^2 + \kappa_kt)\log(\hc r^2 + \kappa_kt)}\leq \frac{\kappa_k}{\hc\log\hc}\frac{1}{r^2},
\\\label{A.17}
\Big|\frac{\partial_r \La_k}{\La_k}\Big| \leq \frac{2}{\log\hc}\frac{1}{r},\qquad
\Big|\frac{\partial_r^2 \La_k}{\La_k}\Big|\leq \frac{4}{\log\hc} \frac{1}{r^2},\\
    \label{A.18}
    \frac{\partial_r(r^{2\beta}\La_k^2)}{r^{2\beta-1}\La_k^2}\leq2|\beta| +\frac{4}{\log\hc}, \quad  \frac{\partial_r^2(r^{2\beta}\La_k^2)}{r^{2\beta-2}\La_k^2} \leq2\beta(2\beta-1) + \frac{1}{\log\hc}(16\beta+8) + \frac{8}{\log^2\hc}.
\end{gather}
\end{subequations}}
\end{lemma}
\begin{proof} We first observe that
\begin{equation*}
    \begin{split}
&\partial_t \La_k = \frac{\kappa_k}{\hc r^2 + \kappa_kt},\quad   \partial_r \La_k = \frac{2\hc r}{\hc r^2 + \kappa_kt}, \quad \partial_r^2 \La_k =\frac{2\hc \kappa_k t-2\hc^2 r^2}{(\hc r^2 + \kappa_kt)^2},
    \end{split}
\end{equation*}
which leads to \eqref{A.16} and \eqref{A.17}.

Whereas observing that
\begin{equation*}
    \begin{split}
\partial_r(r^{2\beta}\La_k^2) = 2\beta r^{2\beta-1}\La_k^2 + 2r^{2\beta}\La_k\pa_r\La_k,
    \end{split}
\end{equation*}
and
\begin{align*}
\partial_r^2(r^{2\beta}\La_k^2)&=\pa_r^2(r^{2\beta}) \La_k^2 + 2\pa_rr^{2\beta} \pa_r\La_k^2 +r^{2\beta} \pa_r^2\La_k^2\\
&=2\beta(2\beta-1) r^{2\beta-2}\La_k^2 +8\beta r^{2\beta-1}\La_k \pa_r\La_k+2r^{2\beta}(\La_k \pa_r^2 \La_k +|\pa_r\La_k|^2 ),
\end{align*}
we deduce \eqref{A.18} from \eqref{A.17}.
\end{proof}

\section*{Acknowledgement}
T. Li is partially supported by  National Natural Science Foundation of China under Grant 12421001.  P. Zhang is partially  supported by National Key R$\&$D Program of China under grant 2021YFA1000800 and by National Natural Science Foundation of China under Grant 12421001, 12494542 and 12288201.

\section*{Declarations}

\subsection*{Conflict of interest} The authors declare that there are no conflicts of interest.

\subsection*{Data availability}
This article has no associated data.

\end{document}